%% file: main.tex
\documentclass[
12pt, 
english, 
singlespacing, 
headsepline, 
]{MastersDoctoralThesis} 

\usepackage[utf8]{inputenc} 
\usepackage[T1]{fontenc} 

\usepackage{mathpazo} 

\usepackage{amsmath}
\usepackage{amsthm}
\usepackage{latexsym}
\usepackage{graphicx}
\usepackage{subfig}
\usepackage{wrapfig}
\usepackage{listings}
\usepackage{color} 
\definecolor{mygreen}{RGB}{28,172,0} 
\definecolor{mylilas}{RGB}{170,55,241}

\newtheorem{theorem}{Theorem}[section]
\newtheorem{lemma}[theorem]{Lemma}
\newtheorem{oss}[theorem]{Remark}
\newtheorem{prop}[theorem]{Proposition}
\newtheorem{definizione}[theorem]{Definition}
\newtheorem{cor}[theorem]{Corollary}

\usepackage[backend=bibtex,style=authoryear,natbib=true]{biblatex} 

\usepackage[autostyle=true]{csquotes} 

\thesistitle{Stability of space-time isogeometric methods for wave propagation problems} 
\supervisor{Giancarlo \textsc{Sangalli}, Andrea \textsc{Moiola}}{
\examiner{} 
\degree{} 
\author{Sara \textsc{Fraschini}} 
\addresses{} 

\subject{Mathematics} 
\university{\href{https://web.unipv.it/}{Università degli Studi di Pavia}} 
\department{\href{http://department.university.com}{Dipartimento di Matematica ``Felice Casorati''}} 
\group{} 
\faculty{} 

\AtBeginDocument{
\hypersetup{pdftitle=\ttitle} 
\hypersetup{pdfauthor=\authorname} 
\hypersetup{pdfkeywords=\keywordnames} 
}

\begin{document}

\frontmatter 

\pagestyle{plain} 


\begin{titlepage}
\begin{center}

\vspace*{.06\textheight}
{\scshape\LARGE \univname \par}\vspace{1.5cm} 
\centering
\includegraphics[width=0.3\textwidth]{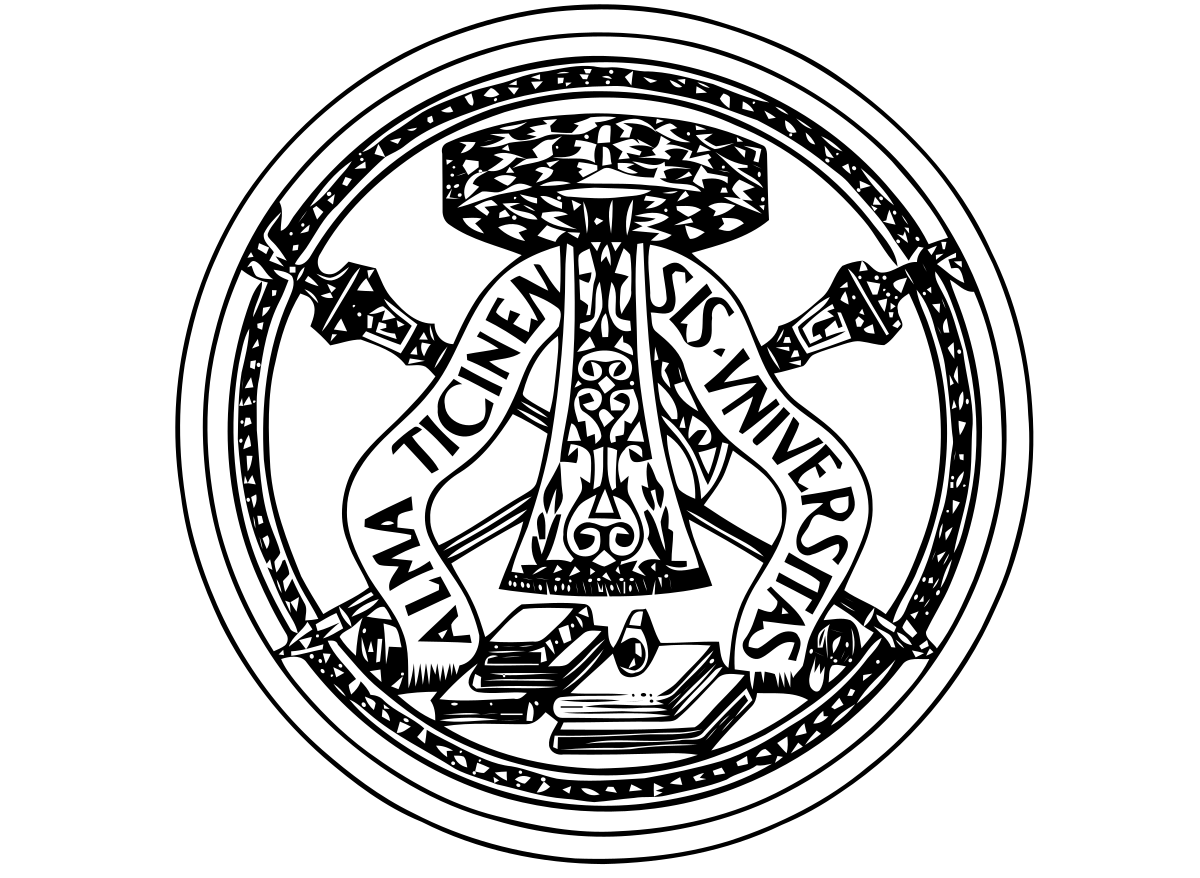}\\
\vspace{1cm}
\textsc{\Large Laurea Magistrale in Matematica}\\[0.5cm] 
\textsc{\small (Master's Thesis in Mathematics)}\\[0.5cm] 
\HRule \\[0.4cm] 
{\LARGE \bfseries \ttitle \par}\vspace{0.4cm} 
\HRule \\[1.5cm] 
 
\begin{minipage}[t]{0.4\textwidth}
\begin{flushleft} \large
\emph{Author:}\\
{\authorname} 
\end{flushleft}
\end{minipage}
\begin{minipage}[t]{0.4\textwidth}
\begin{flushright} \large
\emph{Supervisors:} \\
 \href{https://euler.unipv.it/moiola/} {Andrea \textsc{Moiola}} 
 \href{http://euler.unipv.it/sangalli/}{Giancarlo \textsc{Sangalli}}
\end{flushright}
\end{minipage}\\[3cm]
 
\vfill

\large
\deptname \\[2cm] 
 
\vfill

{\large July 20, 2021}\\[4cm] 

\vfill
\end{center}
\end{titlepage}

\cleardoublepage







\dedicatory{}
\vspace{-28cm}
\begin{flushright}
	\textit{Alla mia famiglia}
\end{flushright}


\begin{abstract}
\addchaptertocentry{\abstractname} 
The goal of this thesis is to investigate the first steps towards an unconditionally stable space-time isogeometric method (IGA method) with maximal regularity, using a tensor-product approach, for the homogeneous Dirichlet problem for the second-order wave equation. 

The unconditional stability of space-time discretization for wave propagation problems is a topic of significant and recent interest, by virtue of the advantages of space-time methods when compared with more standard space discretization plus time-stepping. These methods gained an increasing interest only recently, due to the improvement in computer technology, which is now capable of handling algorithms with significantly high computational costs. However, it is important to find efficient solvers for such problems: this is another interesting challenge together with the establishment of unconditional stability.

In the case of classic Galerkin continuous finite element methods (FEM), several stabilizations have been proposed, most of them relying only on \textit{heuristics}, without actually proving their effectiveness. Inspired by one of these works, we address the stability problem by studying the IGA method for an ordinary differential equation (ODE) closely related to the wave equation.
 
We study the conditioned stability of the IGA discretization for the following problem 
\begin{equation*}
\partial_{tt}u(t)+ \mu u(t)=f(t), \quad \text{for} \ t \in (0,T), \quad u(0)=\partial_{t}u(t)_{|t=0}=0,
\end{equation*}
where $T>0$ is the final instant of time and $\mu >0$ is the square of the wave-number. Eventually, we propose a stabilization, whose good behaviour is supported by numerical tests, and a possible extension to the wave-equation is suggested.

\end{abstract}



\tableofcontents 

\mainmatter 

\pagestyle{thesis} 


\include{Chapters/Chapter1}
\include{Chapters/Chapter2} 
\include{Chapters/Chapter3}
\include{Chapters/Chapter4} 
\include{Chapters/Chapter5}


\appendix 




\printbibliography[heading=bibintoc]

\end{document}

%% file: Chapters/Chapter1.tex

\chapter{Introduction} 

\label{Chapter1} 


\newcommand{\keyword}[1]{\textbf{#1}}
\newcommand{\tabhead}[1]{\textbf{#1}}
\newcommand{\code}[1]{\texttt{#1}}
\newcommand{\file}[1]{\texttt{\bfseries#1}}
\newcommand{\option}[1]{\texttt{\itshape#1}}


The most widely used approaches for the numerical solution of time-de\-pen\-dent linear partial differential equations are based on semi-discretization in space and time, using a time-stepping approach: at each fixed discrete time instant, a corresponding discretization in space is considered. The proposed methodology is applied to parabolic differential equations and to hyperbolic problems in Chapter 6 and Chapter 8 of (\cite{quarteroni2009modellistica}). A possible alternative is the simultaneous discretization in space and time, i.e., a space-time discretization. Space-time methods for time-dependent linear partial differential equations present some advantages when compared with more standard space discretization plus time-stepping. For example, the approximate solutions are available at all times in the interval of interest and it could also be possible to use techniques that are already been used for elliptical problems, such as multigrid methods (\cite{multigrid}) and domain decomposition (Chapter 14 of \cite{quarteroni2009modellistica}), which may allow parallelisation in time (\cite{timeparallel}). A possible drawback is that a global linear system must be solved at once. Therefore, fast solvers and preconditioning become essential. For instance, in the case of space-time isogeometric discretization for parabolic problems see (\cite{parabsang}).

Let us now consider the homogeneous Dirichlet problem for the second-order wave equation:
\begin{equation}\label{eqonde}
\begin{cases}
\partial_{tt}u(x,t)-\Delta_xu(x,t)=g(x,t) \quad (x,t) \in \Omega \times (0,T)\\
u(x,t)=0 \quad (x,t) \in \partial{\Omega} \times [0,T]\\
u(x,0)=\partial_tu(x,t)_{|t=0}=0 \quad x \in \Omega,
\end{cases}
\end{equation}
where $\Omega \subset \mathbb{R}^d$, with $d=1,2,3$, is an open bounded Lipschitz domain and, for a real value $T>0$, $(0,T)$ is a time interval. In (\cite{Coercive}) the authors introduce a space-time variational formulation of \eqref{eqonde}, where integration by parts is also applied with respect to the time variable, and the classic anisotropic Sobolev spaces with homogeneous initial and boundary conditions are employed. Stability of a conforming tensor-product space-time discretization of this variational formulation with piecewise polynomial, continuous solution and test functions, requires a Courant – Friedrichs – Lewy (CFL) condition, i.e., 
\begin{equation}\label{CFL}
h_t \leq C h_x,
\end{equation}
with a constant $C > 0$, depending on the constant of a spatial inverse inequality, where $h_t$ and $h_x$ are the (uniform) mesh-sizes in time and space, e.g, see (\cite{Steinbach2019, Zank2020}, for the piecewise linear case). The CFL condition is necessary for convergence while solving certain partial differential equations (usually hyperbolic PDEs) numerically. It arises in the numerical analysis of explicit time integration schemes. According to that, the time step must be less than a certain value, otherwise the simulation produces incorrect results. The condition is named after Richard Courant, Kurt Friedrichs, and Hans Lewy who described it in (\cite{CFL}).

Several approaches have been proposed in order to overcome restriction \eqref{CFL}. In (\cite{Steinbach2019}) the authors, following the work of (\cite{zlotnik}), introduce a perturbation of the tensor-product space-time piecewise linear discretization of \eqref{eqonde}. As a result, they can prove unconditional stability and optimal convergence rates in space-time norms. In particular, they start by considering the ordinary differential equation 
\begin{equation}\label{ode intro}
\partial_{tt}u(t)+ \mu u(t)=f(t), \quad \text{for} \ t \in (0,T), \quad u(0)=\partial_{t}u(t)_{|t=0}=0,
\end{equation}
where $\mu >0$, and its piecewise linear finite element discretization, since its stability is linked to the stability of the space-time standard FEM discretization of \eqref{eqonde}. They perturb the conforming discrete bilinear form, by considering the $L^2$ orthogonal projection on the piecewise constant finite element space. As a consequence, the new discrete system is unconditionally stable and convergence results hold without restrictions on the mesh-size. Finally, they extend this stabilization to the discretization for the wave-equation \eqref{eqonde}. In (\cite{higherorderzank}), M. Zank generalises this stabilization idea to an arbitrary polynomial degree with global continuity. In particular, he provides numerical examples for a one-dimensional spatial domain,
where the unconditional stability and optimal convergence rates in space-time norms are shown. On the other hand, theoretical considerations showing that such stabilization works are left to future papers. 

In (\cite{löscher2021numerical}) the authors consider a suitable linear transformation that defines an isomorphism between the anisotropic solution and test spaces. In this way they are able to define a Galerkin-Bubnov formulation that is unconditionally stable without further perturbations. In particular, the operator they use is the modified Hilbert transformation introduced in (\cite{Coercive, SteinbachZanknote, Zankexact}). However, in (\cite{löscher2021numerical}), they only give numerical
examples for a one- and a two-dimensional spatial domain, where the unconditional stability and optimal convergence rates in space-time norms are illustrated, and theoretical results are left to a future work. 

As it is proven in (\cite{Zank2020}), although the variational formulation to find the weak solution of \eqref{eqonde} in a suitable anisotropic Sobolev subspace $W$ of $H^1(\Omega \times (0,T))$ is well–defined for the right-hand-side $g$ being in the dual of the anisotropic Sobolev test space $V$, it is not possible to establish unique solvability. Indeed, the \textit{solution-to-data} operator between $W$ and $V'$ is not bijective. This is due to the fact that a stability condition, with respect to the dual norm of the right-hand-side, is not satisfied, see Theorem 4.2.24 of (\cite{Zank2020}). As a consequence, by the \textit{bounded inverse Theorem} (or \textit{inverse mapping Theorem}), the \textit{solution-to-data} linear map cannot be bijective. To ensure existence and uniqueness of a weak solution, we need to assume that $g \in L^2(\Omega \times (0,T))$. This is a
standard assumption to ensure sufficient regularity for the weak solution, and therefore, to obtain linear convergence for piecewise linear finite element approximations, but, as observed before, stability of common finite element discretizations require some CFL condition. In (\cite{steinbach2021generalized}) the authors introduce a new variational setting by enlarging the solution space. In this new framework they can prove that the \textit{solution-to-data} linear map is an isomorphism. Based on these results, they aim to derive a space–time finite element method for the numerical solution of the wave equation that is unconditionally stable. \\ \bigskip

The goal of this thesis is to investigate the first steps towards an unconditionally stable space-time isogeometric method with maximal regularity, using a tensor-product approach, for the wave problem \eqref{eqonde}. In particular, following (\cite{Steinbach2019}), the starting point is the analysis and the stabilization of the conforming discretization of \eqref{ode intro}. 

The choice of isogeometric methods can be advantageous due to the high degree of approximation of B-spline technology (\cite{HUGHES20084104, n-width}) and due to the exact representation of the geometry with non uniform rational B-splines (NURBS), which simplifies mesh refinement, as further communications with CAD are not necessary (\cite{isogeoCAD}). In particular, the choice of isogeometric methods with maximal regularity can be advantageous in the case of wave propagation problems to tackle the so-called \textit{pollution-effect}, which occurs in high-frequency wave problems (\cite{Babuska2000}). Indeed, a typical solution is to raise the order of the method: for the same number of degrees of freedom, methods that use piecewise polynomials of higher degree and regularity should perform better (\cite{HUGHES20084104}). Therefore, the choice of the IGA method with maximal regularity seems to be particularly suitable for this type of problems.

\subsubsection{Outline}
The rest of this thesis is organised as follows: in Chapter \ref{Chapter2} Sobolev spaces, spline spaces and variational methods are fixed and their most important properties are repeated. In Chapter \ref{Chapter3} the quadratic isogeometric method with maximal regularity for the ODE \eqref{ode intro} is investigated. In Chapter \ref{Chapter4} some numerical results are shown. In Chapter \ref{Chapter5} a short summary of the thesis and some suggestions for future work are given.

%% file: Chapters/Chapter2.tex

\chapter{Preliminaries} 

\label{Chapter2} 




In this Chapter we present notations for spaces with their properties and a general variational setting, recalling the Banach-Ne$\breve{c}$as-Babu$\breve{s}$ka Theorem and the Lax-Milgram Lemma.

All the normed spaces that we will consider in the thesis will be real vector spaces. Hence, we are not going to specify this, except for the cases where it is preferable to underline.

In the whole work, for a real value $T>0$, $(0,T)$ is a time interval. 

\section{Sobolev spaces in $(0,T)$}\label{sec sobolev}
In this section we recall some useful Sobolev spaces that we will consider from now on.

With the usual notations, for $p \in \mathbb{N}$, the Hilbert space $H^{p}(0,T)$ is the Sobolev space of (classes of) real-valued functions endowed with the inner product $\langle \cdot, \cdot \rangle_{H^p(0,T)}$ and the induced norm $||\cdot||_{H^p(0,T)}$, i.e.
\begin{gather*}
	H^p(0,T):=\{u \in L^2(0,T): \ \partial^m_t u \in L^2(0,T) \ \forall \ 0\leq m \leq p \},\\
	\langle u,v \rangle_{H^p(0,T)}:=\int_{0}^{T} u(t)v(t) \ dt + \sum_{m=1}^{p} \int_0^T \partial^m_t u(t)\partial^m_t v(t) \ dt, \quad u,v \in H^p(0,T).
\end{gather*}
Also, we will consider the seminorm
\begin{equation*}
	|u|_{H^p(0,T)}:=\sqrt{\int_0^T (\partial^p_t u(t))^2 \ dt}, \quad u \in H^p(0,T).
\end{equation*}

Since $H^1(0,T) \subset AC([0,T])$, the space of absolutely continuous functions, see (\cite{Brezis2011}), we can define the following subspaces of $H^1(0,T)$ 
\begin{gather*}
	H^1_{0,*}(0,T)=\{u \in H^1(0,T): \ u(0)=0\},\\
	H^1_{*,0}(0,T)=\{u \in H^1(0,T): \ u(T)=0\},
\end{gather*}
endowed with the Hilbertian norm
\begin{gather*}
	\|u\|_{H^1_{0,*}(0,T)}:=|u|_{H^1(0,T)}=\|\partial_t u\|_{L^2(0,T)},\\
	\|v\|_{H^1_{*,0}(0,T)}:=|v|_{H^1(0,T)}=\|\partial_t v\|_{L^2(0,T)},
\end{gather*}
for $u \in H^1_{0,*}(0,T)$ and for $v \in H^1_{*,0}(0,T)$.

Clearly, for $1 \leq p \leq q$, $H^q(0,T) \subset H^p(0,T)$ with a continuous embedding, thus we can also define the following spaces
\begin{gather*}
H^p_{0,*}(0,T)=\{u \in H^p(0,T): \ u(0)=0\},\\
H^p_{*,0}(0,T)=\{u \in H^p(0,T): \ u(T)=0\}.
\end{gather*}

Let $\delta \in \mathbb{R}$, with $\delta >T$. Recalling that the space $C^{\infty}_{c}(0,\delta)$ is dense in $H^1_0(0,\delta)$, the set
\begin{equation*}
	C^{\infty}_c(0,T]:=\{\phi_{|(0,T]}: \ \phi \in C^{\infty}_{c}(0,\delta)\}
\end{equation*}
is dense in $H^1_{0,*}(0,T)$. Let $\eta \in \mathbb{R}$, with $\eta<0$. Analogously,  the set 
\begin{equation*}
	C^{\infty}_c[0,T):=\{\phi_{|[0,T)}: \ \phi \in C^{\infty}_{c}(-\eta,T)\}
\end{equation*}
is dense in $H^1_{*,0}(0,T)$.

In $H^1_{0,*}(0,T)$ and $H^1_{*,0}(0,T)$ there hold inequalities of Poincaré type with sharp constants, see Lemma 3.4.5 of (\cite{Zank2020}), i.e., for all $u \in H^1_{0,*}(0,T)$ and $v \in H^1_{*,0}(0,T)$, there hold
\begin{equation}\label{Poinc}
	\|u\|_{L^2(0,T)} \leq \frac{2T}{\pi} \|\partial_{t}u\|_{L^2(0,T)} \quad \text{and} \quad \|v\|_{L^2(0,T)} \leq \frac{2T}{\pi} \|\partial_{t}v\|_{L^2(0,T)},
\end{equation}
and the constants in these inequalities are sharp. Thus, $\|\cdot\|_{H^1(0,T)}$ and \\ $|\cdot|_{H^1(0,T)}$ are equivalent Hilbertian norm in $H^1_{0,*}(0,T)$ and in $H^1_{*,0}(0,T)$.

\begin{oss}
	The estimates \eqref{Poinc} are more accurate (by a factor of about $0.1$) than those proposed in (\cite{Steinbach2019}). Therefore, the estimate proven by using \eqref{Poinc} are slightly different from those in (\cite{Steinbach2019}).  
\end{oss}

The dual spaces $[H^1_{0,*}(0,T)]'$ and $[H^1_{*,0}(0,T)]'$ are Hilbert spaces, and they can be characterised as completions of $L^2(0,T)$ with respect to their dual Hilbertian norm, see (\cite{wloka1987partial}).\\ \bigskip

\section{Spline spaces over a real  interval}\label{sec splines} 
In this Chapter we will give a brief overview of B-splines and of the most common spline spaces, all seen in the simple case of one-dimensional domains, which are the subjects of interest of Chapter \ref{ch 3}. Our main references are (\cite{de1978practical, schumaker2007spline, Hllig2013ApproximationAM}) and Chapter 1 of (\cite{chapter1}). Note that there is a lot of literature on splines, given to the fact that they have application in several branches of the sciences.

First of all, let us remark that splines, in one- or plus- dimensional cases, in the broad sense of the term, are functions consisting of pieces of smooth functions glued together in a certain smooth way. The most popular species is the one where the pieces
are algebraic polynomials and inter-smoothness is imposed by means of equality of derivatives up to a given order. This species is the one of interest for isogeometric analysis, and is therefore the one we will consider in this Chapter in the one-dimensional case.

\subsubsection{Univariate B-splines}
The concept of knot vector is essential to define univariate B-splines.
\begin{definizione}
	A knot vector (or knot sequence) $\Xi$ is a nondecreasing sequence of real numbers,
	\begin{equation*}
	\Xi := \{\xi_i\}_{i=1}^m=\{\xi_1\leq \ldots \leq \xi_m\}, \quad m \in \mathbb{N}_{>0}.
	\end{equation*} 
The elements $\xi_i$ are called knots and, if $\xi_1 \neq \xi_m$, the nondecreasing distinct knots are called break points.
\end{definizione}

Provided that $m \geq p+2$, we can define univariate B-splines of degree $p$ over the knot vector $\Xi$.

\begin{definizione}
	Suppose for a nonnegative integer $p$ and some integer $j$ that $\xi_j \leq \ldots \leq \xi_{j+p+1}$ are $p+2$ real numbers taken from a knot vector $\Xi$. The $j$-th B-spline of degree $p$ is defined recursively by
	\begin{equation}\label{cox de boor}
	b_{j,p,\Xi}(x):=\begin{cases} \frac{x-\xi_j}{\xi_{j+p}-\xi_j} b_{j,p-1,\Xi}(x)+\frac{\xi_{j+p+1}-x}{\xi_{j+p+1}-\xi_{j+1}}b_{j+1,p-1,\xi}(x) \quad x \in [\xi_j,\xi_{j+p+1})\\
	0 \quad \text{otherwise}
	\end{cases}
	\end{equation}
	starting with
	\begin{equation*}
	b_{i,0,\Xi}(x):=\begin{cases}
	1 \quad x \in [\xi_i,\xi_{i+1}),\\
	0 \quad \text{otherwise},
	\end{cases}
	\end{equation*}
	where we adopt the convention $0/0=0$.
\end{definizione}
Formula \eqref{cox de boor} is known as \textit{Cox-de Boor recursion formula}. \\ \bigskip 

Let us underline the following properties. For a proof and for more details on B-splines properties we refer to Chapter 1 of (\cite{chapter1})
\begin{itemize}
	\item For degree $0$ the B-spline $B_{i,0,\Xi}$ is simply the characteristic function of the half open interval  $[\xi_i;\xi_{i+1})$, if $\xi_{i} \neq \xi_{i+1}$.
	\item A B-spline is right-continuous.
	\item A B-spline is locally supported on the interval given by the extreme
	knots used in its definition, i.e., supp($b_{j,p,\Xi}$) $\subseteq [\xi_j,\xi_{j+p+1}]$ is a compact subinterval of $[\xi_j,\xi_{j+p+1}]$.
	\item A B-spline is nonnegative everywhere, and positive in the interior part of its
	support (if $\xi_{j+p+1} \neq \xi_j$), i.e.,
	\begin{gather*}
	b_{j,p,\Xi}(x) \geq 0 \quad \forall x \in \mathbb{R},\\
	b_{j,p,\Xi}(x) > 0 \quad \forall x \in (\xi_j,\xi_{j+p+1}).
	\end{gather*}
	\item A B-spline $b_{j,p,\Xi}$ has a piecewise polynomial structure of degree less then or equal to $p$ over the subintervals defined by the knots.
	\item We say that a knot has multiplicity $m$ if it occurs exactly $m$ times in the knot sequence. If $\xi$ is a knot of $b_{j,p,\Xi}$ of multiplicity $m \leq p+1$, then 
	\begin{equation*}
	b_{j,p,\Xi} \in C^{p-m}(\xi),
	\end{equation*}
	i.e., its derivatives of order $0,\ldots,p-m$ are continuous at $\xi$, and $C^{-1}(\xi)$ denotes discontinuity at $\xi$. In particular, the maximal regularity at knots is $C^{p-1}$.
\end{itemize}

\begin{figure}[h]	
	\centering
		\includegraphics[scale=0.65]{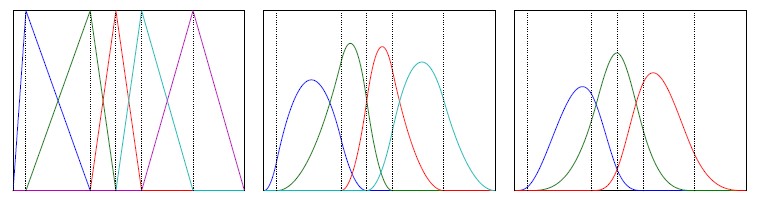}
		\caption{Several examples of B-splines of degree $p = 1,2,3$ respectively. The knot positions are visualized by vertical dotted lines. The same knot vector is chosen for the different degrees, with only simple knots, i.e., with knots with multiplicity equal to 1.}
\end{figure}

\subsubsection{Spline spaces}

Let $n>p \geq 0$ be two integers and let 
\begin{equation}\label{knot vector basis}
\Xi := \{\xi_i\}_{i=1}^{n+p+1}=\{\xi_1\leq \ldots \leq \xi_{n+p+1}\},
\end{equation}
be a given knot vector. This knot sequence allows us to define a set of $n$ B-splines of degree $p$, i.e.,
\begin{equation*}
\{ b_{1,p,\Xi}, \ldots, b_{n,p,\Xi}\}.
\end{equation*}

We are now interested in considering a family of B-splines which is a basis of the space of piecewise polynomials of degree $p$ on the intervals defined by the break points of the knot vector $\Xi$. Since the B-splines we are going to consider are restricted to the interval $[\xi_{p+1},\xi_{n+1}]$, we define the B-splines to be left continuous at the right endpoint, in order to avoid an asymmetry. Namely, we require that its value at $\xi_{n+1}$ is obtained by taking limits from the left:
\begin{equation*}
b_{j,p,\Xi}(\xi_{n+1}):=\lim_{x \rightarrow \xi_{n+1}^{-}} b_{j,p,\Xi}(x), \quad j=1,\ldots,n.
\end{equation*}

\begin{definizione}
The knot vector \eqref{knot vector basis} is said to be open on an interval $[a,b]$ if it satisfies
\begin{equation*}
a=\xi_1=\ldots=\xi_{p+1} < \xi_{p+2} \leq \ldots \leq \xi_n < \xi_{n+1} = \ldots = \xi_{n+p+1}=b.
\end{equation*}
\end{definizione}
	From now on we will simply say \textit{open knot vector} without specifying the defining range $[a,b]$, except in cases where we are interested in underling the defining domain.
\begin{oss}
	The B-splines $\{ b_{1,p,\Xi}, \ldots, b_{n,p,\Xi}\}$ defined by an open knot vector $\Xi$ satisfy the following properties:
	\begin{itemize}
		\item \textbf{Partition of unity}: 
		\begin{equation*}
		\sum_{i=1}^n b_{i,p,\Xi}(x)=1, \quad \forall x \in [\xi_{p+1},\xi_{n+1}].
		\end{equation*}
		\item \textbf{Interpolation property}: $b_{1,p,\Xi}$ and $b_{n,p,\Xi}$ are interpolatory at $\xi_{p+1}$ and $\xi_{n+1}$ respectively, i.e.,
		\begin{gather*}
		b_{1,p,\Xi}(\xi_{p+1})=1,\\
		b_{n,p,\Xi}(\xi_{n+1})=1.
		\end{gather*} 
		\item \textbf{Linear independence}: they are linearly independent on $[\xi_{p+1},\xi_{n+1}]$.
	\end{itemize}		 
\end{oss}

\begin{figure}[h]	
	\centering
	\includegraphics[scale=0.65]{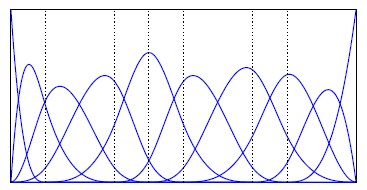}
	\caption{The B-spline basis of degree $p = 3$ on an open knot sequence. The knot positions are visualized by vertical dotted lines. The interior knots are simple, i.e., their multiplicity is equal to 1.}
\end{figure}

Let $\Delta$ be a sequence of real numbers:
\begin{equation}\label{delta}
\Delta:=\{\eta_0 < \eta_1 < \ldots < \eta_{l+1}\}.
\end{equation}
Furthermore, let $\mathbf{r}:=(r_1,\ldots,r_l)$ be a vector of integers such that $-1 \leq r_i \leq p-1$ for $i = 1,\ldots,l$. The space $S_p^{\mathbf{r}}(\Delta)$ of piecewise
polynomials of degree p with smoothness $\mathbf{r}$ over the partition \eqref{delta} is defined by
\begin{equation}\label{pol tratti}
\begin{split}
S_p^{\mathbf{r}}(\Delta):= \{ s: [\eta_0, \eta_{l+1}] \rightarrow \mathbb{R} \ | \ s \in \mathbb{P}_p([\eta_i, \eta_{i+1})) \ i=0,\ldots,l-1&, \ \\s \in \mathbb{P}_p([\eta_l, \eta_{l+1}]), \ s \in C^{r_i}(\eta_i)  \ i =1,\ldots,l \}.
\end{split}
\end{equation}

\begin{theorem}[\cite{chapter1}]
	The piecewise polynomial space \eqref{pol tratti} is characterized in terms of B-splines by
	\begin{equation}\label{spline space}
	S_p^{\mathbf{r}}(\Delta)= \text{span}\{b_{i,p,\Xi}\}_{i=1}^n,
	\end{equation}
	where $\Xi:=\{\xi_1\leq \ldots \leq \xi_{n+p+1}\}$ is an open knot sequence with $n:=\text{dim}(S_p^{\mathbf{r}}(\Delta))$ such that 
	\begin{equation*}
	\xi_1=\ldots=\xi_{p+1}:=\eta_0, \quad \xi_{n+1}=\ldots=\xi_{n+p+1}:=\eta_{l+1},
	\end{equation*}
	and
	\begin{equation*}
	\xi_{p+2},\ldots,\xi_{n}:=\overset{p-r_1}{\overbrace{\eta_1,\ldots,\eta_1}}, \ldots, \overset{p-r_l}{\overbrace{\eta_l,\ldots,\eta_l}}.
	\end{equation*}
\end{theorem}

The space \eqref{spline space} is called \textit{spline space}. We will denote it by $S^p_\Xi([\xi_{p+1},\xi_{n+1}])$ in the case of $C^{p-1}$ global regularity, or with $S^p_h([\xi_{p+1},\xi_{n+1}])$, where $h$ is defined as $h:=\max{\{|\xi_{i+1}-\xi_i|\ : \ i=1,\ldots,n+p+1}\}$. Our choice will depend on which notation is most convenient for the case we consider.

\begin{definizione}
	The coefficients of the B-splines generating the spline space are called control points or degrees of freedom. The elements in \eqref{spline space} are also called B-spline curves.
\end{definizione}

Let us note that an affine transformation of a B-spline curve is obtained by applying the transformation to control points (\cite{isogeoCAD}). This is one of the many reasons why these curves are widely used in Computer-Aided Design (CAD).

\begin{oss}
In general, control points do not interpolate B-spline curves in knots. Instead, extremal control points interpolate B-spline curves at the extremal knots $\xi_{p+1}$ and $\xi_{n+1}$.
\end{oss}

\begin{oss}
	In this section on spline spaces we could have followed a more general treatment, instead of considering only open knot vectors. This was not of interest to us, since open knot vectors are the ones used in CAD and isogeometric analysis. See (\cite{chapter1}) if you are interested in a more general discussion.
\end{oss}

\section{Variational methods}

\subsection{Well-posedness of abstract problems}
In this Section we introduce an abstract variational problem and determine the conditions under which this problem is well-posed. The main references of this Section are Chapter 2 of (\cite{lect2}) and (\cite{saito2018notes}). 

Let $V$ and $W$ be two vector spaces endowed with norms $\|\cdot\|_V$ and $\|\cdot\|_W$. Let $a: W \times V \rightarrow \mathbb{R}$ be a given bounded bilinear form and let $\mathcal{F} \in V'$. Let us consider the following abstract variational problem
\begin{equation}\label{pbl var}
\begin{cases}
\text{Find} \ u \in W \ \text{such that}\\
a(u,v)=\langle \mathcal{F},v\rangle_{V', V} \quad \forall v \in V.
\end{cases}
\end{equation}
$W$ is called the \textit{solution space} and $V$ is called the \textit{test space}.

\begin{oss}
	A significant question is: what do we mean by \textbf{solution} of a partial differential equation (PDE)? One could ask for the PDE of order k to be solved pointwise by one (or more) functions of class $C^k$, or even by one (or more) of class $C^\infty$: in this case the solution(s) is (are) called classical or strong solution(s). However, in general, solutions to relevant problems are not so regular. Therefore, the concept of a weak solution is introduced for a given PDE and, if necessary, its regularity is studied. Typically, the variational formulation \eqref{pbl var} results from the weak formulation of PDEs. Consider, for example, the Poisson problem:
	\begin{equation*}
	\begin{cases}
	-\Delta u(x) = f(x) \quad x \in \Omega\\
	u(x) = 0 \quad x \in \partial{\Omega},
	\end{cases}
	\end{equation*}
	where $\Omega \subset \mathbb{R}^n$ is an open bounded Lipschitz domain. By integration by part, its weak formulation reads as follows
	\begin{equation}\label{poiss var}
	\begin{cases}
	\text{Find} \ u \in H^1_0(\Omega) \quad \text{such that}\\
	\langle \nabla u, \nabla v \rangle_{L^2(\Omega)} = \langle f, v \rangle_\Omega \quad \forall v \in H^1_0(\Omega),
	\end{cases}
	\end{equation}
	where $f \in H^{-1}(\Omega)$ is given. In particular, in \eqref{poiss var}, the bilinear form is defined as
	\begin{gather*}
	a: H^1_0(\Omega) \times H^1_0(\Omega) \longrightarrow \mathbb{R} \quad \text{s.t.}\\
	a(u,v)=\langle \nabla u, \nabla v \rangle_{L^2(\Omega)} \quad \forall (u,v) \in H^1_0(\Omega) \times H^1_0(\Omega).
	\end{gather*}
	
\end{oss}

\begin{definizione}[Hadamard]
	Problem \eqref{pbl var} is said to be well-posed if it
	admits one and only one solution and if the following a priori estimate holds:
	\begin{equation}\label{stab}
		\exists C>0: \ \forall \mathcal{F} \in V', \ \|u\|_{W} \leq C \|\mathcal{F}\|_{V'}.
	\end{equation}
\end{definizione}

\begin{oss}
	The notion of well-posedness of a problem captures many of the desirable characteristics for a solution of a PDE. In particular, condition \eqref{stab} is very important for problems arising from physical applications. Indeed, it is clearly preferable that the solution has a ``little'' change when the specific conditions of the problem have ``little'' changes.
\end{oss}
\begin{oss}
	A bounded, linear operator $\mathcal{A}: W \rightarrow V'$ is associated with the bounded bilinear form $a(\cdot,\cdot)$ by setting
	\begin{equation*}
		\langle \mathcal{A}u,v \rangle_{V', V}:=a(u,v), \quad u \in W, v \in V.
	\end{equation*}
	Therefore, problem \eqref{pbl var} amounts to seeking $u \in W$ such that $\mathcal{A}u=\mathcal{F}$ in $V'$.
	
	The two following statements are equivalent:
	\begin{itemize}
		\item The variational problem \eqref{pbl var} is well-posed.
		\item The bounded linear operator $\mathcal{A}: W \rightarrow V'$ associated with the continuous bilinear form $a(\cdot,\cdot)$ is an isomorphism.
	\end{itemize}

Also, every bounded, linear operator $\mathcal{A}: W \rightarrow V'$ defines a bounded bilinear form $a: W \times V \rightarrow \mathbb{R}$ by setting
\begin{equation*}
a(u,v):=\langle \mathcal{A}u,v \rangle_{V', V}, \quad u \in W, v \in V.
\end{equation*}
\end{oss}
\hspace{0.2cm} 
\\
\textbf{The Banach-Ne$\mathbf{\breve{c}}$as-Babu$\mathbf{\breve{s}}$ka Theorem}. This Theorem gives \textit{necessary and sufficient} conditions for the well-posedness of \eqref{pbl var}.
\begin{theorem}[Banach-Ne$\mathbf{\breve{c}}$as-Babu$\mathbf{\breve{s}}$ka, (\cite{saito2018notes})]\label{BNB}
	Let $(W,\|\cdot\|_W)$ be a real Banach space, let $(V,\|\cdot\|_V)$ be a real reflexive, Banach space. Then the following statements are equivalent:
	\begin{itemize}
		\item[1.] The variational problem \eqref{pbl var} is well-posed.
		\item[2.] For the bounded bilinear form $a: W \times V \rightarrow \mathbb{R}$ there hold:
		\begin{itemize}
			\item There exists a constant $\beta >0$ such that
			\begin{equation}\label{infsup}
			\inf_{u \in W} \sup_{v \in V} 
			\frac{a(u,v)}{\|u\|_W \|v\|_V} \geq \beta.
			\end{equation}
			\item For each $v \in V$,
			\begin{equation}\label{non null}
			(\forall u \in W, \ a(u,v)=0) \Longrightarrow (v=0)
			\end{equation}
		\end{itemize}
		\item[3.] There exist two constants $\beta, \gamma >0$ such that
		\begin{gather*}
			\inf_{u \in W} \sup_{v \in V} 
			\frac{a(u,v)}{\|u\|_W \|v\|_V} \geq \beta,\\
			\inf_{v \in V} \sup_{u \in W} 
			\frac{a(u,v)}{\|u\|_W \|v\|_V} \geq \gamma.
		\end{gather*}
	\end{itemize}
\end{theorem}

Condition \eqref{infsup} is usually called the \textit{inf-sup} condition or the \textit{Babu$\breve{s}$ka-Brezzi} condition.

\begin{oss}
	The following statements hold:
	\begin{itemize}
		\item Given $V$ and $W$ two real normed spaces, condition \eqref{infsup} is expressed equivalently as
		\begin{equation*}
			\exists \beta>0: \ \beta \|u\|_W \leq \sup_{v \in V} 
			\frac{a(u,v)}{\|v\|_V} \quad \forall u \in W.
		\end{equation*}
		\vspace{0.5cm} \\
		Under the assumptions of Theorem \ref{BNB}:
		\item The well-posedness of \eqref{pbl var} is satisfied with
		\begin{equation*}
			\|u\|_{W} \leq  \frac{1}{\beta} \|\mathcal{F}\|_{V'}.
		\end{equation*}
		In general, given $V$ and $W$ two normed spaces, if \eqref{pbl var} admits unique solution for every $\mathcal{F} \in V'$, then condition \eqref{stab} is equivalent to \eqref{infsup} with $\beta = \frac{1}{C}$.
		\item We can translate conditions \eqref{infsup}, \eqref{non null} into conditions for the linear operator $\mathcal{A}(\cdot)$, see (\cite{lect2}). \\ 
		\begin{gather*}
			\eqref{infsup} \Longleftrightarrow (\ker{(\mathcal{A})}=\{0\} \ \text{and} \ \text{Im}(\mathcal{A}) \ \text{is closed}) \Longleftrightarrow (\mathcal{A}^{*} \ \text{is surjective}),\\ 
			\eqref{non null} \Longleftrightarrow (\ker{(\mathcal{A}^*)}=\{0\}) \Longleftrightarrow (\mathcal{A}^* \ \text{is injective}),
		\end{gather*}
		where $\mathcal{A}^*$ is the adjoint operator of $\mathcal{A}$. \vspace{0.3cm} \\
		In general, given $X$ and $Y$ two real normed vector spaces and a linear operator $\mathcal{A}: X \rightarrow Y$, if $\mathcal{A}$ is invertible, then the adjoint operator of $\mathcal{A}$ is invertible and satisfies $(\mathcal{A}^*)^{-1}=(\mathcal{A}^{-1})^*$, see (\cite{adjoint}) for the case of Hilbert spaces (the generalisation of the results to the case of any normed space is straightforward). In particular, as a consequence of $(\mathcal{A}^*)^{-1}=(\mathcal{A}^{-1})^*$ and of the equality in norm between an operator and its adjoint, if $\mathcal{A}$ is isomorphism, then $\mathcal{A}^*$ is an isomorphism with the same continuity constants (also for the inverse maps). The BNB Theorem \ref{BNB} guarantees that, if $X=W$ is a real Banach space, and $Y=V'$, with $V$ that is a real reflexive Banach space, the reverse is also true, i.e., if $\mathcal{A}^*$ is an isomorphism, then $\mathcal{A}$ is an isomorphism.
		\item The well-posedness of the adjoint problem 
		\begin{equation*}
		\begin{cases}
		\text{Find} \ v \in V \ \text{such that}\\
		a(u,v)=\langle \mathcal{G},u\rangle_{W', W} \quad \forall u \in W,
		\end{cases}
		\end{equation*}
		where $\mathcal{G} \in W'$,
		is equivalent to $\mathcal{A}^*$ being an isomorphism.
	\end{itemize}
\end{oss}
\hspace{0.2cm} 
\\
\textbf{The Lax-Milgram Lemma}. Consider the case where the solution space and the test space are identical Hilbert spaces. Thus, the abstract variational problem has the following formulation:
\begin{equation}\label{pbl var lm}
\begin{cases}
\text{Find} \ u \in V \ \text{such that}\\
a(u,v)=\langle \mathcal{F},v\rangle_{V', V} \quad \forall v \in V
\end{cases}
\end{equation}
The Lax-Milgram Lemma gives \textit{sufficient} conditions under which problem \eqref{pbl var lm} is well-posed.
\begin{lemma}[Lax-Milgram]\label{lax-milgram}
	Let $V$ be a Hilbert space, let $a: V \times V \rightarrow \mathbb{R}$ be a bounded bilinear form. Assume that $a(\cdot,\cdot)$ is coercive, i.e, 
	\begin{equation}\label{coerc}
	\exists \alpha >0: \ \forall u \in V, \ a(u,u)\geq \alpha \|u\|_V^2.
	\end{equation}
	Let $\mathcal{F} \in V'$. Then problem \eqref{pbl var} is well-posed with the following a priori estimate
	\begin{equation*}
		\|u\|_V \leq \frac{1}{\alpha} \|\mathcal{F}\|_{V'}.
	\end{equation*}
\end{lemma}

The Lax-Milgram Lemma can be viewed as a Corollary of the Banach-Ne$\breve{c}$as-Babu$\breve{s}$ka Theorem, since the coercivity of $a(\cdot,\cdot)$ implies statement $2$ of Theorem \ref{BNB}.

\begin{lemma}\label{LM then BNB}
	The coercivity condition \eqref{coerc} implies conditions \eqref{infsup}, \eqref{non null} of the BNB Theorem. 
\end{lemma}
\begin{proof}
	Let $u \in V$. Condition \eqref{infsup} is immediately deduced from
	\begin{equation*}
		\alpha \|u\|_V \overset{\eqref{coerc}}\leq \frac{a(u,u)}{\|u\|_V} \leq \sup_{v \in V} \frac{a(u,v)}{\|v\|_V}.
	\end{equation*}
	Let now $v \in V$. As a consequence of coercivity there hold
	\begin{equation*}
		\sup_{u \in V} a(u,v) \geq a(v,v) \overset{\eqref{coerc}} \geq \alpha \|v\|_V^2.
	\end{equation*}
	Therefore, $\sup_{u \in V} a(u,v)=0$ implies $v=0$, i.e., \eqref{non null}.
\end{proof}

\begin{oss}
	The reversal of Lemma \ref{LM then BNB} is wrong, i.e., \eqref{coerc} is not equivalent to the well-posedness of \eqref{pbl var lm}. However, when the bilinear form $a(\cdot,\cdot)$ is symmetric and positive, coercivity is equivalent to well-posedness, see (\cite{lect2}).
\end{oss}

\subsection{Galerkin method}\label{galerkin method}
Let $W_h$ be a finite-dimensional subspace of $W$ and $V_h$ be a finite-dimensional subspace of $V$. The Galerkin method constructs an approximation of the solution $u$ of the abstract variational problem \eqref{pbl var} by solving the following problem
\begin{equation}\label{galerkin petrov}
\begin{cases}
\text{Find} \ u_h \in W_h \ \text{such that}\\
a(u_h,v_h)=\langle \mathcal{F}_{|V_h},v_h\rangle_{V_h', V_h} \quad \forall v_h \in V_h.
\end{cases}
\end{equation}
$W_h$ is called the \textit{solution space} or the \textit{trial space}, whereas $V_h$ is called the \textit{test space}. If the trial space and the test space are the same, the Galerkin method is the following problem:
\begin{equation}\label{galerkin bubnov}
\begin{cases}
\text{Find} \ u_h \in V_h \ \text{such that}\\
a(u_h,v_h)=\langle \mathcal{F}_{|V_h},v_h\rangle_{V_h', V_h} \quad \forall v_h \in V_h.
\end{cases}
\end{equation}
Typically, the former is called \textit{Galerkin-Petrov method}, whereas the latter is called \textit{Galerkin-Bubnov method}.

\begin{oss}
	We can also consider $W_h$ and $V_h$ as closed subspaces of $W$ and $V$ respectively, but, for the numerical approximation of $u$, we are interested in finite-dimensional subspaces of the trial and test spaces.
\end{oss}

We are interested in investigating the well-posedness of the approximate problems \eqref{galerkin petrov}, \eqref{galerkin bubnov}, in proving a stability condition, i.e., a uniform (w.r.t a numerical parameter $h$ indexing the discrete spaces) a priori estimate \eqref{stab} for the discrete problem, and convergence results (w.r.t $h$) of the discrete solution to the abstract one. Let us begin by considering the first and second questions. The following Proposition is straightforward.
\begin{prop}
	If the variational problem \eqref{pbl var lm} satisfies Lax-Milgram hypothesis, then the variational problem \eqref{galerkin bubnov} is well-posed. In particular, for all $\mathcal{F} \in V'$ the stability estimate $\|u_h\|_V \leq \frac{1}{\alpha} \|\mathcal{F}\|_{V'}$ holds.
\end{prop}
\begin{proof}
	We can apply Lax-Milgram Lemma \ref{lax-milgram} since the bilinear form $a(\cdot,\cdot)$ is coercive in $V_h$.
\end{proof}

In general, instead, there is no guarantee that conditions \eqref{infsup}, \eqref{non null} of the BNB Theorem are automatically transferred from the abstract problem to the approximate problem. Thanks to the BNB Theorem, the well-posedness of \eqref{galerkin petrov} is equivalent to the following discrete conditions:
\begin{itemize}
	\item There exists a constant $\beta_h >0$ such that
	\begin{equation}\label{infsup discreta}
	\inf_{u_h \in W_h} \sup_{v_h \in V_h} 
	\frac{a(u_h,v_h)}{\|u_h\|_{W_h} \|v_h\|_{V_h}} \geq \beta_h.
	\end{equation}
	\item For each $v_h \in V_h$,
	\begin{equation}\label{non null discreta}
	(\forall u_h \in W_h, \ a(u_h,v_h)=0) \Longrightarrow (v_h=0)
	\end{equation}
\end{itemize}
Condition \eqref{infsup discreta} is usually called \textit{discrete inf-sup} condition.
\begin{oss}
	In many cases, the subscript $h$ denotes the mesh-size of the discretization and, in general, we are interested in proving an a priori estimate \eqref{stab} for the discrete problem that is independent of $h$. Hence, the discrete inf-sup condition \eqref{infsup discreta} is not sufficient for the desired stability, if no information on the value of $\beta_h$ is available.
\end{oss}

\begin{oss}\label{remark discrete system}
	Evaluating the bilinear form $a(\cdot,\cdot)$ and the linear operator $\mathcal{F}(\cdot)$ on the basis functions of the spaces $W_h$ and $V_h$, we obtain a linear system that is equivalent to the general approximate problem \eqref{galerkin petrov}, see (\cite{lect2}). We use the notation $\mathbf{A}$ for the system matrix, which will be refer to as the \textit{stiffness matrix}. \\
	There hold the following statements.
	\begin{itemize}
		\item[1.] The well-posedness of the approximate problem is equivalent to non-singularity of  $\mathbf{A}$.
		\item[2.] If the abstract bilinear form $a(\cdot,\cdot)$ of problem \eqref{pbl var lm} is coercive, $\mathbf{A}$ is positive definite, i.e., defined $M:=\text{dim}(V_h)$, there holds 
		\begin{multline*}
		(\exists \alpha >0: \ \forall u \in V, \ a(u,u) \geq \alpha \|u\|_V^2) \Longrightarrow\\
		(\forall X \in \mathbb{R}^M, \ (\mathbf{A}X,X) \geq 0 \ \text{and} \ ((\mathbf{A}X,X)=0 \Longleftrightarrow X=0) ).
		\end{multline*}
		\item[3.] If $a(\cdot,\cdot)$ is symmetric, $\mathbf{A}$ is symmetric.
		\item[4.] Condition \eqref{infsup discreta} is equivalent to $\ker{(A)}=\{0\}$, i.e., to the injectivity of $\mathbf{A}$.
		\item[5.] Condition \eqref{non null discreta} is equivalent to rank$(\mathbf{A})=$dim$(V_h)$, i.e., to the surjectivity of $\mathbf{A}$.
		\item[6.] If dim$(V_h)=$ dim$(W_h)$, \eqref{infsup discreta} is equivalent to \eqref{non null discreta}.
	\end{itemize}
	For a proof we refer to (\cite{lect2}).
\end{oss}
\vspace{0.3cm} 
Let us now consider the convergence of the discrete solution to the abstract solution. In this connection, we recall some classic results.
\begin{lemma}[Galerkin orthogonality]
	Let $u$ be the solution of the general abstract problem \eqref{pbl var} and $u_h$ be the solution of the general approximate problem \eqref{galerkin petrov}, then there holds
	\begin{equation}\label{gal ort}
		a(u-u_h,v_h)=0 \quad \forall v_h \in V_h.
	\end{equation}
\end{lemma}
\begin{proof}
	This is immediate using the bilinearity of $a(\cdot,\cdot)$.
\end{proof}

\begin{lemma}[Céa]\label{céa}
	Let the hypothesis of Lax-Milgram Lemma \ref{lax-milgram} be satisfied. Assuming $u$ is the solution of \eqref{pbl var lm} and $u_h$ is the solution of \eqref{galerkin bubnov}, then the following estimate holds
	\begin{equation}\label{céa bound}
		\|u-u_h\|_V \leq \frac{C_a}{\alpha} \inf_{v_h \in V_h} \|u-v_h\|_V,
	\end{equation}
	where $C_a$ is the continuity constant of $a(\cdot,\cdot)$.
\end{lemma}
\begin{proof}
	Let $v_h \in V_h$. As a consequence of coercivity, bilinearity, continuity and Galerkin ortogonality there hold
	\begin{equation*}
		\begin{split}
			\alpha \|u-u_h\|^2 &\overset{\eqref{coerc}}\leq a(u-u_h,u-u_h) \\
			&=a(u-u_h,u-v_h+v_h-u_h)\\
			&=a(u-u_h,u-v_h)+a(u-u_h,v_h-u_h) \quad \text{(bilinearity of $a(\cdot,\cdot)$)}\\
			&\overset{\eqref{gal ort}}=a(u-u_h,u-v_h)\\
			&\leq C_a \|u-u_h\|_V\|u-v_h\| \quad \text{(continuity of $a(\cdot,\cdot)$)}.
		\end{split}
	\end{equation*}
\end{proof}

\begin{oss}
	Note that \eqref{céa bound} is a quasi-optimality estimate. Typically, a quasi-optimality bound is an important result, since, according to that, the error committed by the Galerkin method depends on two terms. The first one, which is $\frac{C_a}{\alpha}$ in this case, is a stability term: it is only related to the continuous problem. The second one, i.e., the best approximation error $\inf_{v_h \in V_{h}} \|u-v_h\|_V$, measures how well the discrete space is able to approximate the solution $u$. In particular, if $(V_h)_h$ are dense in $V$, i.e., if
	\begin{equation*}
	 \inf_{v_h \in V_{h}} \|u-v_h\|_V \underset{h \rightarrow 0}{\longrightarrow} 0 \quad \forall u \in V,
	\end{equation*}
	estimate \eqref{céa bound} tells us that we eventually reach convergence.
\end{oss}

In the Banach-Ne$\mathbf{\breve{c}}$as-Babu$\mathbf{\breve{s}}$ka setting, quasi-optimality estimates are also obtained.

Let us assume that the general abstract problem \eqref{pbl var} is well-posed. We denote by $C_a$ the continuity constant of the bilinear form $a(\cdot,\cdot)$. Let us suppose that the discrete problem \eqref{galerkin petrov} satisfies the two BNB conditions \eqref{infsup discreta}, \eqref{non null discreta}, and that the discrete inf-sup condition is independent of the index $h$ (this is the case of greatest interest), and we denote it by $\beta_{dis}$. The following result holds.

\begin{prop}\label{err BNB}
	Assuming $u$ is the solution of \eqref{pbl var lm} and $u_h$ is the solution of \eqref{galerkin petrov}, then the following quasi-optimality estimate holds
	\begin{equation*}
	\|u-u_h\|_W \leq \Bigg[1+\frac{C_a}{\beta_{dis}}\Bigg]\inf_{w_h \in W_h}\|u-w_h\|_W,
	\end{equation*}
	where $C_a$ is the continuity constant of $a(\cdot,\cdot)$ and $\beta_{dis}$ is the inf-sup uniform constant of the discrete problem \eqref{galerkin petrov}.
\end{prop}

\begin{proof}
	For any $w \in W$ we define $w_h:= G_h w \in W_h$ as the Galerkin projection satisfying 
	\begin{equation*}
	a(G_hw,v_h)=a(w, v_h) \quad \forall v_h \in V_h,
	\end{equation*}
	which is well defined thanks to the well-posedness of the discrete problem. Hence, by using the stability estimate \eqref{infsup discreta} with $\beta_{dis}$ and the continuity of $a(\cdot,\cdot)$ with $C_a$, there hold
	\begin{equation*}
	\begin{split}	
	\beta_{dis}\|G_hw\|_{W}  
	&\leq \sup_{0 \neq v_h \in V_h} \frac{a(G_hw, v_h)}{\|v_h\|_{V}} \quad \text{(discrete infsup)}\\
	&=\sup_{0 \neq v_h \in V_h} \frac{a(w, v_h)}{\|v_h\|_{V}}\\
	 &\leq  C_a\|w\|_{W} \quad \text{(boundness of $a(\cdot,\cdot)$)}.
	 \end{split}	
	\end{equation*}
	Since $u_h=G_h u$ and $w_h = G_h w_h$ for all $w_h \in W_h$, we conclude
	\begin{equation*}
	\begin{split}
	\|u-u_h\|_{W} &\leq \|u-w_h\|_{W}+\|G_h(w_h-u)\|_{W} \\
	& \leq \Bigg[ 1 + \frac{C_a}{\beta_{dis}} \Bigg] \|u-w_h\|_{W}.
	\end{split}
	\end{equation*}
\end{proof}

\section{Preliminaries on compact perturbations}\label{sec prel ode}
In this Section we recall the definitions of two important classes of bounded linear operators between Hilbert spaces, emphasizing the most important properties that will be useful in Chapter \ref{ch 3}. The main references of this Section are (\cite{sayas, Brezis2011, var_techniques, moiola}). 
\begin{definizione}[Compact operator]
	Let $H_1$ and $H_2$ be two Hilbert (or Banach) spaces. A bounded linear operator $K: H_1\rightarrow H_2$ is compact if the image of a bounded sequence admits a converging subsequence, i.e., the image of a bounded set in $H_1$ is pre-compact in $H_2$.
\end{definizione} 

\begin{definizione}[Fredholm operator]\label{fredh}
	Let $H_1$ and $H_2$ be two Hilbert spaces. A bounded linear operator $T: H_1\rightarrow H_2$ is a Fredholm operator of index 0 if it is the sum of an invertible one and a compact one.
\end{definizione} 
\begin{oss}
	Actually, definition \ref{fredh} is a possible characterization of Fredholm operators of index 0. E.g. (\cite{Brezis2011}) defines $T: H_1\rightarrow H_2$ (bounded linear operator between Hilbert spaces) as Fredholm operator of $Ind(T):=dim(KerT)-dim(ImT)^\perp$ if $dim(KerT), \ dim(ImT)^\perp < \infty$. For a proof of the equivalence with this classic definition we refer to (\cite{moiola}). 
\end{oss}
Henceforth, Fredholm operators of index 0 will be refer to as Fredholm operators.

An important result is the \textit{Fredholm alternative}, which, in its simplest form, reads as follows; see (\cite{Brezis2011}).
\begin{theorem}[Fredholm alternative]\label{fredh th}
	Let $T: H_1\rightarrow H_2$ be a Fredholm operator. Then $T$ is injective if and only if it is surjective. In this case its inverse is bounded.
\end{theorem}

\begin{oss}
	Note that the boundedness of the inverse of an invertible Fredholm operator is due to the ``bounded inverse Theorem''.
\end{oss}

\begin{oss}
	In a finite dimensional setting Fredholm operators are precisely those associated to square matrices. Indeed, an invertible linear operator between finite-dimensional spaces corresponds to an invertible square matrix, and all finite-range operators are compact, since all bounded sequences of $\mathbb{R}^n$ admits converging subsequences. Thus, Theorem \ref{fredh} is an extension of the finite-dimensional case.
\end{oss}
\hspace{0.2cm}
\\
\textbf{Abstract problem and Galerkin method}. Let $H$ be an Hilbert space and $\mathcal{F} \in H'$. Let now consider the following variational problem
\begin{equation}\label{compact var}
\begin{cases}
\text{Find} \ u \in H \quad \text{such that}\\
a(u,v):=b(u,v)+d(u,v)=\langle\mathcal{F},v \rangle_{H',H} \quad \forall v \in H,
\end{cases}
\end{equation}
where:
\begin{itemize}
	\item[1.] The bilinear form $b(\cdot,\cdot)$ is bounded and coercive.
	\item[2.] The bilinear form $d(\cdot,\cdot)$ defines a compact operator $\mathcal{D}: H \rightarrow H'$ by setting
	\begin{equation*}
	\langle \mathcal{D}u,v \rangle_{H',H}:=d(u,v), \quad u,v \in H.
	\end{equation*}
	\item[3.] The linear operator $\mathcal{A}:=\mathcal{B}+\mathcal{D}: H \rightarrow H'$ associated to the bilinear form $a(\cdot,\cdot)$ is injective, where
	\begin{equation*}
	\langle \mathcal{B}u,v \rangle_{H',H}:=b(u,v) \quad u,v \in H.
	\end{equation*}
\end{itemize}
Assumptions $1$ and $2$, and Lax-Milgram Lemma \ref{lax-milgram}, ensure that the operator $\mathcal{A}$ is Fredholm. As a consequence of assumption 3 and Theorem \ref{fredh}, problem \eqref{compact var} is well-posed.

Let now consider a family of finite-dimensional subspaces $V_h \subset H$ directed in a real non-negative parameter $h \rightarrow 0$ and let $\pi_h: H \rightarrow V_{h}$ be the orthogonal projection. Let us also assume that
\begin{equation}
\pi_h u \underset{h \rightarrow 0}{\longrightarrow} u, \quad \forall u \in H,
\end{equation}
i.e., $(V_{h})_h$ is a dense discrete family of subspaces in $H$.
As in Section \ref{galerkin method}, we refer to the following problem as the Galerkin approximation of \eqref{compact var}:
\begin{equation}\label{gal comp}
\begin{cases}
\text{Find} \ u_h \in V_{h} \quad \text{such that}\\
a(u_h,v_h)=\langle\mathcal{F}_{|V_h},v_h \rangle \quad \forall v_h \in V_{h}.
\end{cases}
\end{equation}
In our work the subscript $h$ corresponds to the sequence of mesh-sizes of our discretization. 

The results that we are going to recall establish the uniform (w.r.t. $h$) well posedness of problem \eqref{gal comp}, provided the mesh-size is \textit{small enough} (Proposition 8.8 of \cite{var_techniques}). Moreover, a quasi-optimality estimate will follow.

\begin{prop}\label{prop infsup comp}
	In the hypothesis (1)-(3) for the bilinear forms, there exist two constants $C,\overline{h}>0$ such that the following inf-sup estimate holds
	\begin{equation*}
	\beta \|u_h\|_{H} \leq  \sup_{0 \neq u_h \in V_{h}} \frac{b(u_h,v_h)+d(u_h,v_h)}{\|v_h\|_H} \quad \forall u_h \in V_{h}, \ \forall h \leq \overline{h}.
	\end{equation*}
\end{prop}

\begin{oss}
	The bound on the mesh-size and the inf-sup constant are not explicit because the proof of this result is made by contradiction.
\end{oss}

\begin{cor}\label{well-pos comp dis}
	Let the hypothesis of Proposition \ref{prop infsup comp} be satisfied. Let $h \leq \overline{h}$. Then problem \eqref{gal comp} is well-posed, with the following stability estimate
	\begin{equation*}
	\|u_h\|_H \leq \frac{1}{\beta} \|\mathcal{F}\|_{H'}, \quad \mathcal{F} \in H'.
	\end{equation*}
\end{cor}
Note that Corollary \ref{well-pos comp dis} is a consequence  of Remark \ref{remark discrete system} and of BNB Theorem \ref{BNB}.
\begin{cor}\label{céa comp}
	Under the hypothesis of Proposition \ref{prop infsup comp}, let $h \leq \overline{h}$.  Then a quasi-optimality estimate holds
	\begin{equation}\label{quas opt comp}
	\|u-u_h\|_H \leq \Bigg(1+\frac{\|\mathcal{B}+\mathcal{D}\|}{\beta} \Bigg) \inf_{v_h \in V_{h}} \|u-v_h\|_H,
	\end{equation}
	where $u,u_h$ are the solutions of \eqref{compact var},\eqref{gal comp} respectively and with $\|\cdot\|$ we denote the operator norm of linear bounded maps from $H$ to $H'$. 
\end{cor}

\begin{oss}
	Note that Corollary \ref{céa comp} is a consequence of Proposition \ref{err BNB} and tells that we eventually reach convergence for $h \rightarrow 0$, since the succession of discrete spaces is dense in $H$.
\end{oss}

\subsubsection{Galerkin method applied to G\aa rding-type problems}
In this Section we consider a special case of Fredholm operators. Roughly speaking, we consider Fredholm operators whose compact perturbation we can "quantify". For these operators it is possible to have stability and error estimates with explicit constants. The theory we are interested in is that of \textit{Galerkin method applied to G\aa rding-type problems}. Our main reference are (\cite{spence14, moiola}).

Let $H$ be a real Hilbert space, $V_h \subset H$ a finite-dimensional subspace, $a(\cdot,\cdot)$ and $\mathcal{F}(\cdot)$ a bounded bilinear form and a bounded linear operator on $H$, respectively. Let us now consider the following abstract variational problem 
\begin{equation}\label{var gard}
\begin{cases}
\text{Find} \ u \in H \ \text{such that}\\
a(u,v)=\langle \mathcal{F},v\rangle_{H', H} \quad \forall v \in H,
\end{cases}
\end{equation}
and its Galerkin discretization
\begin{equation}\label{gal gard}
\begin{cases}
\text{Find} \ u_h \in V_h \ \text{such that}\\
a(u_h,v_h)=\langle \mathcal{F}_{|V_h},v_h\rangle_{V_h', V_h} \quad \forall v_h \in V_h.
\end{cases}
\end{equation}

\begin{theorem}[Galerkin method with G\aa rding inequality]\label{theo gard}
	Let $H \subset V$ be real Hilbert spaces and the inclusion be compact. Let $a(\cdot,\cdot)$ be a bounded bilinear form on $H$ with respect to a continuity constant $C_a >0$:
	\begin{equation*}
	|a(v,w)|\leq C_a \|v\|_H \|w\|_H \quad \forall v,w \in H, 
	\end{equation*}
	that satisfies the G\aa rding inequality with respect to $\alpha, C_v >0$
	\begin{equation}\label{gard}
	a(v,v) \geq \alpha \|v\|_H^2-C_V \|v\|_V^2  \quad \forall v \in H.
	\end{equation}
	Assume that the only $u_0 \in H$ such that $a(u_0,v)=0$ for all $v \in H$ is $u_0=0$ (so that the variational problem \eqref{var gard} is well-posed for any right-hand side). Let $\mathcal{F}(\cdot)$ be a bounded linear operator on H and $u$ be the solution of the variational problem \eqref{var gard}.\\
	Given $g \in V$, let $z_g \in H$ be the solution of the adjoint problem
	\begin{equation}\label{adj}
	a(v,z_g)=(g,v)_V \quad \forall v \in H,
	\end{equation}
	where $(\cdot,\cdot)_V$ is the scalar product in $V$. Let $V_h \subset H$ be a finite-dimensional subspace of H and define
	\begin{equation}\label{eta gard}
	\eta(V_h):=\sup_{0 \neq g \in V} \inf_{v_h \in V_h} \frac{\|z_g-v_h\|_H}{\|g\|_V}.
	\end{equation}
	If $\eta(V_h)$ satisfies the threshold condition
	\begin{equation}\label{cond eta}
	\eta(V_h) \leq \frac{1}{C_a} \sqrt{\frac{\alpha}{2 C_V}},
	\end{equation}
	then the Galerkin method \eqref{gal gard} is well-posed with the following stability estimate
	\begin{equation*}
	\|u_h\|_H \leq C_{stab}\Bigg(1+\frac{2C_a}{\alpha}\Bigg)\|\mathcal{F}\|_{H'},
	\end{equation*}
	where $C_{stab}$ is the stability constant of the abstract problem \eqref{var gard}, and its solution $u_h$ satisfies the quasi-optimality bound
	\begin{equation*}
	\|u-u_h\|_H \leq \frac{2 C_a}{\alpha} \inf_{v_h \in V_h} \|u-v_h\|_H.
	\end{equation*}
\end{theorem}

\begin{oss}\label{oss gard}
	Note the following facts.
	\begin{itemize}
		\item [1.] The G\aa rding inequality \eqref{gard} and the compact inclusion $H \subset V$ imply that the \textit{solution-to-data} operator $\mathcal{A}: H \rightarrow H'$,$u \mapsto a(u,\cdot)$ is Fredholm (of index $0$). For a proof we refer to (\cite{spence14}). Therefore, the assumption
		\begin{center}
			``the only $u_0 \in H$ such that $a(u_0,v)=0$ for all $v \in H$ is $u_0=0$''
		\end{center}
		implies the well-posedness of the variational problem \eqref{var gard}, since the above assumption is equivalent to the injectivity of $\mathcal{A}$. In particular we can apply Theorem \ref{fredh th}.
		\item[2.] The well-posedness of the adjoint problem \eqref{adj} (w.r.t. the dual norm $\|g\|_{H'}$, and hence w.r.t. $\|g\|_V$) is a consequence of the well-posedness of the primal problem \eqref{var gard}. Indeed, if $\mathcal{A}:H \rightarrow H'$ is an isomorphism, its adjoint operator $\mathcal{A^*}:H \rightarrow H'$ is an isomorphism, with $(\mathcal{A}^*)^{-1}=(\mathcal{A}^{-1})^*$, see (\cite{adjoint}). Also, as a consequence of $(\mathcal{A}^*)^{-1}=(\mathcal{A}^{-1})^*$, there holds $\|(\mathcal{A}^*)^{-1}\|=\|\mathcal{A}^{-1}\|$. Hence, the stability constant of the adjoint problem (w.r.t. the dual norm $\|g\|_{H'}$) is equal to the stability constant of the primal problem.
		\item [3.] The parameter $\eta(V_h)$ precisely quantify how well $V_h$ approximates the solution of the adjoint problem, whose data $g$ is an element of the larger space $V$. Thus, roughly speaking, Theorem \ref{theo gard} states that if the discrete space is sufficiently fine, the Galerkin method applied to a G\aa rding-type (well-posed) problem is well-posed, stable and quasi-optimal.
	\end{itemize}
\end{oss}

%% file: Chapters/Chapter3.tex

\chapter{Second-order ordinary differential equation}\label{ch 3} 

\label{Chapter3} 




\section{Variational formulation for $\partial_{tt}u+ \mu u=f$}\label{sect var form}
As a model problem we consider the following second-order linear equation: 
\begin{equation}\label{eq ode}
\partial_{tt}u(t)+ \mu u(t)=f(t), \quad \text{for} \ t \in (0,T), \quad u(0)=\partial_{t}u(t)_{|t=0}=0,
\end{equation}
where $\mu >0$.

The variational formulation of \eqref{eq ode} reads as follows:
\begin{equation}\label{var ode}
\begin{cases}
\text{Find} \ u \in H^1_{0,*}(0,T) \quad \text{such that}\\
a(u,v)=\langle f,v \rangle_{(0,T)} \quad  \forall v \in H^1_{*,0}(0,T),
\end{cases}
\end{equation}
where $T>0$ and $f \in [H^1_{*,0}(0,T)]' $ are given, and where the bilinear form $a(\cdot,\cdot): H^1_{0,*}(0,T) \times H^1_{*,0}(0,T) \longrightarrow \mathbb{R}$ is 
\begin{equation}\label{bil ode}
a(u,v):=-\langle \partial_{t}u, \partial_{t}v \rangle_{L^2(0,T)}+ \mu \langle u,v \rangle_{L^2(0,T)},
\end{equation} 
for all $u \in H^1_{0,*}(0,T)$, $v \in H^1_{*,0}(0,T)$. The notation $\langle \cdot,\cdot \rangle_{(0,T)}$ denotes the duality pairing as extension of the inner product in $L^2(0,T)$, and the Sobolev spaces $H^1_{0,*}(0,T), H^1_{*,0}(0,T)$ are introduced in Section \ref{sec sobolev}. Note that the first initial condition $u(0)=0$ is incorporated in the solution space $H^1_{0,*}(0,T)$, whereas the second initial condition $\partial_{t}u(t)_{|t=0}=0$ is considered as a natural
condition in the variational formulation.

Thanks to the Cauchy-Schwarz inequality and the Poincaré inequality \eqref{Poinc}, the bilinear form $a(\cdot,\cdot)$ is bounded with
\begin{equation}\label{cont of a}
|a(u,v)| \leq \Bigg(1+\frac{4 T^2 \mu }{\pi^2} \Bigg) |u|_{H^1(0,T)}|v|_{H^1(0,T)},
\end{equation}
for all $(u,v) \in H^1_{0,*}(0,T) \times H^1_{*,0}(0,T)$. 

In order to prove the well-posedness of \eqref{var ode} we consider an equivalent variational problem. Thus, we introduce the isomorphism
\begin{gather}\label{HT}
\overline{\mathcal{H}}_T: H^1_{0,*}(0,T) \longrightarrow H^1_{*,0}(0,T)\\
u \mapsto u(T)-u(\cdot),
\end{gather}
where its inverse is given by
\begin{gather*}
\overline{\mathcal{H}}_T^{-1}: H^1_{*,0}(0,T) \longrightarrow H^1_{0,*}(0,T)\\
v \mapsto v(0)-v(\cdot).
\end{gather*}
Since the above mappings are actually isometries with respect to $|\cdot|_{H^1(0,T)}$, then the well-posedness of  \eqref{var ode} is equivalent to the well-posedness (with the same stability constant) of the following problem
\begin{equation}\label{var ode equiv}
\begin{cases}
\text{Find} \ u \in H^1_{0,*}(0,T) \ \text{such that}\\
a(u,\overline{\mathcal{H}}_Tv)=\langle f,\overline{\mathcal{H}}_Tv \rangle_{(0,T)} \quad  \forall v \in H^1_{0,*}(0,T),
\end{cases}
\end{equation}
where the solution and test space coincide. Note that the continuity of the bilinear form in \eqref{var ode equiv} is an immediate consequence of estimate \eqref{cont of a}. Unfortunately, at least for $\mu$ sufficiently large, the bilinear form of \eqref{var ode equiv} is not coercive, hence, we cannot rely on the classical Lax-Milgram Lemma \ref{lax-milgram}, but we need more specific tools. The branch of functional analysis that we exploit is called \textit{Fredholm theory} and studies compact perturbations of linear bounded operators. In particular, there holds the following Theorem.

\begin{theorem}[\cite{Zank2020}]\label{zank}
	For a given $f \in [H^1_{*,0}(0,T)]'$, there exists a unique solution $u \in H^1_{0,*}(0,T)$ of the variational formulation \eqref{var ode equiv}, and the following a priori estimate holds
	\begin{equation}\label{stab ode}
	|u|_{H^1(0,T)}\leq \Bigg(\frac{2+\sqrt{\mu}T}{2}\Bigg) \|f\|_{[H^1_{*,0}(0,T)]'}.
	\end{equation}
	In addition, \eqref{stab ode} is optimal with respect to the order of $\mu$ and $T$.
\end{theorem}
The well-posedness of \eqref{var ode equiv} is an immediate consequence of the results presented in Section \ref{sec prel ode}. Indeed, defining two bounded linear operators $\mathcal{B},\mathcal{D}:H^1_{0,*}(0,T) \rightarrow [H^1_{0,*}(0,T)]'$ by
\begin{gather*}
\langle \mathcal{B}u,v \rangle _{(0,T)} := \langle \partial_tu,\partial_t v\rangle_{L^2(0,T)}, \quad u,v \in H^1_{0,*}(0,T)\\
\langle \mathcal{D}u,v \rangle _{(0,T)} := \langle u,v(T)-v \rangle_{L^2(0,T)}, \quad u,v \in H^1_{0,*}(0,T),
\end{gather*}
 the map $\mathcal{B}+\mu \mathcal{D}: H^1_{0,*}(0,T) \rightarrow [H^1_{0,*}(0,T)]'$ is an injective Fredholm operator of index 0, where the compactness of $\mathcal{D}$ follows from the properties of trace operators. However, estimate \eqref{stab ode} gives an explicit dependency relation of its stability constant on $T,\mu$. In order to prove \eqref{stab ode}, we use the notation $\mathcal{A}$ for the bounded linear operator related to the bilinear form \eqref{bil ode}, i.e.,
\begin{gather*}
\mathcal{A}:  H^1_{0,*}(0,T) \rightarrow [H^1_{*,0}(0,T)]' \quad \text{s.t.}\\
\langle \mathcal{A}u,v \rangle _{(0,T)}:=a(u,v), \quad u \in H^1_{0,*}(0,T), v \in H^1_{*,0}(0,T).
\end{gather*}
Since $\mathcal{A}$ is an isomorphism, its adjoint operator $\mathcal{A}^*:H^1_{*,0}(0,T)\rightarrow  [H^1_{0,*}(0,T)]' $ is an isomorphism with $(\mathcal{A}^*)^{-1}=(A^{-1})^*$, see (\cite{adjoint}). As a consequence, given $g \in [H^1_{0,*}(0,T)]'$, the adjoint problem 
\begin{equation}\label{adj dim}
\begin{cases}
\text{Find} \ z \in H^1_{*,0}(0,T) \quad \text{such that}\\
a(w,z)=\langle g,w \rangle_{(0,T)} \quad \forall w \in H^1_{0,*}(0,T)
\end{cases}
\end{equation}
is well-posed. In particular, it is possible to compute the exact solution of problem \eqref{adj dim} using a Green's function, if the right-hand side  $g \in [H^1_{0,*}(0,T)]'$ depends on a fixed $u \in H^1_{0,*}(0,T)$; for more details we refer to (\cite{Zank2020}). For the optimality of the estimate we refer to Theorem 4.2.6 of (\cite{Zank2020}).

\section{Isogeometric discretization}\label{sec iga}
As discrete spaces for the Galerkin discretization of \eqref{var ode} we choose to consider spline spaces of degree two with maximal regularity, i.e., piecewise polynomials of degree two with global $C^1$ regularity.

Given a positive integer $N$, let $\Xi:=\{t_1, \ldots, t_{N+3}\}$ be an open knot vector in $[0,T]$ with the interior knots that appear just once, i.e., $0=t_1=t_2=t_3 < \ldots  <t_{N+1}=t_{N+2}=t_{N+3}=T$. By means of Cox-de Boor recursion formulas \eqref{cox de boor} we define the quadratic univariate B-spline basis functions ${b}_{i,2}: [0,T] \rightarrow \mathbb{R}$ for $i=1, \ldots, N$ (we omit the subscript $\Xi$ to lighten the notation). Thus, the univariate spline space is defined as 
\begin{equation*}
{S}^2_h:= \text{span}\{{b}_{i,2}\}_{i=1}^N,
\end{equation*}
where $h$ is the mesh-size, i.e., $h:=\max{\{|t_{i+1}-t_i|\ : \ i=1,\ldots,N+2}\}$. 

We introduce the spline space with initial conditions as
\begin{equation}\label{iga space}
\begin{split}
V^h_{0,*} &:= {S}^2_h \cap H^1_{0,*}(0,1)
= \{v_h \in {S}^2_h \ : \ {v}_h(0)=0 \}\\
&= \text{span}\{{b}_{i,2}\}_{i=2}^N,
\end{split}
\end{equation}
which is our isogeometric solution space, and the spline space 
\begin{equation}\label{iga test space}
\begin{split}
{V}^h_{*,0} &:= {S}^2_h \cap H^1_{*,0}(0,1)
= \{ {v}_h \in {S}^2_h \ : \ {v}_h(T)=0 \}\\
&= \text{span}\{ {b}_{i,2}\}_{i=1}^{N-1},
\end{split}
\end{equation}
which is our isogeometric test space. Note that $V^h_{0,*}=\text{span}\{{b}_{i,2}\}_{i=2}^N$  and ${V}^h_{*,0}=\text{span}\{ {b}_{i,2}\}_{i=1}^{N-1}$ since the first and last B-spline basis functions of an open knot vector are interpolating functions.

\begin{oss}
	Typically, in isogeometric discretizations, and in the GeoPDEs library that we are going to use for the numerical experiments, one considers the parametric domain $[0,1]^n$ and the splines/NURBS on this domain. Once these spaces are constructed, 
	one maps, via a NURBS function $F$,	
	the parametric domain to the physical domain of interest, \begin{minipage}[r]{0.6\linewidth} 
		\includegraphics[width=\textwidth]{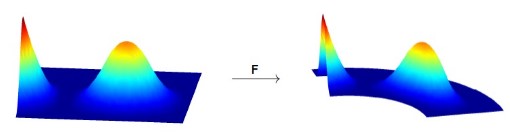}
	\end{minipage} \\ and the discrete solution and test spaces are the pushforward by $F$ of splines/NURBS spaces  on $[0,1]^n$. \\
	In our case, the map $F$ would be $F: [0,1] \rightarrow [0,T]$ s.t. $\tau \mapsto T\tau$. Therefore, doing the pushforward of the spline spaces on the parametric interval, we would obtain spline spaces on $[0,T]$ with B-spline basis functions that are the pushforward of the B-spline basis functions on $[0,1]$. It is then sufficient for us to construct the discrete spaces directly on $[0,T]$.
\end{oss}

A conforming Galerkin-Petrov isogeometric discretization
of \eqref{var ode} is the following problem
\begin{equation}\label{iga ode}
\begin{cases}
\text{Find} \ u_h \in V^h_{0,*} \quad \text{such that}\\
a(u_h,v_h)=\langle f,v_h \rangle_{(0,T)} \quad  \forall v_h \in V^h_{*,0}.
\end{cases}
\end{equation}
As in the continuous framework, the restricted operator
\begin{gather*}
\overline{\mathcal{H}}_{T_{|V^h_{0,*}}}: V^h_{0,*} \longrightarrow V^h_{*,0}\\
u_h \mapsto u_h(T)-u_h(\cdot)
\end{gather*}
is actually an isometric isomorphism with respect to $|\cdot|_{H^1(0,T)}$. Therefore, the well-posedness of \eqref{iga ode} is  equivalent to the well-posedness (with the same stability constant) of the conforming Galerkin-Bubnov isogeometric discretization of \eqref{var ode equiv}:
\begin{equation}\label{iga ode equiv}
\begin{cases}
\text{Find} \ u_h \in V^h_{0,*} \ \text{such that}\\
a(u_h,\overline{\mathcal{H}}_Tv_h)=\langle f,\overline{\mathcal{H}}_Tv_h \rangle_{(0,T)} \quad  \forall v_h \in V^h_{0,*}.
\end{cases}
\end{equation}
The isogeometric spaces $V^h_{0,*}$ define a dense discrete sequence in $H^1_{0,*}(0,T)$ (as we will see in Section \ref{approx properties spline}), directed in the real parameter $h$. Hence, we can conclude the following Theorem.

\begin{theorem}\label{cond well-posed iga}
	There exists two constants $C(T,\mu),\overline{h}>0$ such that, if $h\leq \overline{h}$, problem \eqref{iga ode equiv} is well-posed, with the stability estimate
	\begin{equation*}
	|u_h|_{H^1(0,T)} \leq C(T,\mu) \|f\|_{[H^1_{*,0}(0,T)]'} \quad f \in [H^1_{*,0}(0,T)]',
	\end{equation*}
	where $u_h$ is the unique solution of \eqref{iga ode equiv}.
	Moreover, if $h\leq \overline{h}$, a quasi-optimality estimate holds
	\begin{equation*}
	|u-u_h|_{H^1(0,T)} \leq \Big(1+C(T,\mu) \|\mathcal{A}\| \Big) \inf_{v_h \in V^h_{0,*}} |u-v_h|_{H^1(0,T)},
	\end{equation*}
	where $u \in H^1_{0,*}(0,T)$ is the unique solution of \eqref{var ode equiv}, and where $\|\cdot\|$ is the norm of the \textit{solution-to-data} operator $\mathcal{A}:H^1_{0,*}(0,T) \rightarrow [H^1_{0,*}(0,T)]'$, related to the bilinear form of the variational formulation \eqref{var ode equiv}.
\end{theorem}
\begin{proof}
	Theorem \ref{cond well-posed iga} is a straightforward consequence of the results presented in Section \ref{sec prel ode}, i.e., Proposition \ref{prop infsup comp}, Corollary \ref{well-pos comp dis} and Corollary \ref{céa comp}.
\end{proof}

\subsection{Approximation properties of spline spaces with an initial boundary condition}\label{approx properties spline}
So far we have shown that, if the IGA discretization is \textit{sufficiently fine}, problem \eqref{iga ode equiv} is well posed, stability holds and we eventually reach convergence. However, we would like to make explicit the threshold on the mesh-size and (possibly) the stability and quasi-optimality constants. In order to get these results we need a priori error estimate in the Sobolev semi-norm $|\cdot|_{H^1(0,T)}$
for approximation in spline spaces of maximal smoothness on grids defined by arbitrary break points. As pointed out in the paper (\cite{n-width}), classical error estimates for spline approximation are expressed in terms of:
\begin{itemize}
	\item [1.] A power of the mesh-size.
	\item [2.] An appropriate semi-norm of the function to be approximated.
	\item [3.] A constant which is independent of the previous quantities.
\end{itemize}

However, we are interested in estimates with the constant of point 3 made explicit. In this respect, article (\cite{n-width}) is relevant to us, since we slightly modify the construction of this work in order to obtain the estimates we need for test and trial spaces of our interest. In particular, the authors study the approximation properties of spline spaces, without boundary conditions, and of periodic spline spaces. We partially extend their work by including an initial boundary condition.

Let $S_\Xi^p(0,T)$ be the spline space of degree $p \geq 0$ and maximal regularity, where $\Xi$ is the open knot vector that defines the B-spline basis functions of $S_\Xi^p(0,T)$. We firstly observe that we use the terms \textit{knot vector} and \textit{break points} as introduced in Section \ref{sec splines}, unlike (\cite{n-width}) in which the sequence of break points is called knot vector. Let $Q_p^q: H^q(0,T) \rightarrow S_\Xi^p(0,T)$, $q=0,\ldots,p$, be a sequence of bounded linear operators such that 
\begin{gather}
Q^0_p:=P_p \quad \text{is the $L^2(0,T)$ orthogonal projection on} \ S^p_\Xi(0,T) \label{Qp0}, \\
\begin{split}
(Q^q_p&u)(t):=u(0) + \int_0^t(Q_{p-1}^{q-1}\partial_t u)(s) \ ds \label{Qpq},\\
& \quad \ 1 \leq q \leq p, \ \forall u \in H^q(0,T),
\end{split}
\end{gather}
Firstly, let now observe the reason why the operators $Q^p_q$ maps $H^q(0,T)$ into $S^p_\Xi(0,T)$. Let $K$ be the integral operator such that, for $f \in L^2(0,T)$, 
\begin{equation}\label{K}
Kf(t):=\int_0^t f(s) \ ds.
\end{equation}
Recall from Theorem 17 of (\cite{chapter1}) that the spline space $S_\Xi^p(0,T)$ satisfies 
\begin{equation}\label{der spline}
\partial^q_tS_\Xi^p(0,T)=S_\Xi^{p-q}(0,T), \quad \text{for any} \ q=0,\ldots,p,
\end{equation}
where, by abuse of notation, we denote by the letter $\Xi$ both the open knot vector $\{t_1, \ldots, t_{N+p+1}\}$ (with $0=t_1=\ldots=t_{p+1} <\ldots  <t_{N+1}=\ldots =t_{N+p+1}=T$) and the one obtained from $\{t_1, \ldots, t_{N+p+1}\}$ by reducing the external nodes from $p+1$ to $p-q+1$. 
Thus, as a consequence of \eqref{der spline} and of the Fundamental Theorem of Calculus, there holds
\begin{equation}
S_\Xi^p(0,T)=\mathbb{P}_0 + K\Big(S_\Xi^{p-1}(0,T)\Big), \quad \forall p \geq 1
\end{equation}
where $\mathbb{P}_0$ is the space of constant functions. Hence, by an inductive argument, the range of $Q_p^q$ is a subspace of $S_\Xi^p(0,T)$. The linearity and continuity of $Q^p_q$ are straightforward.

\begin{prop}
	The maps $Q_p^q$ defined in \eqref{Qp0} and \eqref{Qpq} are projection operators with range $S_\Xi^p(0,T)$. They also satisfies
	\begin{equation}\label{commut Q}
	\partial_t Q_p^q = Q^{q-1}_{p-1}\partial_t, \quad \text{for all} \ p \geq 1, \ q=1,\ldots,p.
	\end{equation}
\end{prop}

\begin{proof} 
	Equality \eqref{commut Q} is satisfied by definition. 
	
	If $q=0$, $Q^q_p$ is a projection operator by definition. We use an inductive argument in order to prove that $Q^q_p\big(Q^q_pu \big) = Q^q_p u$ for all $u \in H^q(0,T)$ if $p \geq 1$ and $q=1,\ldots,p$. Let $p=q=1$, then
	\begin{equation*}
	\begin{split}
	Q^1_1\big(Q^1_1u \big)(t)&=Q^1_1\Big( u(0) +\int_0^{(\cdot)}(Q_0^0\partial_t u)(s) \ ds \Big)(t) \\
	&=Q^1_1 \big( u(0) \big) + Q^1_1\Big(\int_0^{(\cdot)}(Q_0^0\partial_t u)(s) \ ds \Big) (t) \quad \text{(linearity of $Q^p_q$)} \\
	&=u(0)+ \int_0^t Q^0_0 \Big(\partial_t \int_0^{(\cdot)}(Q_0^0\partial_t u)(s) \ ds \Big)(\tau) \ d\tau \quad \text{(definition of $Q^p_q$)} \\
	&\overset{\eqref{commut Q}}=u(0)+ \int_0^t Q^0_0 \Big(\partial_t \int_0^{(\cdot)}(\partial_t Q_1^1u)(s) \ ds \Big)(\tau) \ d\tau\\
	&=u(0)+ \int_0^t Q^0_0 \big(\partial_t Q^1_1u(\tau)-\partial_t(Q^1_1u(0)) \big) \ d\tau \quad \text{(Fundamental Th. of Calculus)}\\
	&\overset{\eqref{commut Q}}=u(0)+ \int_0^t Q^0_0 \big(Q^0_0\partial_t u(\tau)\big) \ d\tau \\
	&=u(0)+\int_0^t (Q^0_0 \partial_t u)(\tau) \ d\tau \quad \text{($Q^0_0$ is $L^2(0,T)$-projection)}.
	\end{split}
	\end{equation*}
    Let now assume that the statement is true for $p \geq 1, \ q=1,\ldots,p$, and let now consider $p+1, \ q=1,\ldots,p+1$. Hence, as before, there hold the following identities
	\begin{equation*}
	\begin{split}
	Q^q_{p+1}\big(Q^q_{p+1}u \big)(t)&=Q^q_{p+1}\Big( u(0) +\int_0^{(\cdot)}(Q_p^{q-1}\partial_t u)(s) \ ds \Big)(t) \\
	&=u(0)+ \int_0^t Q_p^{q-1} \Big(\partial_t \int_0^{(\cdot)}(Q_p^{q-1}\partial u)(s) \ ds \Big)(\tau) \ d\tau \\
	&=u(0)+\int_0^t (Q_p^{q-1} \partial_t u)(\tau) \ d\tau.
	\end{split}	
	\end{equation*}
If $q=0$, the range of $Q^q_p$ is clearly $S^p_\Xi(0,T)$. Indeed $Q^0_p$ is the $L^2(0,T)$-projection on $S^p_\Xi(0,T)$, hence, in particular, ${Q^0_p}_{|S^p_\Xi(0,T)} \equiv Id$. If $p \geq 1$ and $q=1,\ldots,p$, the fact that the range of $Q^q_p$ is $S^p_\Xi(0,T)$ is straightforward. Indeed, by using an inductive argument and \eqref{der spline}, one can prove that ${Q^q_p}_{|S^p_\Xi(0,T)} \equiv Id$ also for $p \geq 1$ and $q=1,\ldots,p$.
\end{proof}

Let now recall Theorem $1$ of (\cite{n-width}).

\begin{theorem}[E. Sande, C. Manni, H. Speleers]\label{sande,manni,speleers}
	For any sequence of break points that defines the open knot vector $\Xi$ with maximal regularity, let $h$ denote its maximal knot distance, and let $P_p$ denote the $L^2(0,T)$ orthogonal projection onto the spline space $S^p_\Xi(0,T)$. Then, for any function $u \in H^r(0,T)$ with $r \geq 1$,
	\begin{equation}\label{proj L2}
	\|u-P_pu\|_{L^2(0,T)} \leq \Big ( \frac{h}{\pi} \Big)^r |u|_{H^r(0,T)},
	\end{equation}
	for all $p \geq r-1$.
\end{theorem}

\begin{oss}
	Let $P$ be the $L^2(0,T)$ orthogonal projection onto a finite dimensional subspace $\mathcal{X}$ of $L^2(0,T)$. For $A \subseteq L^2(0,T)$ we define
	\begin{equation*}
	E(A,\mathcal{X})=\sup_{u \in A}{\|u-Pu\|_{L^2(0,T)}},
	\end{equation*}
	i.e., $E(A,\mathcal{X})$ is the ``maximal of the minimal distances'' between $A$ and the projection space $\mathcal{X}$.
	The \textit{Kolmogorov $L^2$ n-width} of $A$ is defined as 
	\begin{equation*}
	d_n(A)=\inf_{\mathcal{X} \subset L^2(0,T), dim\mathcal{X}=n} E(A,\mathcal{X}).
	\end{equation*}
	The projection space $\mathcal{X}$ is called an optimal subspace for $d_n(A)$ if $d_n(A)=E(A,\mathcal{X})$.
	
Let $r \geq 1$ and let $A=\{u \in H^r(0,T): \|\partial^r_t{u}\|_{L^2(0,T)} \leq 1 \}$. 
	Clearly, for any finite subspace $\mathcal{X}$ of $L^2(0,T)$ the following estimate holds
	\begin{equation*}
	\|u-Pu\|_{L^2(0,T)} \leq E(A,\mathcal{X})\|\partial^r_t{u}\|_{L^2(0,T)} \quad \forall u \in H^r(0,T),
	\end{equation*}
	and, in (\cite{n-width}), an error estimate of the form
	\begin{equation}\label{proj}
	\|u-Pu\|_{L^2(0,T)} \leq C\|\partial^r_t{u}\|_{L^2(0,T)} \quad \forall u \in H^r(0,T)
	\end{equation}
	is said to be \textit{sharp} if 
	\begin{equation*}
	C=E(A,\mathcal{X}).
	\end{equation*}
Also, in (\cite{n-width}), a projection error estimate of the form \eqref{proj} is said to be \textit{optimal} if the subspace we project onto is optimal for the Kolmogorov $L^2$ n-width of $A$ and the projection error estimate is sharp.

In (\cite{n-width}) the following results are proven.
\begin{itemize}
	\item If $r=1$ and $p=0$ and the sequence of break points that defines $\Xi$ is uniform, then the estimate \eqref{proj L2} is optimal.
	\item If $r=1$ and $p>0$ and the sequence of break points that defines $\Xi$ is uniform, then the estimate \eqref{proj L2} is asymptotically $\big($with respect to the dimension of the spline space $S^p_\Xi(0,T)$ $\big)$ optimal.
\end{itemize}
In general, the authors conjecture that:
\begin{center}
 for all degree $p \geq 0$ there exists a sequence of break points such that for at least an $r=1,\ldots,p+1$, the estimate \eqref{proj L2} is optimal.
 \end{center}
\end{oss}

As a consequence of Theorem \ref{sande,manni,speleers} there holds the following result, partially extending Theorem 3 of (\cite{n-width}).

\begin{theorem}
	Let $u \in H^r(0,T)$ for $r \geq 2$. For any $q=1,\ldots,r-1$ and any sequence of break points that defines the open knot vector $\Xi$ with maximal regularity, let $h$ denote its maximal knot distance, and let $Q^q_p$ be the projection onto $S^p_\Xi(0,T)$ defined in \eqref{Qpq}. Then,
	\begin{equation}\label{err spline}
	\big |u-Q^q_pu \big |_{H^q(0,T)}\leq \Big(\frac{h}{\pi}\Big)^{r-q}\big |u \big |_{H^r(0,T)},
	\end{equation}
	for all $p \geq r-1$.
\end{theorem}

\begin{proof}
Firstly, as a consequence of the Fundamental Theorem of Calculus for absolutely continuous functions, observe that the space $H^r(0,T)$, with $r \geq 1$, satisfies
\begin{equation}\label{Hr}
H^r(0,T)=\mathbb{P}_0+K\big(H^{r-1}(0,T)\big)=\ldots=\mathbb{P}_{r-1}+K^r\big(H^0(0,T)\big),
\end{equation}
where $\mathbb{P}_{r-1}$ is the space of polynomials of degree at most $r-1$ and $K$ is the integral operator defined in \eqref{K}.

From \eqref{Hr} we know that $u \in H^r(0,T)$ can be written as $u=g+K^qv$, for $g \in \mathbb{P}_{q-1} \subset S^{p}_\Xi$, with $q \geq 1$, and $v \in H^{r-q}(0,T)$. Since the projection operator $Q^q_p: H^q(0,T) \rightarrow S^p_\Xi(0,T)$ is surjective and $\partial^q Q^q_p = Q^{q-q}_{p-q}\partial^q=P_{p-q}\partial^q$, as a consequence of Theorem \ref{sande,manni,speleers} there hold the following relations
\begin{equation*}
\begin{split}
\big\|\partial^q\big(u-Q^q_pu \big) \big \|_{L^2(0,T)}&=\big\| v-P_{p-q}v \big \|_{L^2(0,T)} \\
&\overset{(\text{if} \ p-q \geq r-q-1 \geq 0)}\leq \Big(\frac{h}{\pi}\Big)^{r-q} \big \| \partial^{r-q}v \big \|_{L^2(0,T)}=\big \|\partial^r u \big \|_{L^2(0,T)}, 
\end{split}
\end{equation*}
$q \leq r-1$ and $p \geq r-1$. Since $q = 1, \ldots, r-1$, the last inequality holds for $r \geq 2$.
\end{proof}

\begin{oss}
	The density of the family of spline spaces $(V^h_{0,*})_h$ in $H^1_{0,*}(0,T)$ is a consequence of the result \eqref{err spline} with $p=r=2$, $q=1$, and of the density of $C^\infty_c(0,T]$ in $H^1_{0,*}(0,T)$ observed in Section \ref{sec sobolev}. 
	\begin{proof}
		Let $u \in H^1_{0,*}(0,T)$ and let $\epsilon >0$ be fixed. As a consequence of the density of $C^\infty_c(0,T]$ in $H^1_{0,*}(0,T)$, there exists $\phi \in C^\infty_c(0,T]$ such that\\ $|u-\phi|_{H^1(0,T)} \leq \frac{\epsilon}{2}$. As a consequence of \eqref{err spline}, there exists $\overline{h}>0$ such that \\ $\big |\phi-Q^1_2\phi \big |_{H^1(0,T)} \leq \frac{\epsilon}{2}$ for all $h \leq \overline{h}$. We then obtain:
		\begin{equation*}
		\begin{split}
		\inf_{v_h \in V^h_{0,*}}|u-v_h|_{H^1(0,T)} &\leq |u-Q^1_2\phi \big |_{H^1(0,T)} \\
		&\leq |u-\phi|_{H^1(0,T)} + |\phi-Q^1_2\phi \big |_{H^1(0,T)} \leq \epsilon \quad \forall h \leq \overline{h}.
		\end{split}
		\end{equation*}
	\end{proof}
\end{oss}

Let us recall that we have modified the projection operator of (\cite{n-width}) in order to get a projection operator whose restriction to $H^1_{0,*}(0,T) \cap H^q(0,T)$ ($q \geq 1$) assumes values in $V^h_{0,*}$ defined in \eqref{iga space}.

\subsection{Bound on the mesh-size, stability and quasi-optimality constants}
In this Section we give two results of conditioned stability \textit{with an explicit bound on the mesh-size} and we also make explicit the stability and quasi-optimality constants.
\subsubsection{Extension to quadratic IGA with maximal regularity of conditioned stability for piecewise continuous linear FEM}
A first result is an extension to the IGA discretization of Theorem 4.7 of (\cite{Coercive}).

\begin{theorem}\label{teo stab IGA zank}
	Let 
	\begin{equation}\label{h bound}
	h \leq \frac{\pi^2}{\sqrt{2}(2+\sqrt{\mu}T)\mu T}
	\end{equation}
	be satisfied. Then, the bilinear form $a(\cdot,\cdot)$ as defined in \eqref{bil ode} satisfies the inf-sup condition
	\begin{equation}\label{infsup zank}
	\frac{2 \pi^2}{(2+\sqrt{\mu}T)^2(\pi^2+4\mu T^2)} |u_h|_{H^1(0,T)} \leq \sup_{0 \neq v_h \in V^h_{0,*}} \frac{a(u_h,\overline{\mathcal{H}}_T v_h)}{|v_h|_{H^1(0,T)}},
	\end{equation}
	for all $u_h \in V^h_{0,*}$.
\end{theorem}

\begin{proof}
	Let $u_h \in V^h_{0,*}$. As a consequence of Lax-Milgram Lemma \ref{lax-milgram}, let $w \in H^1_{0,*}(0,T)$ be the unique solution of the variational problem
	\begin{equation}\label{1 var}
	-\int_0^T \partial_t w(t) \partial_t(\overline{\mathcal{H}}_T v)(t) \ dt = -\mu \int_0^T u_h(t)(\overline{\mathcal{H}}_T v)(t) \ dt  \quad \forall v \in H^1_{0,*}(0,T).
	\end{equation}
	Hence, recalling that $(\overline{\mathcal{H}}_T v)(t)=v(T)-v(t)$, 
	\begin{equation}\label{auh-w}
	\begin{split}
		a(u_h,\overline{\mathcal{H}}_T(u_h-w))&=-\int_0^T \partial_t u_h(t) \partial_t[(\overline{\mathcal{H}}_T u_h)(t)-(\overline{\mathcal{H}}_T w)(t)] \ dt \\
		 & \quad + \mu \int_0^T u_h(t) [(\overline{\mathcal{H}}_T u_h)(t)-(\overline{\mathcal{H}}_T w)(t)] \ dt\\
	&\overset{\eqref{1 var}}= \int_0^T \partial_t u_h(t) \partial_t[u_h(t)- w(t)] \ dt - \int_0^T \partial_t w(t) [\partial_t u_h(t)-\partial_t w(t)] \ dt\\
	&= \int_0^T [\partial_t u_h(t)-\partial_t w(t)]^2 \ dt=|u_h-w|_{H^1(0,T)}^2.
	\end{split}
	\end{equation}
	Also, thanks to Lax-Milgram Lemma \ref{lax-milgram}, let $z \in H^1_{0,*}(0,T)$ be the unique solution of the following variational problem
	\begin{equation}\label{2 var}
	\begin{split}
	-\int_0^T \partial_t z(t) \partial_t (\overline{\mathcal{H}}_T v)(t) \ dt = -\int_0^T &\partial_t u_h(t) \partial_t(\overline{\mathcal{H}}_T v)(t) \ dt \\
	& + \mu \int_0^T u_h(t)(\overline{\mathcal{H}}_T v)(t)\ dt \quad \forall v \in H^1_{0,*}(0,T)
	\end{split}
	\end{equation}
	Since $w \in H^1_{0,*}(0,T)$ is the solution of \eqref{1 var}, problem \eqref{2 var} is equivalent to
	\begin{equation*}
	-\int_0^T \partial_t [z(t)-(u_h(t)-w(t))] \partial_t(\overline{\mathcal{H}}_T v)(t) \ dt =0 \quad \forall v \in H^1_{0,*}(0,T),
	\end{equation*}
	from which we conclude, by choosing $v=z-(u_h-w) \in H^1_{0,*}(0,T)$, that
	$z(t)=u_h(t)-w(t)$ for all $t \in [0,T]$, since $u_h(0)=w(0)=z(0)=0$. Therefore, from \eqref{auh-w}, we conclude that
	\begin{equation}\label{a=z}
	a(u_h,\overline{\mathcal{H}}_T(u_h-w))=|z|^2_{H^1(0,T)}.
	\end{equation}
	On the other hand, the variational formulation \eqref{2 var} gives
	\begin{equation*}
	\begin{split}
	|z|_{H^1(0,T)} &\overset{\eqref{2 var}}= \frac{a(u_h,\overline{\mathcal{H}}_T z)}{|z|_{H^1(0,T)} } \leq \sup_{0 \neq v \in H^1_{0,*}(0,T)} \frac{a(u_h,\overline{\mathcal{H}}_T v)}{|v|_{H^1(0,T)}} \\
	&\overset{\eqref{2 var}} = \sup_{0 \neq v \in H^1_{0,*}(0,T)} \frac{ \langle \partial_tz,\partial_tv \rangle _{L^2(0,T)}}{|v|_{H^1(0,T)}} \overset{(C-S)}\leq |z|_{H^1(0,T)}.
	\end{split}
	\end{equation*}
	Hence, by using \eqref{stab ode}, there hold the following
	\begin{equation*}
	|z|_{H^1(0,T)}= \sup_{0 \neq v \in H^1_{0,*}(0,T)} \frac{a(u_h,\overline{\mathcal{H}}_T v)}{|v|_{H^1(0,T)}} \overset{\eqref{stab ode}}\geq \frac{2}{2+\sqrt{\mu}T} |u_h|_{H^1(0,T)},
	\end{equation*}
	from which, recalling \eqref{a=z}, we conclude
	\begin{equation}\label{magg a}
	a(u_h,\overline{\mathcal{H}}_T(u_h-w)) \overset{\eqref{a=z}}= |z|^2_{H^1(0,T)} \geq \frac{4}{(2+\sqrt{\mu}T)^2}|u_h|^2_{H^1(0,T)}.
	\end{equation}
	We now discretize the first variational formulation \eqref{1 var}. Let $w_h \in V^h_{0,*}$ be the unique solution of the following variational problem
	\begin{equation}\label{disc dim}
	\int_0^T \partial_t w_h(t) \partial_t v_h(t) dt = -\mu \int_0^T u_h(t) (\overline{\mathcal{H}}_T v_h)(t) \ dt \quad \forall v_h \in V^h_{0,*}.
	\end{equation}
	By using Céa's Lemma \ref{céa} and the error estimate \eqref{err spline} with $r=2$, $p=2$, $q=1$, there hold the following relations
	\begin{equation}\label{w-wh}
	\begin{split}
	|w-&w_h|_{H^1(0,T)}\leq \inf_{0 \neq v_h \in V^h_{0,*}}|w-v_h|_{H^1(0,T)} \quad \text{(Céa)}\\
	& \leq |w-Q^1_2 w|_{H^1(0,T)} \overset{\eqref{err spline}}\leq \frac{h}{\pi}\|\partial_{tt}w\|_{L^2(0,T)} \overset{\eqref{disc dim}}= \mu \frac{h}{\pi}\|u_h\|_{L^2(0,T)}.
	\end{split}
	\end{equation}
	Furthermore, Galerkin orthogonality \eqref{gal ort} is satisfied
	\begin{equation}\label{gal ort dim}
	\int_0^T[\partial_t w(t)-\partial_t w_h(t)] \partial_t v_h(t) \ dt =0 \quad \forall v_h \in V^h_{0,*}.
	\end{equation}
	As a consequence of \eqref{gal ort dim}, we have
	\begin{equation}\label{magg uh}
	\begin{split}
	a(u_h,\overline{\mathcal{H}}_T(w-w_h))&=\int_0^T \partial_t u_h(t)\partial_t[w(t)-w_h(t)] \ dt + \mu \int_0^T u_h(t)[\overline{\mathcal{H}}_T(w-w_h)](t) \ dt \\
	&\overset{\eqref{gal ort dim}}=\mu \int_0^T u_h(t)[\overline{\mathcal{H}}_T(w-w_h)](t) \ dt \\  
	& \leq \mu \|u_h\|_{L^2(0,T)} \|\overline{\mathcal{H}}_T(w-w_h)\|_{L^2(0,T)}.
	\end{split}
	\end{equation}
	As a consequence of Lax-Milgram Lemma \ref{lax-milgram}, let now $\psi \in H^1_{0,*}(0,T)$ be the unique solution of the following variational problem
	\begin{equation}\label{3 var}
	- \int_0^T \partial_t \psi(t) \partial_t [\overline{\mathcal{H}}_Tv](t) \ dt = \int_0^T(\overline{\mathcal{H}}_T(w-w_h))(t)(\overline{\mathcal{H}}_Tv)(t) \ dt \quad \forall v \in H^1_{0,*}(0,T).
	\end{equation}
	In particular, by choosing $v=w-w_h \in H^1_{0,*}(0,T)$ and recalling \eqref{gal ort dim}, we obtain
	\begin{equation*}
	\begin{split}
	\|\overline{\mathcal{H}}_T(w-w_h)\|^2_{L^2(0,T)}&=\int_0^T[\overline{\mathcal{H}}_T(w-w_h)](t)[\overline{\mathcal{H}}_T(w-w_h)](t) \ dt \\
	&\overset{\eqref{3 var}}=-\int_0^T\partial_t \psi(t)\partial_t [\overline{\mathcal{H}}_T(w-w_h)](t) \ dt \\
	&=\int_0^T\partial_t \psi(t)[\partial_t w(t)-\partial_t w_h(t)] \ dt \\
	&\overset{\eqref{gal ort dim}}=\int_0^T \partial_t[\psi(t)-Q^1_2 \psi(t)][\partial_t w(t)-\partial_t w_h(t)] \ dt\\
	& \leq |\psi-Q^1_2 \psi|_{H^1(0,T)}|w-w_h|_{H^1(0,T)}\\
	&\overset{\eqref{err spline},\eqref{w-wh}}\leq \Big(\frac{h}{\pi} \Big)^2\mu\|\partial_{tt} \psi \|_{L^2(0,T)}\|u_h\|_{L^2(0,T)}\\
	&\overset{\eqref{3 var}}= \Big(\frac{h}{\pi} \Big)^2\mu \|\overline{\mathcal{H}}_T(w-w_h)\|_{L^2(0,T)} \|u_h\|_{L^2(0,T)},
	\end{split}
	\end{equation*}
	i.e.,
	\begin{equation*}
	\|\overline{\mathcal{H}}_T(w-w_h)\|_{L^2(0,T)}\leq \frac{h^2}{\pi^2}\mu\|u_h\|_{L^2(0,T)}.
	\end{equation*}
	Therefore, by using \eqref{magg uh} and Poincaré inequality \eqref{Poinc},
	\begin{equation}\label{stim}
     a(u_h,\overline{\mathcal{H}}_T(w-w_h)) \overset{\eqref{magg uh}}\leq \Big(\frac{h}{\pi} \mu \Big)^2  \|u_h\|^2_{L^2(0,T)} \overset{\eqref{Poinc}}\leq \Big(\frac{2 T}{\pi^2} \mu h \Big)^2|u_h|_{H^1(0,T)}^2
  	\end{equation}
  	follows. Hence, as a consequence of \eqref{magg a} and \eqref{stim} we conclude
  	\begin{equation*}
  	\begin{split}
  	a(u_h,\overline{\mathcal{H}}_T(u_h-w_h))&=a(u_h,\overline{\mathcal{H}}_T(u_h-w))+a(u_h,\overline{\mathcal{H}}_T(w-w_h))\\
  	& \overset{\eqref{magg a},\eqref{stim}}\geq \Bigg[ \frac{4}{(2+\sqrt{\mu}T)^2} - \Big(\frac{2 T}{\pi^2} \mu h \Big)^2 \Bigg] |u_h|^2_{H^1(0,T)}\\
  	& \geq \frac{2}{(2+\sqrt{\mu}T)^2} |u_h|^2_{H^1(0,T)},
  	\end{split}
  	\end{equation*}
  	if
  	\begin{equation*}
  	\Big(\frac{2 T}{\pi^2} \mu h \Big)^2 \leq \frac{2}{(2+\sqrt{\mu}T)^2}
  	\end{equation*}
  	is satisfied, i.e.,
  	\begin{equation*}
  	h \leq \frac{\pi^2}{\sqrt{2}(2+\sqrt{\mu}T)\mu T}.
  	\end{equation*}
  	We now consider a lower bound for $|u_h|_{H^1(0,T)}$ dependent of $|u_h-w_h|_{H^1(0,T)}$. Indeed, we have
  	\begin{equation*}
  	\|\partial_t(u_h-w_h)\|_{L^2(0,T)}\leq \|\partial_t u_h\|_{L^2(0,T)}+\|\partial_t w_h\|_{L^2(0,T)},
  	\end{equation*}
  	and, thanks to \eqref{disc dim} and Poincaré inequality \eqref{Poinc},
  	\begin{equation*}
  	\begin{split}
  	\|\partial_t w_h\|^2_{L^2(0,T)}&=-\int_0^T \partial_t w_h(t) \partial_t (\overline{\mathcal{H}}_T w_h)(t) \ dt \overset{\eqref{disc dim}}= -\mu \int_0^T u_h(t)(\overline{\mathcal{H}}_T w_h)(t) \ dt \\
  	&  \leq \mu \|u_h\|_{L^2(0,T)} \|\overline{\mathcal{H}}_T w_h \|_{L^2(0,T)} \overset{\eqref{Poinc}} \leq \frac{4 T^2}{\pi^2}\mu |u_h|_{H^1(0,T)}|w_h|_{H^1(0,T)},
  	\end{split}
  	\end{equation*}
  	i.e.,
  	\begin{equation*}
  	|u_h-w_h|_{H^1(0,T)} \leq \Big(1+\frac{4 T^2}{\pi^2}\mu \Big) |u_h|_{H^1(0,T)}.
  	\end{equation*}
  	Therefore, we obtain the following inequality
  	\begin{equation*}
  	\frac{2 \pi^2}{(2+\sqrt{\mu}T)^2(\pi^2+4\mu T^2)} |u_h|_{H^1(0,T)} \leq \frac{a(u_h,\overline{\mathcal{H}}_T(u_h-w_h))}{|u_h-w_h|_{H^1(0,T)}}.
  	\end{equation*}
\end{proof}

Hence, we obtain the following result.

\begin{theorem}
	Let \eqref{h bound} be satisfied. Then, problem \eqref{iga ode equiv} is well-posed with the stability estimate
	\begin{equation}\label{stab zank}
	|u_h|_{H^1(0,T)} \leq \frac{(2+\sqrt{\mu}T)^2(\pi^2+4\mu T^2)}{2\pi^2}\|f\|_{[H^1_{*,0}(0,T)]'},
	\end{equation}
	where $f \in [H^1_{*,0}(0,T)]'$ and $u_h$ is the unique solution of \eqref{iga ode equiv}. Moreover, a quasi-optimality estimate holds
	\begin{equation}\label{quasi opt zank}
	|u-u_h|_{H^1(0,T)} \leq \Bigg[ 1 + \frac{(2+\sqrt{\mu} T)^2(\pi^2+4 \mu T^2)^2}{2 \pi^4} \Bigg] \inf_{v_h \in V^h_{0,*}} |u-v_h|_{H^1(0,T)},
	\end{equation}
	where $u$ is the unique solution of \eqref{var ode equiv}.
\end{theorem}

\begin{proof}
	The well-posedness with stability estimate \eqref{stab zank} is an immediate consequence of BNB Theorem \ref{BNB}.
	
	In order to prove \eqref{quasi opt zank}, we repeat, with explicit constants, the proof of Proposition \ref{err BNB}. For any $w \in H^1_{0,*}(0,T)$ we define $w_h:= G_h w$ as the Galerkin projection satisfying 
	\begin{equation*}
	a(G_hw,\overline{\mathcal{H}}_T v_h)=a(w,\overline{\mathcal{H}}_T v_h) \quad \forall v_h \in V^h_{0,*},
	\end{equation*}
	which is well defined thanks to the well-posedness of the discrete problem. Hence, by using the stability estimate \eqref{infsup zank} and the continuity of $a(\cdot,\cdot)$ \eqref{cont of a}, there hold
	\begin{equation*}
	|G_h w|_{H^1(0,T)} \leq \frac{(2+\sqrt{\mu}T)^2(\pi^2+4 \mu T^2)^2}{2 \pi^4} |w|_{H^1(0,T)}.
	\end{equation*}
	Indeed,
	\begin{equation*}
	\begin{split}
		\frac{2 \pi^2}{(2+\sqrt{\mu}T)^2(\pi^2+4\mu T^2)} &|G_hw|_{H^1(0,T)}  
		\overset{\eqref{infsup zank}}\leq \sup_{0 \neq v_h \in V^h_{0,*}} \frac{a(G_hw,\overline{\mathcal{H}}_T v_h)}{|v_h|_{H^1(0,T)}}\\
		&=\sup_{0 \neq v_h \in V^h_{0,*}} \frac{a(w,\overline{\mathcal{H}}_T v_h)}{|v_h|_{H^1(0,T)}}\\
		& \overset{\eqref{cont of a}}\leq  \Bigg(1+\frac{4 T^2 \mu }{\pi^2} \Bigg) |w|_{H^1(0,T)}.
	\end{split}
	\end{equation*}
	Since $u_h=G_h u$ and $v_h = G_h v_h$ for all $v_h \in V^h_{0,*}$, we conclude
	\begin{equation*}
	\begin{split}
	|u-u_h|_{H^1(0,T)} &\leq |u-v_h|_{H^1(0,T)}+|G_h(v_h-u)|_{H^1(0,T)} \\
	& \leq \Bigg[ 1 + \frac{(2+\sqrt{\mu} T)^2(\pi^2+4 \mu T^2)^2}{2 \pi^4} \Bigg] |u-v_h|_{H^1(0,T)}.
	\end{split}
	\end{equation*}
\end{proof}

Thus, we are in a position to state a convergence result for the isogeometric solution $u_h$ of the variational formulation \eqref{var ode equiv}.

\begin{cor}
	Let $u \in H^1_{0,*}(0,T)$ and $u_h \in V^h_{0,*}$ be the unique solutions of the variational formulations \eqref{var ode equiv} and \eqref{iga ode equiv}, respectively. Let  $u \in H^3(0,T)$ and \eqref{h bound} be satisfied. Then, there holds true the error estimate
	\begin{equation}\label{ord conv iga zank}
	|u-u_h|_{H^1(0,T)} \leq \frac{1}{\pi^2} \Bigg[ 1 + \frac{(2+\sqrt{\mu} T)^2(\pi^2+4 \mu T^2)^2}{2 \pi^4} \Bigg] h^2 |u|_{H^3(0,T)}.
	\end{equation}
\end{cor}

\begin{proof}
	Estimate \eqref{ord conv iga zank} is a straightforward consequence of quasi-optimality \eqref{quasi opt zank} and of \eqref{err spline} with $r=3,p=2,q=1$.
\end{proof}

\begin{oss}
	Note that the bound on the mesh-size \eqref{h bound} is about $2.015$ times the bound 
	\begin{equation*}
	h \leq \frac{2 \sqrt{3}}{(2+\sqrt{\mu}T)\mu T}
	\end{equation*}
	of (\cite{Steinbach2019}). It's also about $1.81$ times the more accurate bound 
	\begin{equation*}
	h \leq \frac{\sqrt{3}\pi}{\sqrt{2}(2+\mu T)\mu T}
	\end{equation*}	
	of (\cite{Zank2020}). Therefore, in order to get stability and convergence, our mesh-size can be approximately twice the mesh-size of piece-wise linear FEM.
\end{oss}

\subsubsection{Application of Theorem \ref{theo gard} to quadratic IGA with maximal regularity}
Another way to get a bound on the mesh-size, so that, if it is respected, the well-posedness of IGA, stability and convergence (with explicit constants) are guaranteed, is the theory of \textit{Galerkin method applied to G\aa rding-type problems} discussed in Section \ref{sec prel ode}. 

Let now consider problem \eqref{var ode equiv} and its isogeometric discretization \eqref{iga ode equiv}. Actually, these problems lie within the framework of Theorem \ref{theo gard}. Indeed, as a consequence of Rellich-Kondrachov Theorem, the inclusion $H^1_{0,*}(0,T) \subset L^2(0,T)$ is compact. Furthermore, the following result holds.
\begin{lemma}
	Let $b > \frac{\mu T^2}{2}$. The bilinear form defined in \eqref{bil ode} satisfies the G\aa rding inequality:
	\begin{equation}\label{gard ineq iga}
	a(v,\overline{\mathcal{H}}_T v) \geq \Big(1- \frac{\mu T^2}{2b}\Big)|v|_{H^1(0,T)}-\frac{2+b}{2}\mu \|v\|_{L^2(0,T)} \quad \forall v \in H^1_{0,*}(0,T).
	\end{equation}
\end{lemma}
\begin{proof}
	The bilinear form defined in \eqref{bil ode} satisfies the following inequalities
	\begin{equation*}
	\begin{split}
	a(v,\overline{\mathcal{H}}_T v)&= \langle \partial_t v, \partial_t v \rangle _{L^2(0,T)}- \mu \langle v,v \rangle_{L^2(0,T)}+\mu \langle v,v(T) \rangle_{L^2(0,T)} \\
	& \geq |v|^2_{H^1(0,T)}-\mu \|v\|^2_{L^2(0,T)}-\mu \int_0^T |v(t)v(T)| \ dt \\
	& \overset{(C-S)}\geq |v|^2_{H^1(0,T)}-\mu \|v\|^2_{L^2(0,T)} - \mu \|v\|_{L^2(0,T)}\sqrt{T}|v(T)| \\
	& \overset{(F.T.C),(Young)}\geq |v|^2_{H^1(0,T)}-\mu \|v\|^2_{L^2(0,T)} - \frac{\mu b}{2} \|v\|^2_{L^2(0,T)}-\frac{\mu T^2}{2b} |v|^2_{H^1(0,T)}\\
	& = \Big(1- \frac{\mu T^2}{2b}\Big)|v|^2_{H^1(0,T)}-\frac{2+b}{2}\mu \|v\|^2_{L^2(0,T)} ,
	\end{split}
	\end{equation*}
	for all $b >0$, where we used Cauchy Schwarz inequality, Young inequality and the Fundamental Theorem of Calculus. In order to obtain positive coefficients in \eqref{gard ineq iga} we add the constraint $b > \frac{\mu T^2}{2}$.
\end{proof}	
	
Also, Theorem \ref{zank} guarantees that problem \eqref{var ode equiv} is well-posed, then, the only $u_0 \in H^1_{0,*}(0,T)$ such that $a(u_0,\overline{\mathcal{H}}_T v)=0$ for all $v \in H^1_{0,*}(0,T)$ is $u_0=0$. Therefore, we conclude the following result.

\begin{theorem}\label{theo gard iga}
	Let 
	\begin{equation}\label{h bound gard}
	h \leq \frac{\pi^5}{(\pi^2+ 4\mu T^2)[\pi^2+2\mu T^2 (2+\sqrt{\mu} T)]} \sqrt{\frac{2b - \mu T^2}{2b(2+b)\mu}}
	\end{equation}
	be satisfied, with $b > \frac{\mu T^2}{2}$. 
	Then, problem \eqref{iga ode equiv} is well-posed with the stability estimate
	\begin{equation}\label{stab iga gard}
	\|u_h\|_{H^1(0,T)} \leq (2+ \sqrt{\mu} T ) \Bigg[ \frac{3b + \mu T^2\big(\frac{8b}{\pi^2}-\frac{1}{2} \big)}{2b-\mu T^2} \Bigg] \|f\|_{[H^1_{*,0}(0,T)]'},
	\end{equation}
	where $f \in [H^1_{*,0}(0,T)]' $ and $u_h$ is the unique solution of \eqref{iga ode equiv}. Moreover, a quasi-optimality estimate holds
	\begin{equation}\label{opt iga gard}
	|u-u_h|_{H^1(0,T)} \leq \frac{4b}{\pi^2}\frac{ \pi^2+ 4 \mu T^2}{2b-\mu T^2} \inf_{v_h \in V^h_{0,*}} |u-v_h|_{H^1(0,T)},
	\end{equation}
	where $u \in H^1_{0,*}(0,T)$ is the unique solution of \eqref{var ode equiv}.
\end{theorem}

\begin{proof}
	Let $b > \frac{\mu T^2}{2}$ and let $\eta(V^h_{0,*})$ be the parameter defined in \eqref{eta gard} and related to the sequence of isogeometric spaces in \eqref{iga space}. By using the projection operator \eqref{Qpq} with $p=2$, $q=1$, the error estimate \eqref{err spline} with $r=2$, the Poincaré inequality \eqref{Poinc} and the a priori estimate of the abstract problem \eqref{stab ode}, we obtain
	\begin{equation}\label{eta iga}
	\eta(V^h_{0,*}) \leq \frac{h}{\pi} \Bigg[ 1 + \frac{2\mu T^2}{\pi^2}(2+\sqrt{\mu} T) \Bigg].
	\end{equation}
    Let see why \eqref{eta iga} holds. Let $z_g \in H^1_{0,*}(0,T)$ be the solution of the adjoint problem \eqref{adj} with respect to our context, i.e., given $g \in L^2(0,T)$, $z_g$ is the unique element of $H^1_{0,*}(0,T)$ that satisfies
	\begin{equation*}
	a(v,\overline{\mathcal{H}}_T z_g)=(g,v)_{L^2(0,T)} \quad \forall v \in H^1_{0,*}(0,T).
	\end{equation*}
	Therefore, $z:=\overline{\mathcal{H}}_Tz_g \in H^1_{*,0}(0,T)$ is the unique solution of
	\begin{equation*}
	a(v,z)=(g,v)_{L^2(0,T)} \quad \forall v \in H^1_{0,*}(0,T),
	\end{equation*}
	i.e.,
	\begin{equation*}
	-\langle \partial_t v,\partial_t z \rangle_{L^2(0,T)} + \mu \langle v,z \rangle_{L^2(0,T)} = (g,v)_{L^2(0,T)} \quad \forall v \in H^1_{0,*}(0,T).
	\end{equation*}
	As a consequence, the distributional derivative $\partial_{tt} z$ is represented by $g-\mu z \in L^2(0,T)$. Hence, $z \in H^2(0,T)$ and $z_g = \overline{\mathcal{H}}_T^{-1} z = z(0)-z \in H^2(0,T)$ with 
	\begin{equation*}
	\partial_{tt} z_g = \mu \overline{\mathcal{H}}_T z_g - g.
	\end{equation*}
	Note that the adjoint problem has the same stability estimate (w.r.t. the dual norm $\|g\|_{[H^1_{0,*}(0,T)]'}$) of the primal problem, as noted in the second point of Remark \ref{oss gard}. Then, the following relations hold
	\begin{equation*}
	\begin{split}
\inf_{v_h \in V^h_{0,*}} |z_g-v_h|_{H^1(0,T)} &\leq |z_g - Q^1_2 z_g|_{H^1(0,T)} \overset{\eqref{err spline}}\leq \frac{h}{\pi} \|\partial_{tt} z_g\|_{L^2(0,T)}\\
&=\frac{h}{\pi}\|\mu \overline{\mathcal{H}}_T z_g-g\|_{L^2(0,T)}\\
& \leq \frac{h}{\pi} \Big(\|g\|_{L^2(0,T)} + \mu \|\overline{\mathcal{H}}_T z_g\|_{L^2(0,T)} \Big)\\
&\overset{\eqref{Poinc}}\leq \frac{h}{\pi} \Big(\|g\|_{L^2(0,T)} + \frac{2T}{\pi}\mu |z_g|_{H^1(0,T)} \Big)\\
&\overset{\eqref{stab ode}}\leq \frac{h}{\pi} \Big(\|g\|_{L^2(0,T)} + \frac{4T^2}{\pi^2}\mu \frac{2 + \sqrt{\mu} T}{2} \|g\|_{L^2(0,T)} \Big)\\
&=\frac{h}{\pi} \Big[ 1 + \frac{2\mu T^2}{\pi^2}(2+\sqrt{\mu}T) \Big] \|g\|_{L^2(0,T)},
	\end{split}
	\end{equation*}
	which implies estimate \eqref{eta iga}. Therefore, from condition \eqref{cond eta} where \eqref{gard ineq iga} is the G\aa rding inequality of our problem \eqref{iga ode equiv}, we deduce that, if 
	\begin{equation*}
	\frac{h}{\pi} \Bigg[ 1 + \frac{2\mu T^2}{\pi^2}(2+\sqrt{\mu}T) \Bigg] \leq \frac{\pi^2}{\pi^2+4\mu T^2} \sqrt{\frac{2b - \mu T^2}{2b(2+b)\mu}},
	\end{equation*}
i.e.,
\begin{equation*}
h \leq \frac{\pi^5}{(\pi^2+4 \mu T^2)[\pi^2+2\mu T^2 (2+\sqrt{\mu} T)]} \sqrt{\frac{2b - \mu T^2}{2b(2+b)\mu}},
\end{equation*}	
then problem \eqref{iga ode equiv} is well-posed, and conditions \eqref{stab iga gard}, \eqref{opt iga gard} are satisfied. 
\end{proof}

\begin{oss}
	The choice
	\begin{equation}\label{b}
	b=\frac{\mu T^2 + \sqrt{\mu^2 T^4+ 4 \mu T^2}}{2}
	\end{equation}
	maximises $\sqrt{\frac{2b - \mu T^2}{2b(2+b)\mu}}$ in the upper bound of \eqref{h bound gard}.
\end{oss}

As a consequence of Theorem \ref{theo gard iga}, we can state a convergence result for the isogeometric solution $u_h$ of the variational formulation \eqref{var ode equiv}.

\begin{cor}\label{cor gard}
Let $u \in H^1_{0,*}(0,T)$ and $u_h \in V^h_{0,*}$ be the unique solutions of the variational formulations \eqref{var ode equiv} and \eqref{iga ode equiv}, respectively. Let $u \in H^3(0,T)$ and \eqref{h bound gard} be satisfied. Then, there holds true the error estimate
\begin{equation}\label{ord conv iga gard}
|u-u_h|_{H^1(0,T)} \leq \frac{4b}{\pi^4}\frac{ \pi^2+ 4 \mu T^2}{2b-\mu T^2} h^2 |u|_{H^3(0,T)}.
\end{equation}
\end{cor}

\begin{proof}
	Estimate \eqref{ord conv iga gard} is a straightforward consequence of quasi-optimality \eqref{opt iga gard} and of \eqref{err spline} with $r=3,p=2,q=1$.
\end{proof}

\begin{oss}\label{oss grado più alto}
The analysis proposed in this Section, which is based on techniques that exploit the G\aa rding inequality and are alternative to those proposed by O. Steinbach and M. Zank (extended to the IGA case in the previous Section) is useful to understand how the IGA behaves in the case of generic polynomial degree and maximal regularity. Indeed, from the proof of Theorem \ref{theo gard iga} it emerges that the threshold on $h$ \eqref{h bound gard} does not change when the polynomial degree and the regularity of the spline test and trial functions are raised. Only the order of convergence of Corollary \ref{h bound gard} changes: for a generic polynomial degree $p \geq 1$ and an exact solution $u \in H^{p+1}(0,T)$, the order of convergence is $p$. 
On the other hand, how the upper bound \eqref{h bound} behaves in $h$ is not immediately clear from the proof of Theorem \ref{teo stab IGA zank}. This is a consequence of the various auxiliary problems considered, in which the regularity of the corresponding solutions is a key point. 
\end{oss}	

\begin{oss}\label{oss asympt}
	Asymptotically for $\mu \rightarrow \infty$ the optimal value \eqref{b} satisfies $b\simeq\mu T^2$. We then obtain
	\begin{equation*}
	h \leq \overline{h}_2 :=\frac{\pi^5}{(\pi^2+ 4\mu T^2)[\pi^2+2\mu T^2 (2+\sqrt{\mu} T)]} \sqrt{\frac{1}{2\mu(2+\mu T^2)}},
	\end{equation*}
	in place of \eqref{h bound gard} with a general value $b > \frac{\mu T^2}{2}$. 
	
	Let us now recall estimate \eqref{h bound} obtained by extending Theorem $4.7$ of (\cite{Coercive}) to quadratic IGA with maximal regularity, i.e.,
	\begin{equation*}
	h \leq \overline{h}_1 :=\frac{\pi^2}{\sqrt{2}(2+\sqrt{\mu}T)\mu T}.
	\end{equation*}
	Therefore, the techniques of O. Steinbach and M. Zank and those ones using the G\aa rding inequality \eqref{gard} give constraints on $h$ of order
	\begin{gather*}
	\overline{h}_1 \simeq \mathcal{O}(\mu^{-3/2}),\\
	\overline{h}_2 \simeq \mathcal{O}(\mu^{-7/2}),
	\end{gather*}
respectively. Thus, we conclude that, asymptotically, threshold \eqref{h bound gard} is a stronger constraint than \eqref{h bound}. However, asymptotically, stability estimate \eqref{stab iga gard} behaves as $\mathcal{O}(\mu^{3/2})$, whereas, stability estimate \eqref{stab zank} behaves as $\mathcal{O}(\mu^{2})$. Therefore, we conclude that stability estimate \eqref{stab iga gard} has a slower growth than \eqref{stab zank} for $\mu \rightarrow \infty$. 
\end{oss}
	

%% file: Chapters/Chapter4.tex

\chapter{Numerical methods for $\partial_{tt}u+ \mu u=f$} 

\label{Chapter4} 

In this Chapter we numerically study our model problem \eqref{eq ode}. 

All the simulations are performed in MATLAB on a Intel(R) Core(TM) i3-4005U CPU @ 1.70GHz 1.70 GHz laptop, with 4,00 GB RAM.

All the isogeometric discretizations are performed using GeoPDEs, which is an open source and free package for the research and teaching of Isogeometric Analysis, written in Octave and fully compatible with MATLAB. See (\cite{GeoPDEs}) for a complete explanation of its design and its main features.




\section{Conditioned stability}

\subsection{Errors committed by piecewise continuous linear finite element method}\label{sec errors fem}

In (\cite{Steinbach2019, Zank2020, Coercive}) the authors study the conditioned stability of the piecewise continuous linear finite element discretization of \eqref{eq ode} and introduce a stabilized method (\cite{Steinbach2019, Zank2020}), which they then extend to the wave equation \eqref{eqonde} (\cite{Steinbach2019, Zank2020}). Indeed, as noticed in (\cite{Steinbach2019, Zank2020}), the stability of a conforming (w.r.t. classic anisotropic Sobolev spaces) tensor-product space-time discretization with pie\-ce\-wise linear, continuous solution and test functions of \eqref{eqonde}, requires a Cou\-rant – Friedrichs – Lewy (CFL) condition, i.e., 
\begin{equation}\label{CFL 2}
h_t \leq C h_x,
\end{equation}
with a constant $C > 0$, depending on the constant of a spatial inverse inequality, where $h_t$ and $h_x$ are the uniform mesh-sizes in time and space. In particular, constraint \eqref{CFL 2} follows from the conditioned stability of the piecewise continuous linear FEM applied to \eqref{eq ode} with uniform mesh-size. 

In this Section we briefly recall O. Steinbach and M. Zank's main results for the conditioned stability of FEM discretization of the ODE \eqref{eq ode} and we show some numerical results that we obtain by testing their theoretical considerations. 

Let us define the discrete spaces
\begin{gather}
S^1_{h;0,*}:=\{v_h \in S^1_h(0,T)| \ v_h(0)=0\}=S^1_h(0,T) \cap H^1_{0,*}(0,T), \label{fem trial}\\
S^1_{h;*,0}:=\{v_h \in S^1_h(0,T)| \ v_h(T)=0\}=S^1_h(0,T) \cap H^1_{*,0}(0,T), \label{fem test}
\end{gather}
where $S^1_h(0,T)$ is the classic space of piecewise continuous linear functions on $[0,T]$ with maximal mesh-size $h$.

In Theorem 4.7 of (\cite{Coercive}), which we extend to the isogeometric case with Theorem \ref{teo stab IGA zank}, the authors prove that if the mesh-size $h$ satisfies 
\begin{equation}\label{zank 1 stima}
h \leq \frac{2\sqrt{3}}{(2+\sqrt{\mu}T)\mu T},
\end{equation}
 then the discrete Galerkin-Bubnov formulation of \eqref{eq ode} is well-posed with a uniform (w.r.t $h$) lower bound on the discrete inf-sup, i.e.,
 \begin{equation}\label{infsup fem}
 \frac{8}{(2+\sqrt{\mu}T)^2(4+\mu T^2)} |u_h|_{H^1(0,T)} \leq \sup_{0 \neq v_h \in S^1_{h;0,*}} \frac{a(u_h,\overline{\mathcal{H}}_T v_h)}{|v_h|_{H^1(0,T)}},
 \end{equation}
 for all $u_h \in S^1_{h;0,*}$, where $a(\cdot,\cdot)$ is the bilinear form defined in \eqref{bil ode} and $\overline{\mathcal{H}}_T(\cdot)$ is the isometric isomorphism defined in \eqref{HT}, and where we corrected a missing second power of $\frac{T}{2}$.
 
 \begin{oss}
 	As a consequence of the sharp Poincaré's inequalities \eqref{Poinc}, the estimate \eqref{zank 1 stima} can be slightly improved with the more accurate bound 
 	\begin{equation}\label{zank più sharp}
 	h \leq \frac{\sqrt{3}\pi}{\sqrt{2}(2+\sqrt{\mu} T)\mu T},
 	\end{equation}	
 	as observed in (\cite{Zank2020}). Therefore, we consider this bound in our numerical experiments.
 \end{oss}

\begin{oss}
	The well-posedness of the Galerkin-Bubnov FEM discretization of \eqref{eq ode} is equivalent to the well-posedness (with the same stability constants) of its Galerkin-Petrov FEM discretization. Indeed, as in the isogeometric setting, the restriction of operator  $\overline{\mathcal{H}}_T(\cdot)$ to the discrete trial space $S^1_{h;0,*}$ is actually an isometric isomorphism between $S^1_{h;0,*}$ and the test space $S^1_{h;*,0}$.
\end{oss}

With classic arguments, in (\cite{Coercive}) the authors prove that the discrete solution of the piecewise continuous linear FEM discretization of \eqref{eq ode} converges linearly in $|\cdot|_{H^1(0,T)}$ to the solution $u$ of \eqref{var ode} if $u \in H^2(0,T)$ (Theorem 4.8, \cite{Coercive}). Moreover, using Aubin-Nitsche's trick, it is also possible to prove quadratic convergence in $\|\cdot\|_{L^2(0,T)}$ norm, if the exact solution satisfies $u \in H^2(0,T)$. \\ \bigskip

Under the assumption of a uniform mesh-size, the linear system obtained from the Galerkin-Petrov FEM discretization of \eqref{eq ode} can be seen as a finite difference scheme; see Remark 4.2.8 of (\cite{Zank2020}). In these condition, the stability of the corresponding finite difference scheme holds if and only if
\begin{equation}\label{stab Linfty}
h < \sqrt{\frac{12}{\mu}},
\end{equation}
as a consequence of Chapter III.3 of (\cite{Hairer}). Let us note that the stability considered in \eqref{stab Linfty} is a uniform (w.r.t. $h$) boundness condition for the discrete solution in $\|\cdot\|_{L^\infty(0,T)}$ norm, which is necessary for stability of the discrete solution in classic Sobolev seminorms, due to the fact that the norm $\|\cdot\|_{L^\infty(0,T)}$ is controlled in $H^1(0,T)$, as a consequence of Morrey's Theorem.

\begin{oss}
Constraint \eqref{stab Linfty} is related to the CFL condition for the tensor-product space-time discretization with piecewise continuous linear solution and test functions of the wave propagatin problem \eqref{eqonde}. The details are explained in (\cite{Steinbach2019, Zank2020}). 
\end{oss} 

\vspace{0.8cm}
 As in (\cite{Steinbach2019, Zank2020}), as a numerical example for the Galerkin-Petrov finite element methods  we consider a uniform discretization of the time interval $(0,T)$ with $T = 10$ and a mesh-size $h = T/N$. For $\mu = 1000$ we consider the strong solution $u(t) = \sin^2\Big(\frac{5}{4}\pi t\Big)$
and we compute the integrals appearing at the right-hand side using high-order integration rules.

\begin{figure}[h!]	
	\centering
	\includegraphics[scale=0.3]{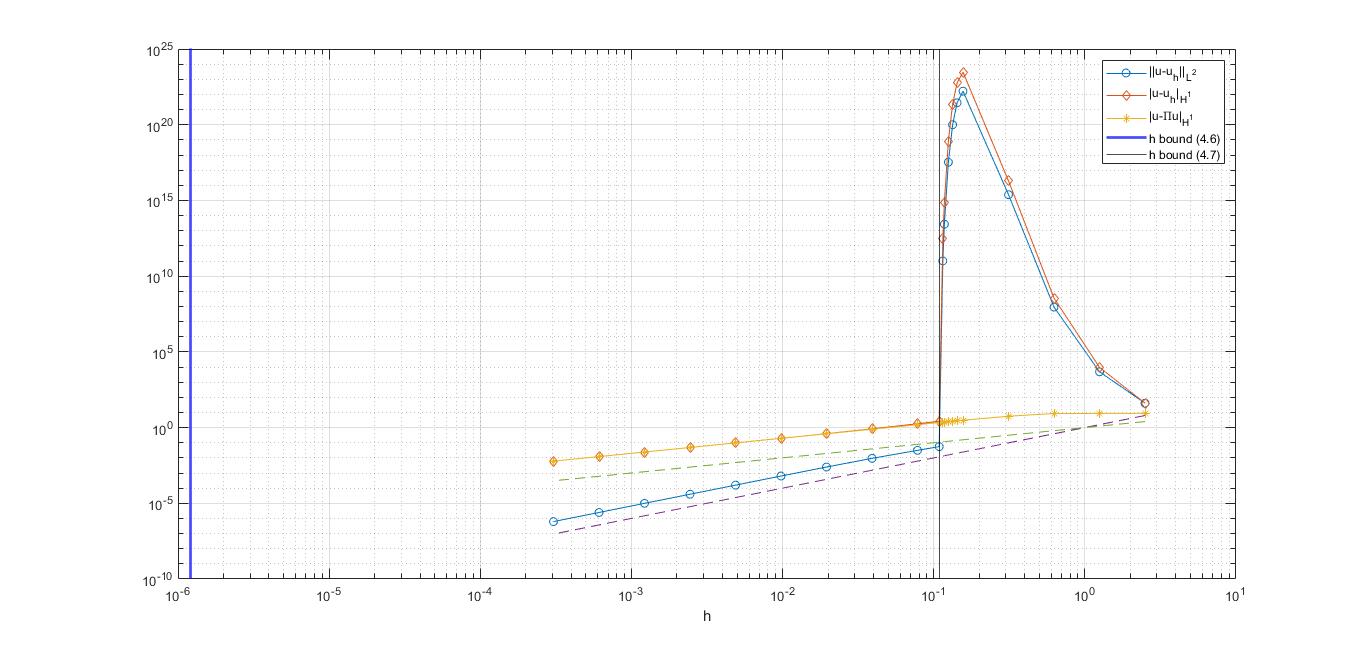}
	\caption{A $\log$-$\log$ plot of errors committed by piecewise continuous linear FEM in $|\cdot|_{H^1(0,T)}$ seminorm and in $\|\cdot\|_{L^2(0,T)}$ norm, with respect to a uniform mesh-size $h$. Also, the best approximation error in $|\cdot|_{H^1(0,T)}$ seminorm and bounds \eqref{zank più sharp}, \eqref{stab Linfty} are represented. The square of the wave number is $\mu=1000$ and the final time is $T=10$.}
	\label{fig:err fem}
\end{figure}

\begin{figure}[h!]
	\centering
	\includegraphics[scale=0.3]{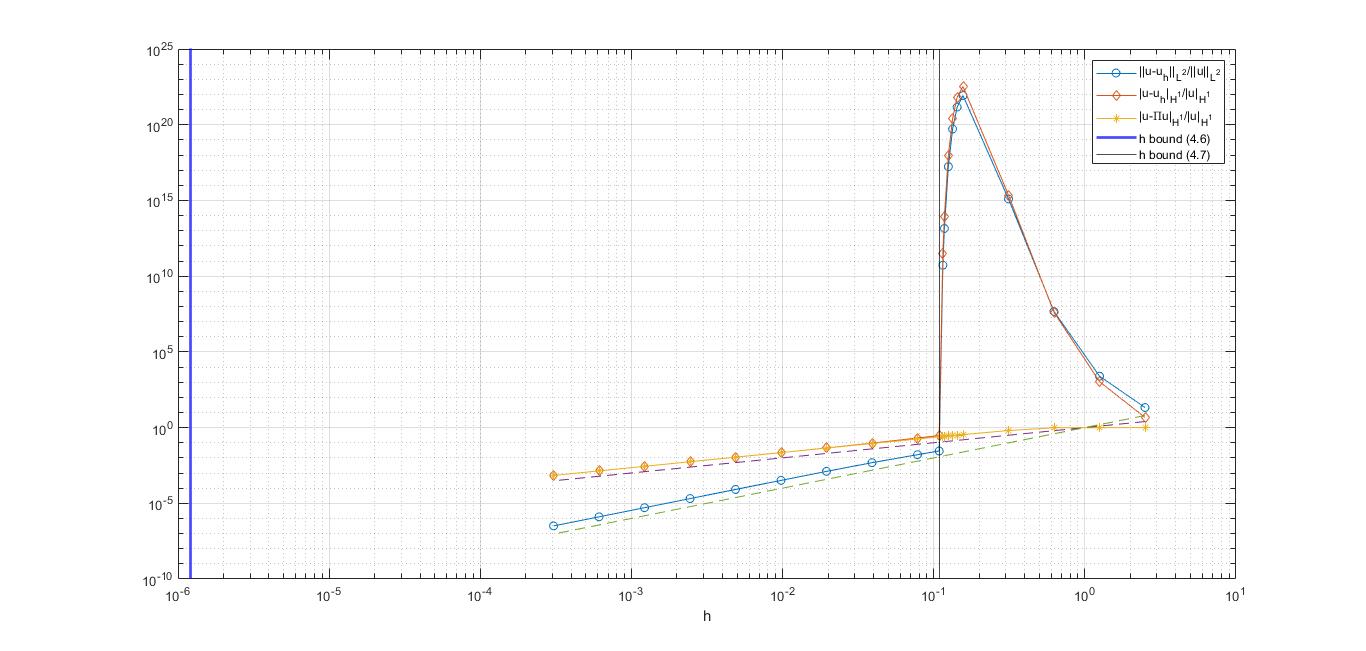}
	\caption{A $\log$-$\log$ plot of relative errors committed by piecewise continuous linear FEM in $|\cdot|_{H^1(0,T)}$ seminorm and in $\|\cdot\|_{L^2(0,T)}$ norm, with respect to a uniform mesh-size $h$. Also, the best approximation error in $|\cdot|_{H^1(0,T)}$ seminorm and bounds \eqref{zank più sharp}, \eqref{stab Linfty} are represented. The square of the wave number is $\mu=1000$ and the final time is $T=10$.}
\end{figure}

\begin{figure}
	\hspace{-1cm}
	\begin{minipage}[h!]{8.5cm}
		\centering
		\includegraphics[width=6.5cm]{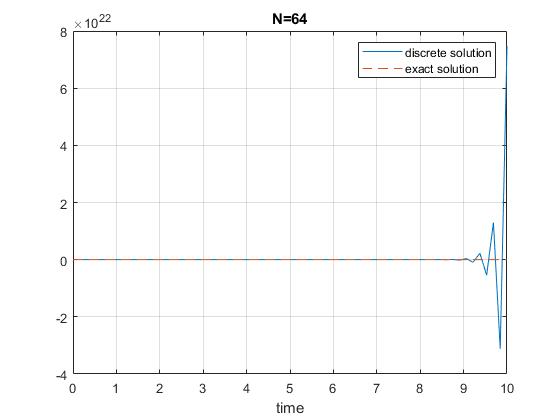}
		\caption{Exact and discrete solutions by piecewise linear FEM for $\mu = 1000$ and $N = 64$ elements (i.e., $h = 0.1563$).}
	\end{minipage} 
	\hspace{-1cm}
	\begin{minipage}[h!]{8.5cm}
		\centering
		\includegraphics[width=6.5cm]{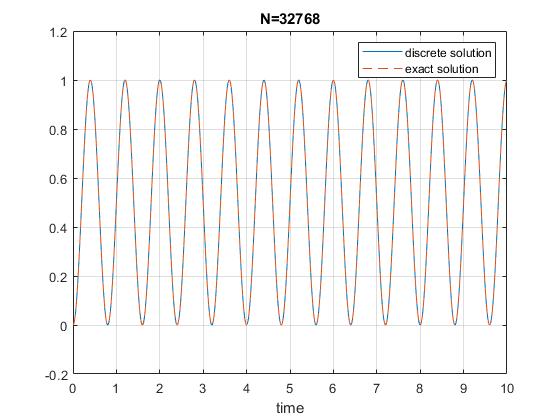}
		\caption{Exact and discrete solutions by piecewise linear FEM for $\mu = 1000$ and $N=32768$ elements (i.e., $h=0.0003$).}
	\end{minipage}
\end{figure}

We consider approximation errors since they are a reflection of instability. 

The minimum number of elements chosen is $N = 4$, the maximum number is $N=32768$, as in (\cite{Steinbach2019}). 

From Figure \ref{fig:err fem}, we can see that the maximal error occurs at $N = 64$: it is of order of $10^{23}$ in seminorm $|\cdot|_{H^1(0,T)}$ and $10^{22}$ in norm $\|\cdot\|_{L^2(0,T)}$, as we expect from (\cite{Steinbach2019}). 

As we can note in Figure \ref{fig:err fem}, there is convergence only for sufficiently small mesh-size $h$.  In particular, convergence is linear in $|\cdot|_{H^1(0,T)}$ seminorm and is quadratic in $\|\cdot\|_{L^2(0,T)}$ norm, as we expect. Clearly, the bound \eqref{zank più sharp} is suboptimal, since convergence starts for $h$ much larger than this threshold. Instead, the bound \eqref{stab Linfty} seems to be sharp with respect to the error committed by the finite elements: this suggests that there is a uniform (w.r.t. $h$) inf-sup value already for $h < \sqrt{\frac{12}{\mu}}$. Thus, it remains open to improve assumption \eqref{zank più sharp} to ensure a uniform inf-sup condition of \eqref{infsup fem} type. 

\begin{oss}\label{pollut effect}
Let us define $k:=\sqrt{\mu}$ to follow the notation used in the literature we refer to in this comment. It is well known that the numerical solution of the Helmholtz equation 
\begin{equation*}
-\Delta_x u(x)-k^2 u(x) = f(x)
\end{equation*}
with boundary conditions, obtained by classic Galerkin FEM, differs significantly from the best approximation with increasing wave number $k$. This phenomena is the so-called \textit{pollution effect}. The Galerkin FEM leads to quasi-optimal error estimates in which the constant factor by which the accuracy of the Galerkin solution differs from the best approximation error increases with increasing wave number (\cite{HARARI199159, IHLENBURG19959, BABUSKA1995325, Babuska2000}). On the other hand, it was shown in (\cite{Aziz1988ATP}) that the condition “$k^2 h$ is small” would be sufficient to guarantee that the error of the Galerkin solution is of the same magnitude as the error of the best approximation. However, this condition involves considerable computational complexity in three dimensions (\cite{Babuska2000}). Therefore, many attempts have been made in the mathematical and engineering literature to overcome this lack of robustness of the classic Galerkin FEM with respect to $k$ (\cite{HARARI199159,BABUSKA1995325}).

Our model problem \eqref{eq ode} is a wave propagation problem of Helmholtz type with initial conditions, instead of boundary conditions. Therefore, morally, we could expect ``a certain pollution effect''. Hence, we decide to plot the best approximation error in order to estimate some kind of pollution error, which could be related to the conditioned stability of the classic Galerkin FEM. 

Note that, unlike Helmholtz, in our model problem \eqref{eq ode} we have two types of error-source frequencies. The first one is the frequency of the equation operator, i.e., $\sqrt{\mu}=\sqrt{1000}$ in our example, which dictates the number of oscillations in the unit time of typical solutions (e.g. homogeneous) of the problem. The second one is the specific frequency of the solution we consider, i.e., $f := \frac{\omega}{2\pi} = \frac{10 \pi}{4}\frac{1}{2 \pi}=\frac{5}{4}$ in our example, which dictates the number of oscillations in the unit time of this specific solution. Only the former is the source of a certain pollution error for this problem, whereas the latter is related to the best approximation error. However, both are related to the choice of the mesh-size in order to get a discrete space whose Galerkin error is ``small''.
\end{oss}

\subsection{Inf-sup tests for piecewise continuous linear finite element method}\label{sec infsup fem}

As noticed in Section \ref{sec errors fem}, the bound \eqref{stab Linfty} seems to be sharp with respect to the error committed by the finite elements: this suggests that there is a uniform (w.r.t. $h$) inf-sup value already for $h < \sqrt{\frac{12}{\mu}}$. Therefore, we make inf-sup tests for the discrete bilinear form of the piecewise continuous linear FEM discretization. The idea is to numerically estimate the discrete inf-sup and visualize its behaviour with respect to $(\mu,h)$: we expect that in the region satisfying $h \geq \sqrt{\frac{12}{\mu}}$ there are infinitesimal inf-sup values, whereas we expect a uniformly (w.r.t. $h$) limited behaviour in the complementary region. \\ \bigskip

\textbf{Numerical estimate of the discrete inf-sup.} Let us define:
\begin{equation}\label{beta dis}
\beta(\mu,h,T):=\inf_{u_h \in S^1_{h;0,*}} \sup_{v_h \in S^1_{h;*,0}} \frac{a(u_h,v_h)}{|u_h|_{H^1(0,T)} |v_h|_{H^1(0,T)}},
\end{equation}
where $a(\cdot,\cdot)$ is the bilinear form defined in \eqref{bil ode}. Let us also define the matrices $\mathbf{A},\mathbf{H}_1,\mathbf{H}_2$ such that 
\begin{gather}
\begin{split}
[\mathbf{A}]_{i,j}:=-\langle \partial_t b_{1,j},\partial_t b_{1,i} \rangle_{L^2(0,T)} &+ \mu \langle  b_{1,j}, b_{1,i} \rangle_{L^2(0,T)}\\ \quad \text{for} \ &i=1,\ldots,M-1, \ j=2,\ldots,M 
\end{split}
\\{[\mathbf{H}_1]}_{i,j}:=\langle \partial_t b_{1,i},\partial_t b_{1,j} \rangle_{L^2(0,T)} \quad \text{for} \ i,j=2,\ldots,M, \\
{[\mathbf{H}_2]}_{i,j}:=\langle \partial_t b_{1,i},\partial_t b_{1,j} \rangle_{L^2(0,T)} \quad \text{for} \ i,j=1,\ldots,M-1,
\end{gather}
where $b_{1,k}$ for $k=1,\ldots,M$ are the classic \textit{hat functions} such that $S^1_h(0,T)=\text{span}\{b_{1,k}\}_{k=1}^M$. Through this matrices we can estimate the discrete inf-sup value \eqref{beta dis}.

\begin{prop}\label{estimate infsup}
	The discrete inf-sup value \eqref{beta dis} satisfies
	\begin{equation*}
	\beta(\mu,h,T)=\sqrt{\lambda_{min}},
	\end{equation*}
	where $\lambda_{min} \geq 0$ is the minimum eigenvalue of the generalised eigenvalue problem:
	\begin{equation*}
	\mathbf{A}^T \mathbf{H}_2^{-1}\mathbf{A}\overset{\rightarrow}{x}=\lambda\mathbf{H}_1 \overset{\rightarrow}{x}, \quad \text{for some} \ \overset{\rightarrow}{x} \in \mathbb{R}^{M-1}.
	\end{equation*} 
\end{prop}

\begin{proof}
	Let $u_h \in S^1_{h;0,*}$ and $v_h \in S^1_{h;*,0}$. They can be represented as
	\begin{gather*}
	u_h=\sum_{j=2}^M u_j b_{1,j},\\
	v_h=\sum_{i=1}^{M-1} v_i b_{1,i}.
	\end{gather*}
	Let us define $\overset{\rightarrow}{u}:=(u_2,\ldots,u_M)^T \in \mathbb{R}^{M-1}$, $\overset{\rightarrow}{v}:=(v_1,\ldots,v_{M-1})^T \in \mathbb{R}^{M-1}$. Therefore,
    \begin{equation}\label{vett}
    \begin{cases}
    a(u_h,v_h)=\overset{\rightarrow}{u}^T\mathbf{A}^T \overset{\rightarrow}{v},\\
    |u_h|^2_{H^1(0,T)}=\overset{\rightarrow}{u}^T\mathbf{H}_1 \overset{\rightarrow}{u},\\
    |v_h|^2_{H^1(0,T)}=\overset{\rightarrow}{v}^T\mathbf{H}_2 \overset{\rightarrow}{v}.
    \end{cases}
    \end{equation}
    In order to lighten the notation we will denote $\overset{\rightarrow}{u}$ with $u$ and $\overset{\rightarrow}{v}$ with $v$. By \eqref{vett} the following equality holds
    \begin{equation*}
    \inf_{u_h \in S^1_{h;0,*}} \sup_{v_h \in S^1_{h;*,0}} \frac{a(u_h,v_h)}{|u_h|_{H^1(0,T)} |v_h|_{H^1(0,T)}}= \inf_{u \in \mathbb{R}^{M-1}} \sup_{v \in \mathbb{R}^{M-1}}\frac{{u}^T\mathbf{A}^T {v}}{\sqrt{u^T\mathbf{H}_1 {u}} \sqrt{{v}^T\mathbf{H}_2 {v}}}.
    \end{equation*}
    $\mathbf{H}_2$ is a symmetric real positive-definite matrix, thus $\mathbf{H}_2=\mathbf{H}_2^\frac{1}{2}\mathbf{H}_2^\frac{1}{2}$. With the change of variable $\tilde{v}:=\mathbf{H}_2^\frac{1}{2}v$, the following equalities hold:
    \begin{equation*}
    \sup_{v \in \mathbb{R}^{M-1}}\frac{{u}^T\mathbf{A}^T {v}}{ \sqrt{{v}^T\mathbf{H}_2 {v}}}=\sup_{\tilde{v} \in \mathbb{R}^{M-1}}\frac{{u}^T\mathbf{A}^T \mathbf{H}_2^{-\frac{1}{2}} \tilde{v}}{ \sqrt{\tilde{v}^T \tilde{v}}} = \Big\|\mathbf{H}_2^{-\frac{1}{2}} \mathbf{A} u\Big\|_2,
    \end{equation*}
    where $\|\cdot\|_2$ denote the euclidean norm of a vector. 
    
    $\mathbf{H}_1$ is a symmetric real positive-definite matrix, thus $\mathbf{H}_1=\mathbf{H}_1^\frac{1}{2}\mathbf{H}_1^\frac{1}{2}$. With the change of variable $\tilde{u}:=\mathbf{H}_1^\frac{1}{2}u$, the following equalities hold:
    \begin{equation*}
    \begin{split}
    \Big\|\mathbf{H}_2^{-\frac{1}{2}} \mathbf{A} u\Big\|_2^2&=u^T\mathbf{A}^T\mathbf{H}_2^{-\frac{1}{2}}\mathbf{H}_2^{-\frac{1}{2}} \mathbf{A} u=u^T\mathbf{A}^T\mathbf{H}_2^{-1} \mathbf{A} u\\
    &=\tilde{u}^T\mathbf{H}_1^{-\frac{1}{2}}\mathbf{A}^T\mathbf{H}_2^{-1} \mathbf{A} \mathbf{H}_1^{-\frac{1}{2}} \tilde{u}.
    \end{split}
    \end{equation*}
    Therefore,
    \begin{equation*}
    \begin{split}
    \beta(\mu,h,T)&=\inf_{\tilde{u} \in \mathbb{R}^{M-1}}\sqrt{\frac{\tilde{u}^T\mathbf{H}_1^{-\frac{1}{2}}\mathbf{A}^T\mathbf{H}_2^{-1} \mathbf{A} \mathbf{H}_1^{-\frac{1}{2}} \tilde{u}}{\tilde{u}^T\tilde{u}}}\\
    &=\sqrt{\inf_{\tilde{u} \in \mathbb{R}^{M-1}} \frac{\tilde{u}^T\mathbf{H}_1^{-\frac{1}{2}}\mathbf{A}^T\mathbf{H}_2^{-1} \mathbf{A} \mathbf{H}_1^{-\frac{1}{2}} \tilde{u}}{\tilde{u}^T\tilde{u}}}\\
    &=\sqrt{\lambda_{min}},
    \end{split}
    \end{equation*}
    where $\lambda_{min}$ is the minimum eigenvalue of the symmetric real positive-semide\-fi\-nite matrix $\mathbf{H}_1^{-\frac{1}{2}}\mathbf{A}^T\mathbf{H}_2^{-1} \mathbf{A} \mathbf{H}_1^{-\frac{1}{2}}$. This means that $\lambda_{min}$ is the smallest number that satisfies
    \begin{equation*}
    \mathbf{H}_1^{-\frac{1}{2}}\mathbf{A}^T\mathbf{H}_2^{-1} \mathbf{A} \mathbf{H}_1^{-\frac{1}{2}} \overset{\rightarrow}{x} = \lambda \overset{\rightarrow}{x}, \quad \text{for some} \ \overset{\rightarrow}{x} \in \mathbb{R}^{M-1},
    \end{equation*}
    i.e.,
    \begin{equation*}
    \mathbf{A}^T\mathbf{H}_2^{-1} \mathbf{A} \mathbf{H}_1^{-\frac{1}{2}} \overset{\rightarrow}{x} = \lambda \mathbf{H}_1^{\frac{1}{2}} \overset{\rightarrow}{x}, \quad \text{for some} \ \overset{\rightarrow}{x} \in \mathbb{R}^{M-1}.
    \end{equation*}
    In order to lighten the notation we will denote $\overset{\rightarrow}{x}$ with $x$. With the change of variable $\tilde{x}:=\mathbf{H}_1^{-\frac{1}{2}} x$, $\lambda_{min}$ is the minimum of 
    \begin{equation*}
    \mathbf{A}^T\mathbf{H}_2^{-1} \mathbf{A} \tilde{x}= \lambda \mathbf{H}_1 \tilde{x}, \quad \text{for some} \ \tilde{x} \in \mathbb{R}^{M-1}.
    \end{equation*}
     The thesis is therefore proven.
\end{proof}

\vspace{0.8cm}

We fix the final time $T=10$ and a uniform mesh. We numerically study the behaviour of $\beta(\mu,h)$ of \eqref{beta dis} with respect to $(\mu,h)$ by means of a p-colour plot of $\log(\beta)$ depending on $(\log(\mu),\log(h))$, so as to visualize the development of $\beta(\mu,h)$ more effectively. 

\begin{figure}[h!]
	\centering
	\includegraphics[scale=0.5]{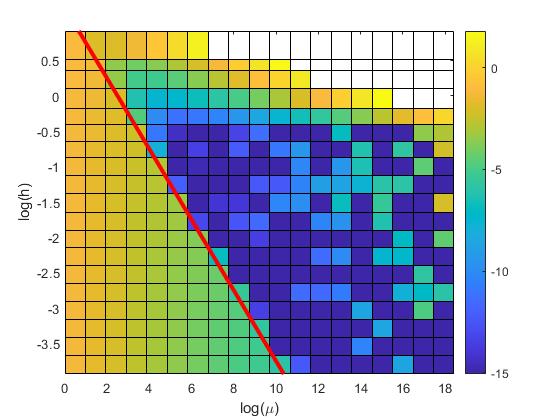}
	\caption{A p-color plot of $\log(\beta)$ of piecewise continuous linear FEM with respect to $(\log(\mu),log(h))$. The red line is the natural logarithm of the upper bound in \eqref{stab Linfty}.}
	\label{fig:beta_fem}
\end{figure}

Firstly, let us note that the MATLAB function \textit{p-color} sets by default the MATLAB values \textit{-Inf} (i.e., numbers whose absolute value is too large to be represented as conventional floating-point values) to dark blue: we have checked this numerically and it is also clarified by Figure \ref{fig:3D fem}, where the empty regions correspond to $log(\beta)=-Inf$.

\begin{figure}[h!]
	\centering
	\includegraphics[scale=0.3]{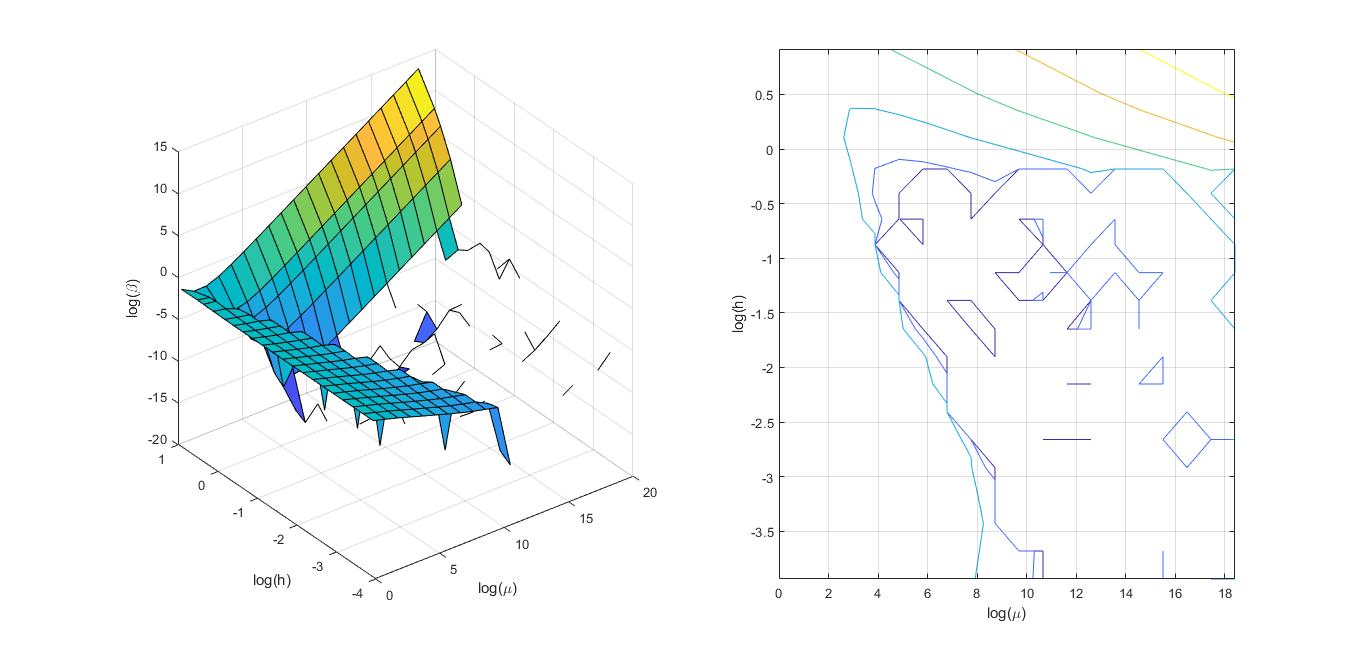}
	\caption{On the left, a three-dimensional plot of $\log(\beta)$ of piecewise continuous linear FEM with respect to $(\log(\mu),\log(h))$. On the right, the contour lines of $\log(\beta)$ with respect to $(\log(\mu),\log(h))$.}
	\label{fig:3D fem}
\end{figure}

The red line of Figure \ref{fig:beta_fem} is the natural logarithm of the upper bound in \eqref{stab Linfty}. We can clearly see in Figure \ref{fig:beta_fem} that it is the separation margin of two regions in which we observe a significantly different behaviour of $\beta(\mu,h)$. In the region below the red line, i.e., for $h < \sqrt{\frac{12}{\mu}}$, we observe a uniformly (w.r.t. $h$) bounded $\beta(\mu,h)$, whereas in the region above the red line, i.e., for $h > \sqrt{ \frac{12}{\mu}}$, we observe a predominantly infinitesimal $\beta(\mu,h)$. This behaviour of the discrete inf-sup is indeed what we expect from Figure \ref{fig:err fem}, in which the bound \eqref{stab Linfty} on $h$ seemed to be sharp, suggesting to us a uniformly (w.r.t. $h$) bounded discrete inf-sup already for $h < \sqrt{\frac{12}{\mu}}$. In the stability region, we can see a dependency of the discrete inf-sup on $\mu$ of order $\mu^{-\frac{1}{2}}$, as noticed in (\cite{Zank2020}).

Note that in Figure \ref{fig:beta_fem} we do not consider the values of discrete inf-sup in the upper-right white region. In this region, the values of $\log(\beta)$ are greater than $10^2$. Actually, as a consequence of $h$ ``coarse'' and ``high'' wave-number, such large values for the discrete inf-sup are natural results, since value \eqref{beta dis} is directly proportional to $\mu$ for wave-number and mesh-size that are ``very large''. This is due to inverse inequalities which allow the term $L^2$ to dominate the derivatives. We are not interested in working under these conditions, since, for those wave numbers, $h$ is too coarse compared to the resolution we expect to need in order to obtain satisfactory numerical results. Therefore, we do not visualize the corresponding inf-sup values.

\subsection{Errors committed by quadratic isogeometric discretization with maximal regularity}\label{sec errs iga}

In this Section we show some numerical results that we obtain by testing our theoretical considerations of Section \ref{sec iga}. 

As in Section \ref{sec errors fem} and in (\cite{Steinbach2019, Zank2020}), as a numerical example for the Galerkin-Petrov quadratic isogeometric discretization with maximal regularity, we consider a uniform discretization of the time interval $(0,T)$ with $T = 10$ and a mesh-size $h = T/N$. For $\mu = 1000$ we consider the strong solution $u(t) = \sin^2\Big(\frac{5}{4}\pi t\Big)$
and we compute the integrals appearing at the right-hand side using high-order integration rules. 

As in Section \ref{sec errors fem} we consider approximation errors since they are a reflection of instability. 

The minimum number of elements chosen is $N = 4$, the maximum number is $N=4096$ (it is not $N = 32768$ due to the memory limits of the laptop used).

\begin{figure}[h!]	
	\centering
	\includegraphics[scale=0.3]{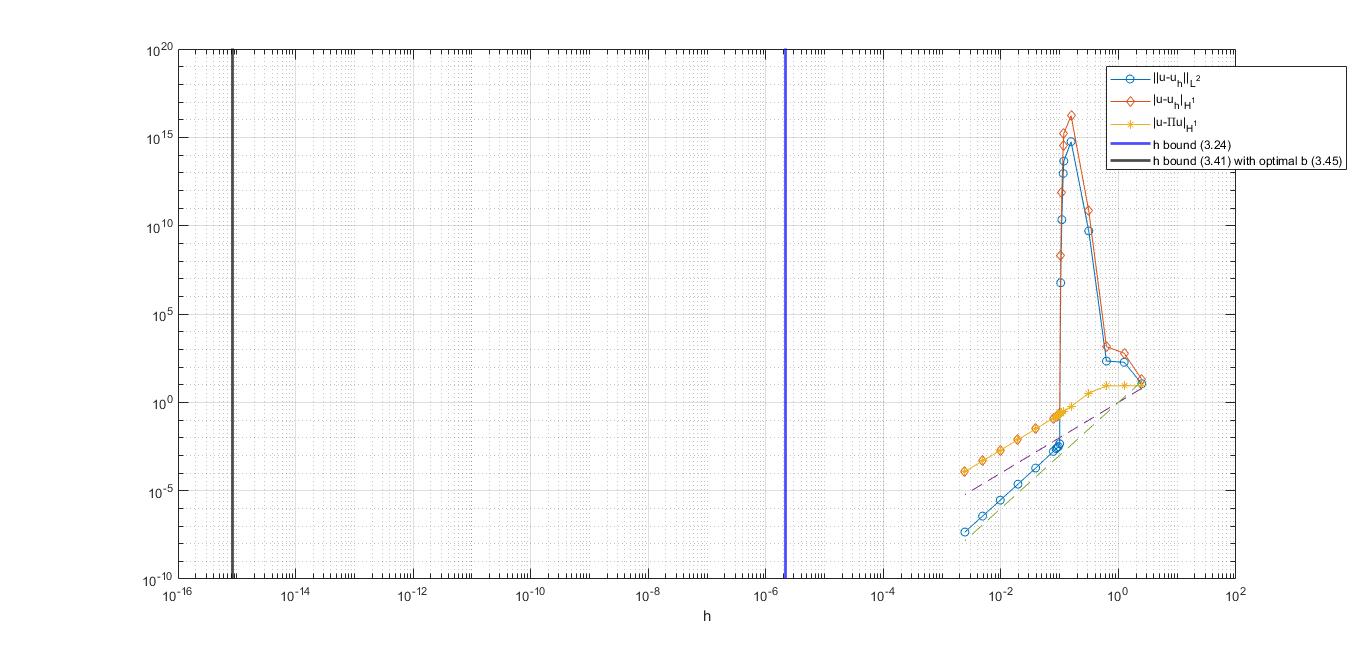}
	\caption{A $\log$-$\log$ plot of errors committed by quadratic IGA, with maximal regularity, in $|\cdot|_{H^1(0,T)}$ seminorm and in $\|\cdot\|_{L^2(0,T)}$ norm, with respect to a uniform mesh-size $h$. Also, the best approximation error in $|\cdot|_{H^1(0,T)}$ seminorm and bounds \eqref{h bound}, \eqref{h bound gard} (with the optimal choice \eqref{b}) are represented. The square of the wave number is $\mu=1000$ and the final time is $T=10$.}
	\label{fig:err iga}
\end{figure}

\begin{figure}[h!]	
	\centering
	\includegraphics[scale=0.3]{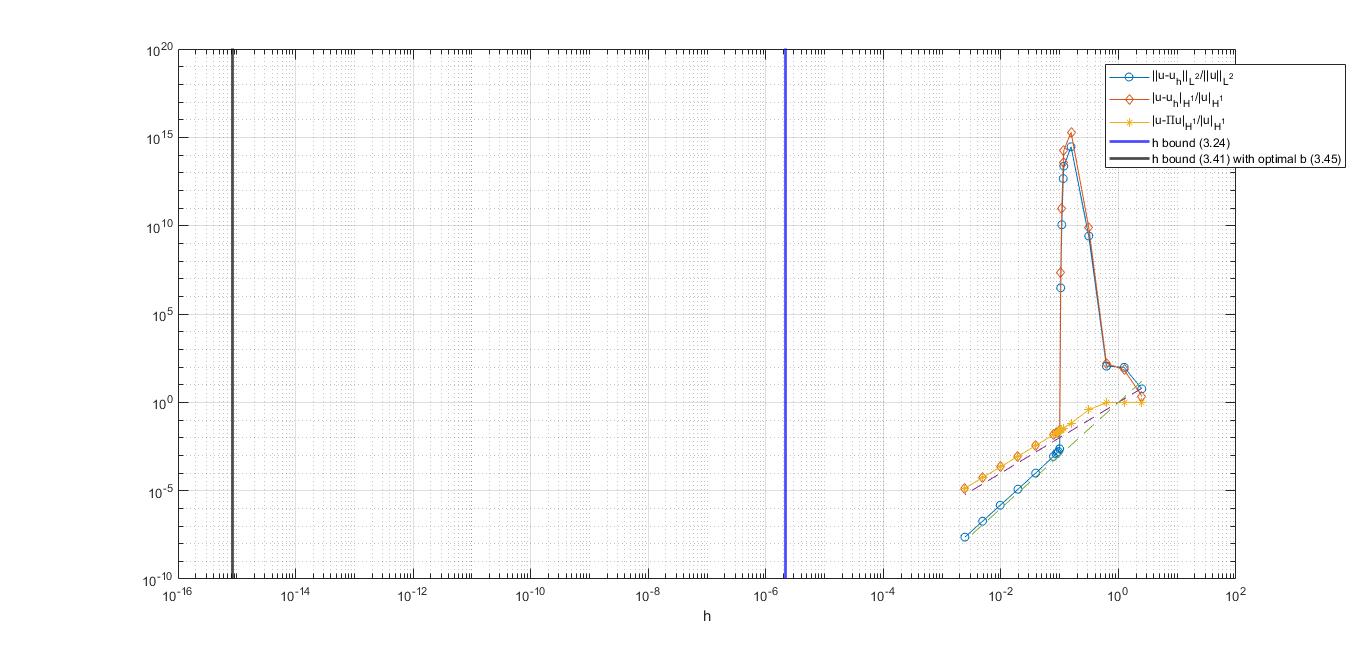}
	\caption{A $\log$-$\log$ plot of relative errors committed by quadratic IGA, with maximal regularity, in $|\cdot|_{H^1(0,T)}$ seminorm and in $\|\cdot\|_{L^2(0,T)}$ norm, with respect to a uniform mesh-size $h$. Also, the best approximation error in $|\cdot|_{H^1(0,T)}$ seminorm and bounds \eqref{h bound}, \eqref{h bound gard} (with the optimal choice \eqref{b}) are represented. The square of the wave number is $\mu=1000$ and the final time is $T=10$.}
	\label{}
\end{figure}

From Figure \ref{fig:err iga}, you can see that the maximal error occurs at $N = 64$: it is of order of $10^{16}$ in seminorm $|\cdot|_{H^1(0,T)}$ and $10^{15}$ in norm $\|\cdot\|_{L^2(0,T)}$. Note that the maximal error, and more generally for all values of $h$, is strictly smaller than the error committed by piecewise continuous linear FEM discretization, see Section \ref{sec errors fem}. This result can be interpreted as a consequence of the good approximation properties of B-spline technology (\cite{HUGHES20084104, n-width}).

\begin{figure}
	\hspace{-1cm}
	\begin{minipage}[h!]{8.5cm}
		\centering
		\includegraphics[width=6.5cm]{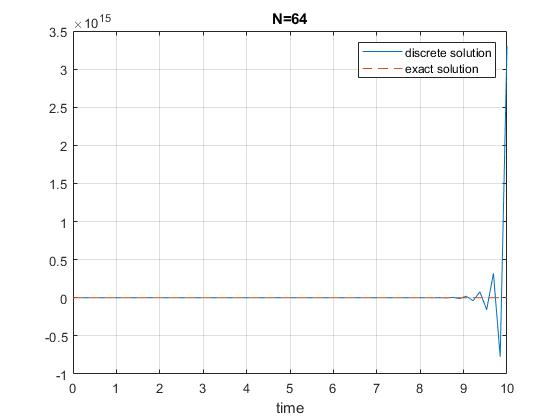}
		\caption{Exact and discrete solutions by quadratic IGA for $\mu = 1000$ and $N = 64$ elements (i.e., $h = 0.1563$).}
	\end{minipage} 
	\hspace{-1cm}
	\begin{minipage}[h!]{9cm}
		\centering
		\includegraphics[width=6.5cm]{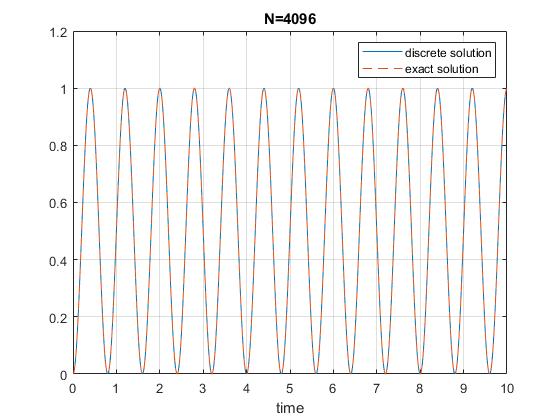}
		\caption{Exact and discrete solutions by quadratic IGA for $\mu = 1000$ and $N=4096$ elements (i.e., $h=0.0024$).}
	\end{minipage}
\end{figure}

As you can see in Figure \ref{fig:err iga}, there is convergence only for sufficiently small mesh-size $h$.  In particular, convergence is quadratic in $|\cdot|_{H^1(0,T)}$ seminorm, as we expect from Section \ref{sec iga}, and it is cubic in $\|\cdot\|_{L^2(0,T)}$ norm. Note that in our theoretical discussion of Section \ref{sec iga} we did not estimate the order of convergence in $\|\cdot\|_{L^2(0,T)}$ norm, since it was beyond our interest. However, using Aubin-Nitsche's trick, one can verify that the $L^2$-error converges cubically if the exact solution $u$ satisfies $u \in H^3(0,T)$. 

Clearly, both bounds \eqref{h bound}, \eqref{h bound gard} (with the optimal choice \eqref{b}) are suboptimal, since a beginning of convergence is observed for $h$ much larger than these thresholds, see Figure \ref{fig:err iga}. Moreover, the constraint \eqref{h bound gard} is more stringent than \eqref{h bound}, as we expect to happen asymptotically (i.e., $\mu \rightarrow \infty$) from Remark \ref{oss asympt}. Thus, it remains open to improve assumptions \eqref{h bound}, \eqref{h bound gard} to ensure stability conditions of \eqref{stab zank}, \eqref{stab iga gard} type, respectively. That is, we are interested in an upper bound on the mesh-size that provides us with stability and is sharp. 

\subsection{Inf-sup tests for quadratic isogeometric discretization with maximal regularity}

In this Section we numerically estimate the discrete inf-sup value and visualize its behaviour with respect to $(\mu,h)$, with a uniform mesh-size $h$ and final instant $T=10$. Since the development of the error committed by the quadratic IGA is similar to that of the linear FEM, we expect the behaviour of the discrete inf-sup to be similar as well. In particular, if we numerically study $\beta(\mu,h)$ of \eqref{beta dis} with respect to $(\mu,h)$ by means of a p-colour plot of $\log(\beta)$ depending on $(\log(\mu),\log(h))$, we expect to visualize a line that is the separation margin of stability and instability regions. This line corresponds to the sharp bound on $h$ that, if satisfied, ensures the desired stability.

\begin{figure}[h!]
	\centering
	\includegraphics[scale=0.5]{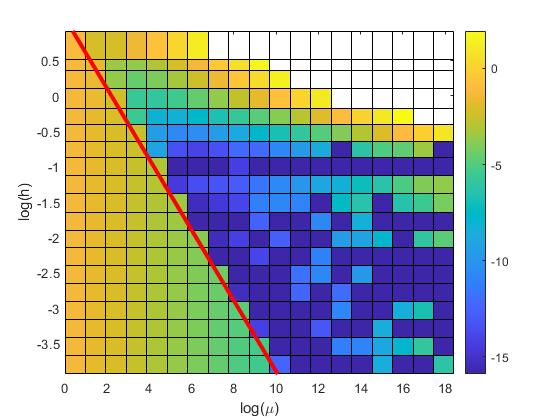}
	\caption{A p-color plot of $\log(\beta)$ of quadratic IGA with maximal regularity with respect to $(\log(\mu),log(h))$. The red line represents $\log(h)=\frac{1}{2}\log(9)-\frac{1}{2}\log(\mu)$.}
	\label{fig:beta_iga}
\end{figure}

\begin{figure}[h!]
	\centering
	\includegraphics[scale=0.3]{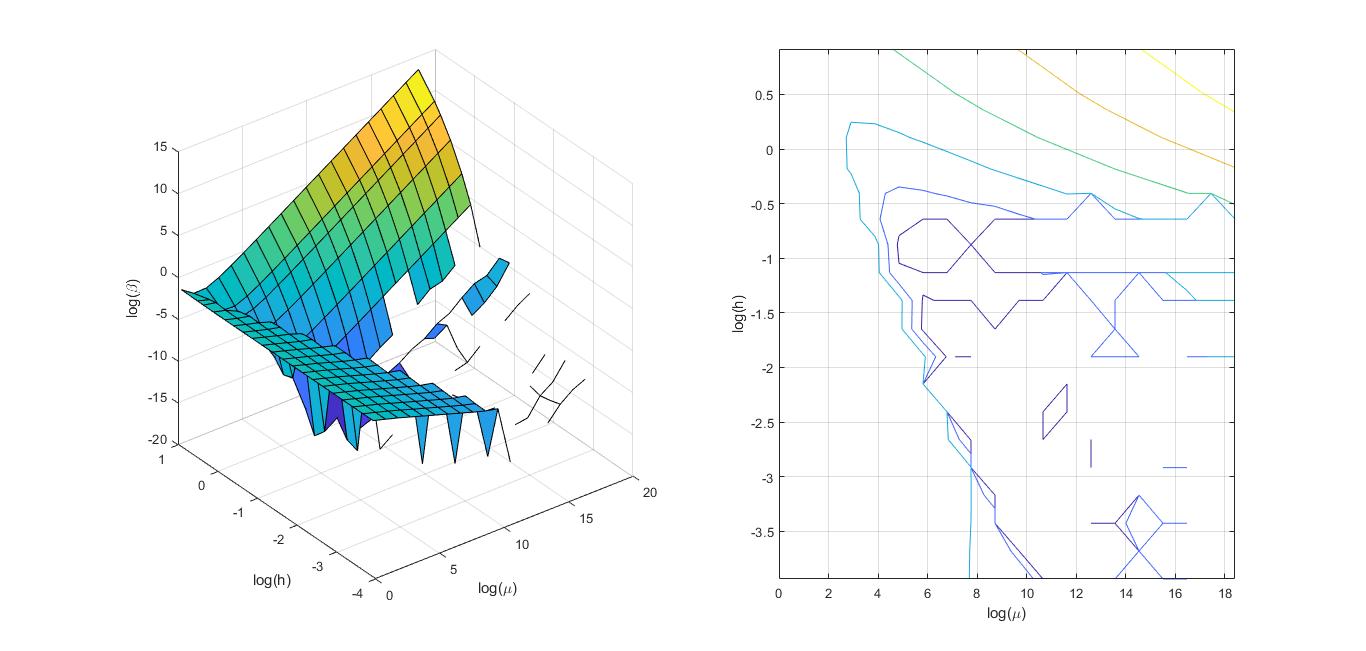}
	\caption{On the left, a three-dimensional plot of $\log(\beta)$ of quadratic IGA with maximal regularity with respect to $(\log(\mu),\log(h))$. On the right, the contour lines of $\log(\beta)$ with respect to $(\log(\mu),\log(h))$.}
	\label{fig:3D iga}
\end{figure}

\bigskip
We estimate the isogeometric discrete inf-sup using Proposition \ref{estimate infsup} with the discrete solution and test spaces defined in \eqref{iga space}, \eqref{iga test space}, respectively. 

The numerical results in Figure \ref{fig:beta_iga} and Figure \ref{fig:3D iga} confirm what we expect. The red line of Figure \ref{fig:beta_iga} corresponds to the constraint
\begin{equation}\label{emp constraint}
h < \sqrt{\frac{9}{\mu}}.
\end{equation}
Indeed, in the region below the red line, i.e., for $h < \sqrt{\frac{9}{\mu}}$, we observe a uniformly (w.r.t. $h$) bounded $\beta(\mu,h)$, whereas in the region above the red line, i.e., for $h > \sqrt{\frac{9}{\mu}}$, we observe a predominantly infinitesimal $\beta(\mu,h)$. In the stability region, we can see a dependency of the discrete inf-sup on $\mu$ of order $\mu^{-1/2}$, as in the continuous linear FEM case of Section \ref{sec infsup fem}. 

In Figure \ref{fig:beta_iga} we do not consider the values of discrete inf-sup that are greater than $10^2$. Indeed these values are not a physical phenomenon, but simply the result of an unsuitable discretization of the problem, as observed in Section \ref{sec infsup fem}. 

\begin{figure}[h!]	
	\centering
	\includegraphics[scale=0.3]{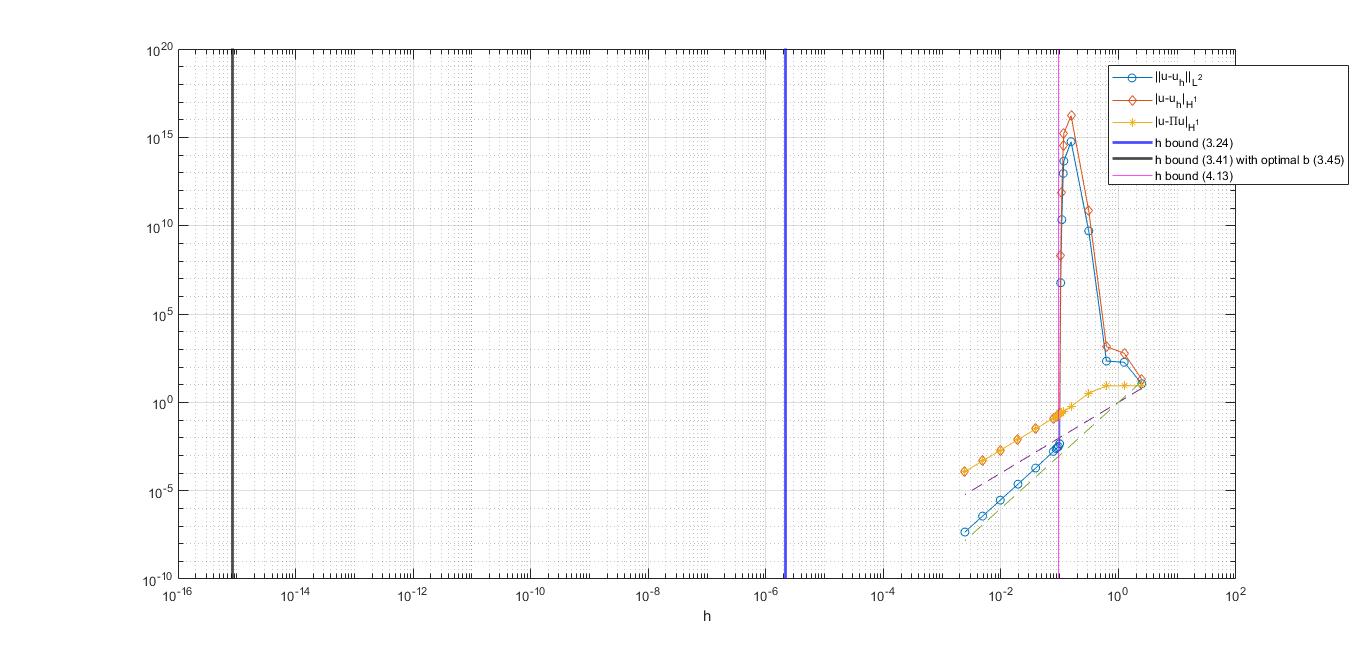}
	\caption{A $\log$-$\log$ plot of errors committed by quadratic IGA, with maximal regularity, in $|\cdot|_{H^1(0,T)}$ seminorm and in $\|\cdot\|_{L^2(0,T)}$ norm, with respect to a uniform mesh-size $h$. Also, the best approximation error in $|\cdot|_{H^1(0,T)}$ seminorm and bounds \eqref{h bound}, \eqref{h bound gard} (with the optimal choice \eqref{b}), \eqref{emp constraint} are represented. The square of the wave number is $\mu=1000$ and the final time is $T=10$.}
	\label{fig:err iga con stima empirica}
\end{figure}
 
In Figure \ref{fig:err iga con stima empirica} one can note the sharpness of the constraint \eqref{emp constraint} with respect to the error committed by the quadratic IGA with maximal regularity. This result is what we actually expect, as a consequence of the inf-sup stability for $h < \sqrt{\frac{9}{\mu}}$.\\ \bigskip

Another interesting numerical test displays, in the stability region given by \eqref{emp constraint}, the numerically estimated isogeometric inf-sup with respect to the mesh-size $h$ at a fixed $\mu$, see Figure \ref{fig: sez infsup depend on h}. This is useful to further see that the discrete inf-sup is actually uniformly limited in the region given by the constraint \eqref{emp constraint}.

\begin{figure}[h!]	
	\centering
	\includegraphics[scale=0.4]{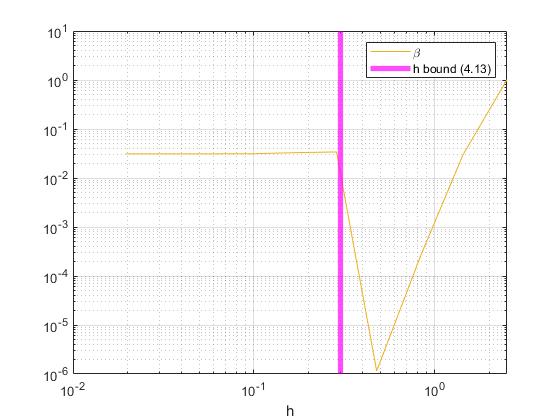}
	\caption{A $\log$-$\log$ plot of the isogeometric inf-sup value with respect to a uniform mesh-size $h$, for a fixed $\mu = 100$.}
	\label{fig: sez infsup depend on h}
\end{figure}

Comparing the theoretical inf-sup estimates \eqref{infsup zank}, \eqref{stab iga gard} of Section \ref{sec iga} with the numerical inf-sup could be useful in order to understand how accurate the theoretical bounds are. Therefore, in Figure \ref{fig: sez infsup depend on mu} we visualize the three stability bounds \eqref{emp constraint}, \eqref{h bound}, \eqref{h bound gard} and the relative inf-sup estimates at fixed $h$. It is sufficient to fix $h$ and compare them as functions of the single variable $\mu$, since in the stability region their mesh-size dependence disappears.
In order to represent these three constraints on the mesh-size with respect to $\mu$, we approximate \eqref{h bound}, \eqref{h bound gard} as in Remark \ref{oss asympt}, i.e., 
\begin{gather*}
h < \mu^{-3/2},\\
h < \mu^{-7/2},
\end{gather*}
respectively, which result in
\begin{gather*}
\mu < h^{-2/3},\\
\mu < h^{-2/7}.
\end{gather*}

Since the approximation of \eqref{h bound gard} with $h < \mu^{-7/2}$ corresponds to the slightly suboptimal value $b = \mu T^2$, plotting the inf-sup estimate of Theorem \ref{theo gard iga} with this choice is natural. 

As one can note in Figure \ref{fig: sez infsup depend on mu}, the inf-sup estimate of Theorem \ref{theo gard iga} with the choice $b = \mu T^2$ decreases more slowly for $\mu \rightarrow \infty$ than the inf-sup \eqref{infsup zank}, as we expect from Remark \ref{oss asympt}. Moreover, the former is more accurate than the latter, being larger at the same $h$ and $\mu$ for which they are comparable.

\begin{figure}[h!]	
	\centering
	\includegraphics[scale=0.5]{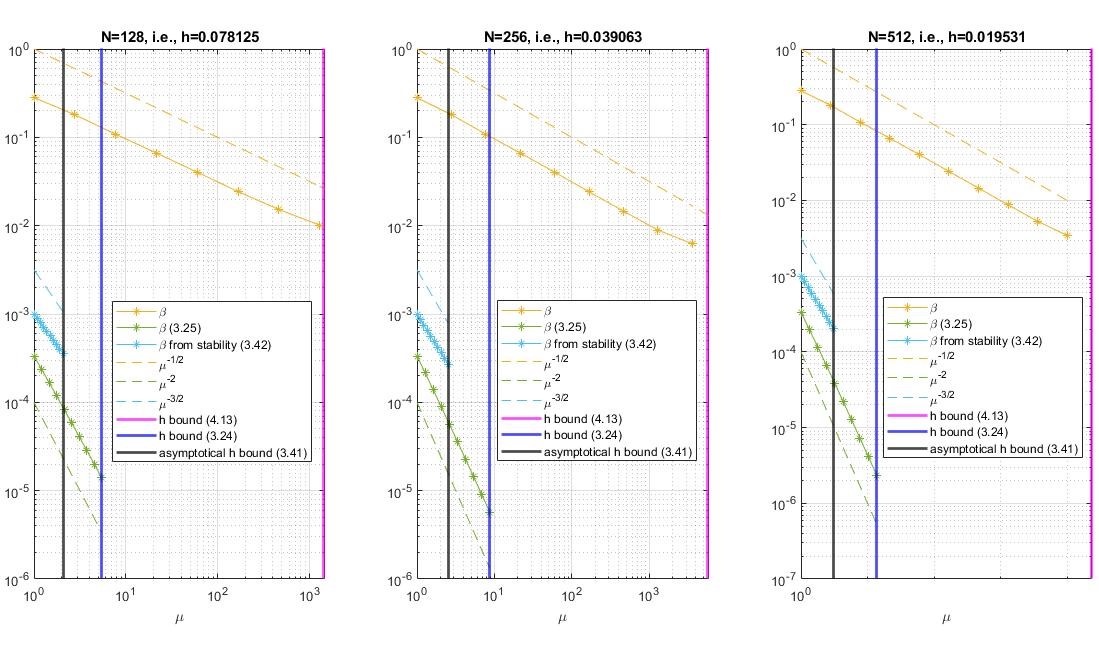}
	\caption{Three $\log$-$\log$ plots of the isogeometric inf-sup value and of its lower-bounds with respect to $\mu$ for a fixed uniform mesh-sizes $h$.}
	\label{fig: sez infsup depend on mu}
\end{figure}

\section{Unconditional stability}
\subsection{Stabilization of piecewise continuous linear finite element method}\label{stab fem}
In (\cite{Steinbach2019, Zank2020}) the authors introduce a stabilized piecewise continuous linear finite element discretization of \eqref{eq ode} in order to overcome the mesh constraints \eqref{zank più sharp}, \eqref{stab Linfty}. In this Section we briefly recall their main results about this stabilization and we show some numerical experiments that we make to test their theoretical considerations. 

O. Steinbach and M. Zank define a new discrete bilinear form \\ $a_h: S^1_{h;0,*} \times S^1_{h;*,0} \rightarrow \mathbb{R}$ such that
\begin{equation}\label{ah fem}
a_h(u_h,v_h):=-\langle \partial_t u_h, \partial_t v_h \rangle _{L^2(0,T)} + \mu \langle u_h, Q_h^0 v_h \rangle_{L^2(0,T)},
\end{equation}
for all $u_h\in S^1_{h;0,*}$, $v_h \in S^1_{h;*,0}$, where $Q^0_h: L^2(0,T) \rightarrow S^0_h(0, T)$ denotes the $L^2$ orthogonal projection on the piecewise constant finite element space $S^0_h(0, T )$. So basically, they do an under-integration in the $L^2(0,T)$ norm. That is, instead of integrating exactly the $L^2(0,T)$ scalar product $\langle u_h, v_h \rangle_{L^2(0,T)}$, they approximate it by projecting $v_h$ on the piecewise constant finite element space $S^0_h(0, T )$.

As a consequence of the properties of the projection operator $Q^0_h$ and of the piecewise linear nodal interpolation operator $I_h: C[0,T] \rightarrow S^1_h(0,T)$ (Lemma $17.1$, Lemma $17.2$, Lemma $17.3$, Lemma $17.4$, Lemma $17.5$ of \cite{Steinbach2019}), the new bilinear form \eqref{ah fem} satisfies the uniform (w.r.t. $h$) inf-sup condition
\begin{equation}\label{infsup perturbata}
\frac{1}{1+\sqrt{2}\mu T^2} |u_h|_{H^1(0,T)} \leq \sup_{0 \neq v_h \in S^1_{h;*,0}} \frac{a_h(u_h, v_h)}{|v_h|_{H^1(0,T)}},
\end{equation}
see Lemma $17.6$ of (\cite{Steinbach2019}). 

Therefore, O. Steinbach and M. Zank consider the following perturbed problem
\begin{equation}\label{perturbed problem}
\begin{cases}
\text{Find} \ u_h \in S^1_{h;0,*} \quad \text{such that}\\
a_h(u_h,v_h)=\langle f,v_h \rangle_{(0,T)} \quad  \forall v_h \in  S^1_{h;*,0},
\end{cases}
\end{equation}
where $T>0$ and $f \in [H^1_{*,0}(0,T)]' $ are given, and the notation $\langle \cdot,\cdot \rangle_{(0,T)}$ denotes the duality pairing as extension of the inner product in $L^2(0,T)$. Thanks to \eqref{infsup perturbata} estimate, problem \eqref{perturbed problem} is well-posed and the following stability condition holds
\begin{equation*}
|\tilde{u}_h|_{H^1(0,T)} \leq \big(1+\sqrt{2}\mu T^2\big) \|f\|_{[H^1_{*,0}(0,T)]'},
\end{equation*}
 where $\tilde{u}_h \in S^1_{h;0,*} $ is the unique solution related to the fixed right-hand side $f \in [H^1_{*,0}(0,T)]'$.

Using an alternative representation (Corollary 17.1 of \cite{Steinbach2019}) of the bilinear form $a(\cdot,\cdot)$ defined in \eqref{bil ode} and some standard techniques, such as Galerkin orthogonality and interpolation error estimates, also linear convergence in $|\cdot|_{H^1(0,T)}$ of the discrete solution $\tilde{u}_h$ of \eqref{perturbed problem} to the solution $u \in H^1_{0,*}(0,T)$ of \eqref{var ode} are proven, if $u \in H^2(0,T)$; see Theorem 17.1 of (\cite{Steinbach2019}) for more details. Furthermore, quadratic convergence in $\|\cdot\|_{L^2(0,T)}$ is shown, if the exact solution satisfies $u \in H^2(0,T)$; see Theorem 4.2.21 of (\cite{Zank2020}). \\ \bigskip

We fix the final time $T=10$, and we test their stabilization by numerically estimating the discrete inf-sup of the bilinear form \eqref{infsup perturbata} by means of Proposition \ref{estimate infsup} with uniform mesh.

\begin{figure}[h!]
	\centering
	\includegraphics[scale=0.45]{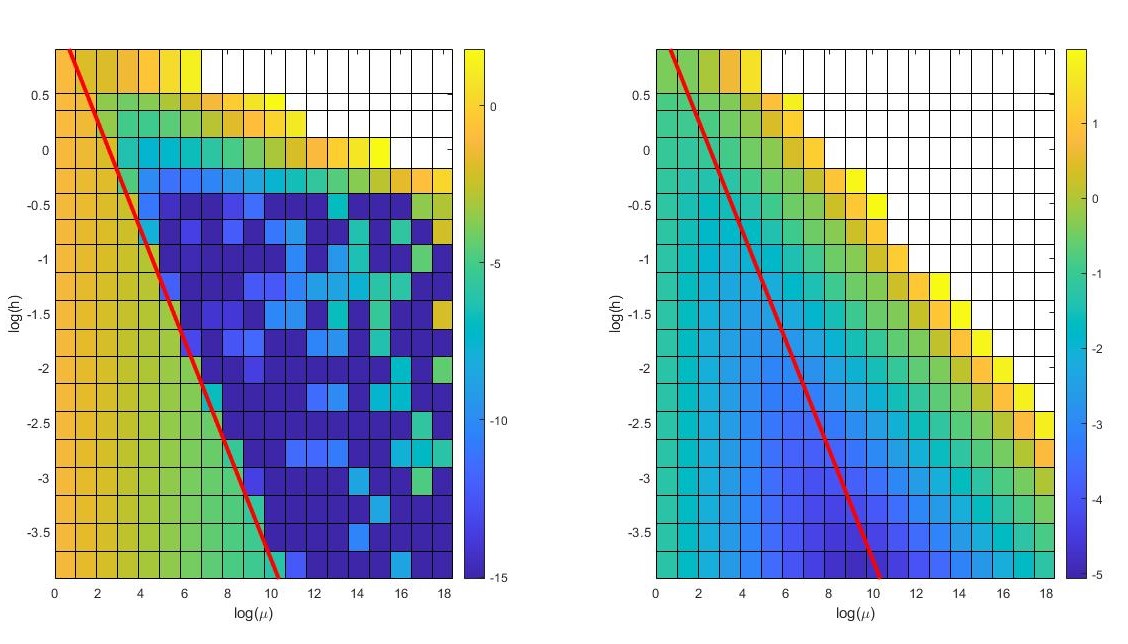}
	\caption{On the left a p-color plot of $\log(\beta)$ with respect to $(\log(\mu),log(h))$, where $\beta$ is the discrete inf-sup value of the bilinear form \eqref{bil ode}. On the right a p-color plot of $\log(\beta)$ with respect to $(\log(\mu),log(h))$, where $\beta$ is the inf-sup value of the bilinear form \eqref{ah fem}. In both Figures the red line is the natural logarithm of the upper bound \eqref{stab Linfty}.}
	\label{fig: beta stab fem}
\end{figure}

\begin{figure}[h!]
	\centering
	\includegraphics[scale=0.45]{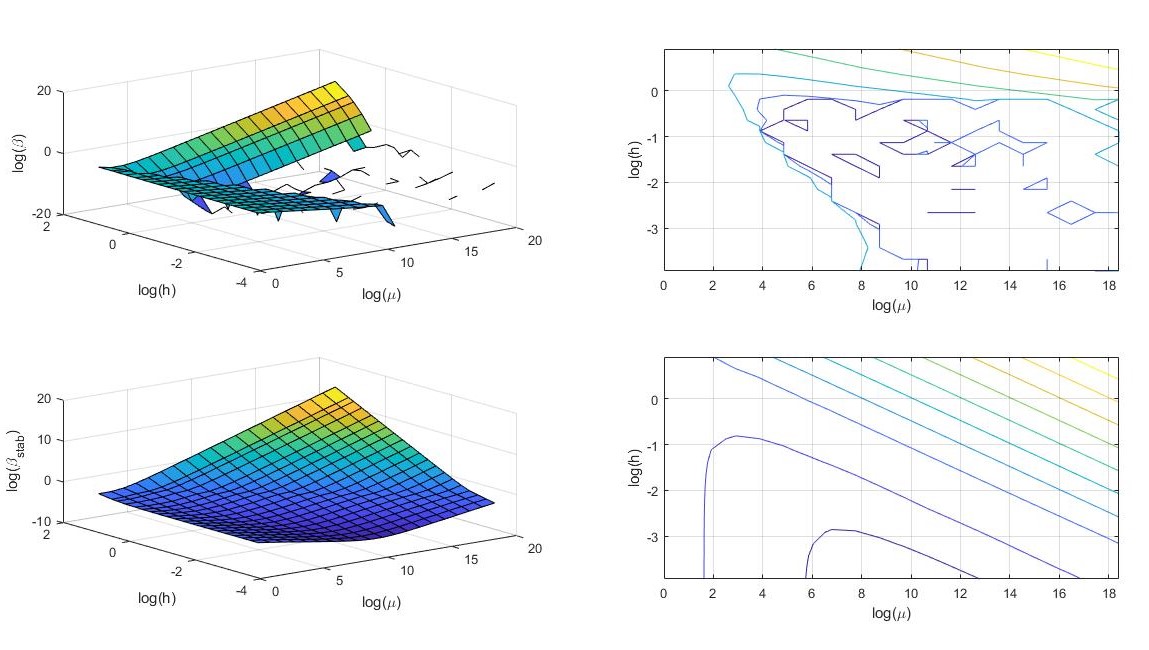}
	\caption{On the left, two three-dimensional plots of $\log(\beta)$ with respect to $(\log(\mu),\log(h))$. On the right, the contour lines of $\log(\beta)$ with respect to $(\log(\mu),\log(h))$. The top two figures correspond to the bilinear form defined in \eqref{bil ode} restricted to the FEM discrete spaces \eqref{fem trial}, \eqref{fem test}, and the bottom two to the stabilised bilinear form \eqref{ah fem}.}
	\label{fig: 3D plot beta stab fem}
\end{figure}

As one can see from Figures \ref{fig: beta stab fem}, \ref{fig: 3D plot beta stab fem}, the behaviour of the discrete inf-sup related to \eqref{ah fem} is uniformly (w.r.t. $h$) bounded. In particular, the red line of Figure \ref{fig: beta stab fem} corresponds to the upper-bound \eqref{stab Linfty} and now, in the stabilized case, it no longer separates two different discrete inf-sup regimes.

As we expect from Remark 4.2.20 of (\cite{Zank2020}), in Figure \ref{fig: beta stab fem} one can note that the optimal discrete inf-sup constant shows a dependency of order $\mu^{-\frac{1}{2}}$, like the discrete inf-sup constant of \eqref{bil ode} in the stability region given by \eqref{stab Linfty}, i.e., defined by
\begin{equation*}
h < \sqrt{\frac{12}{\mu}}.
\end{equation*}

As before, in Figure \ref{fig: beta stab fem} we do not consider the values of the discrete inf-sup that are greater than $10^2$. Indeed these values are not a physical phenomenon, but simply the result of an unsuitable discretization of the problem, as observed in Section \ref{sec infsup fem}. \\ \bigskip

We test O. Steinbach and M. Zank's stabilization also by studying the error committed in the approximation of an exact solution of \eqref{var ode}. As in previous Sections and in (\cite{Steinbach2019, Zank2020}), as a numerical example for the perturbed Galerkin piecewise continuous linear FEM \eqref{perturbed problem}, we consider a uniform discretization of the time interval $(0,T)$ with $T = 10$ and a mesh-size $h = T/N$. For $\mu = 1000$ we consider the strong solution $u(t) = \sin^2\Big(\frac{5}{4}\pi t\Big)$
and we compute the integrals appearing at the right-hand side using high-order integration rules.  

The minimum number of elements chosen is $N = 4$, the maximum number is $N = 32768$, as in Section \ref{sec errors fem}.

\begin{figure}[h!]	
	\centering
	\includegraphics[scale=0.5]{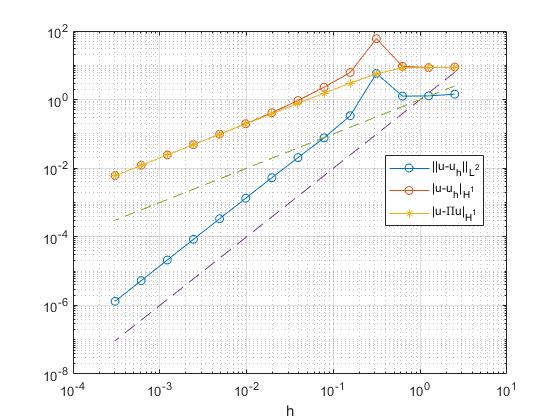}
	\caption{A $\log$-$\log$ plot of errors committed by perturbed piecewise continuous linear FEM \eqref{perturbed problem} in $|\cdot|_{H^1(0,T)}$ seminorm and in $\|\cdot\|_{L^2(0,T)}$ norm, with respect to a uniform mesh-size $h$. Also, the best approximation error in $|\cdot|_{H^1(0,T)}$ seminorm is represented. The square of the wave number is $\mu=1000$.}
	\label{fig: err fem stab}
\end{figure}

\begin{figure}[h!]	
	\centering
	\includegraphics[scale=0.5]{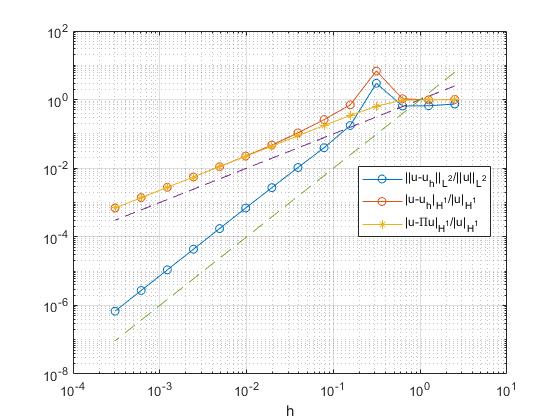}
	\caption{A $\log$-$\log$ plot of relative errors committed by perturbed piecewise continuous linear FEM \eqref{perturbed problem} in $|\cdot|_{H^1(0,T)}$ seminorm and in $\|\cdot\|_{L^2(0,T)}$ norm, with respect to a uniform mesh-size $h$. Also, the best approximation error in $|\cdot|_{H^1(0,T)}$ seminorm is represented. The square of the wave number is $\mu=1000$.}
	\label{fig: rel err fem stab}
\end{figure}

As one can see in Figure \ref{fig: err fem stab}, linear and quadratic convergence, respectively, in $|\cdot|_{H^1(0,T)}$ seminorm and in $\|\cdot\|_{L^2(0,T)}$ norm are confirmed.

\begin{figure}
	\hspace{-1cm}
	\begin{minipage}[h!]{8.5cm}
		\centering
		\includegraphics[width=6.5cm]{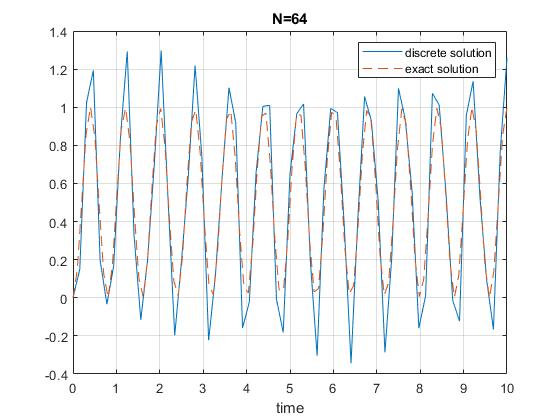}
		\caption{Exact and discrete solutions of the perturbed problem \eqref{perturbed problem} for $\mu = 1000$ and $N = 64$ elements (i.e., $h = 0.1563$).}
	\end{minipage} 
	\hspace{-1cm}
	\begin{minipage}[h!]{9cm}
		\centering
		\includegraphics[width=6.5cm]{Figures/uex_udis_ode_fem_N64}
		\caption{Exact and discrete solutions by unperturbed piecewise linear FEM  for $\mu = 1000$ and $N = 64$ elements (i.e., $h = 0.1563$).}
	\end{minipage}
\end{figure}

\subsubsection{Empirical reason for the unconditional stability of the perturbed FEM discretization}
The following representation holds (Lemma 17.2 of \cite{Steinbach2019})
\begin{equation}\label{rappre ah}
a_h(u_h,v_h)=-\langle \partial_t u_h, \partial_t v_h \rangle _{L^2(0,T)} - \frac{\mu}{12}\sum_{l=1}^{N} h_l^2 \langle \partial_t u_h, \partial_t v_h\rangle_{L^2(\tau_l)} + \mu \langle u_h, v_h \rangle_{L^2(0,T)},
\end{equation}
for all $u_h\in S^1_{h;0,*}$, $v_h \in S^1_{h;*,0}$, where $\tau_l$, for $l=1, \ldots,N$, are the subintervals of $(0,T)$ given by the Galerkin discretization. Note that \eqref{rappre ah} results in
\begin{equation}\label{rappre ah mesh unif}
a_h(u_h,v_h)=-\Bigg(1+\frac{\mu h^2}{12}\Bigg)\langle \partial_t u_h, \partial_t v_h \rangle _{L^2(0,T)} + \mu \langle u_h, v_h \rangle_{L^2(0,T)},
\end{equation}
for all $u_h\in S^1_{h;0,*}$, $v_h \in S^1_{h;*,0}$, in the case of a uniform mesh-size. 

The representation \eqref{rappre ah mesh unif} suggests another justification to the inf-sup stability of the perturbed bilinear form $a_h(\cdot,\cdot)$ in the case of a uniform mesh-size. Indeed, the bilinear form \eqref{rappre ah mesh unif} corresponds to the perturbed problem
\begin{equation*}
\partial_{tt}u(t)+ \mu_h u(t)=f(t), \quad \text{for} \ t \in (0,T), \quad u(0)=\partial_{t}u(t)_{|t=0}=0,
\end{equation*}
where 
\begin{equation*}
\mu_h := \frac{\mu}{1+\frac{\mu h^2}{12}}.
\end{equation*}
Since the mesh-size $h$ always satisfies
\begin{equation*}
h < \sqrt{\frac{12}{\mu_h}} < \sqrt{\frac{12}{\mu}+h^2},
\end{equation*}
the perturbed bilinear form \eqref{ah fem} with a uniform mesh-size is inf-sup stable with an inf-sup value that has a dependency on $\mu_h$ of order $\mu_h^{-\frac{1}{2}}$, as a consequence of the numerical results of Section \ref{sec infsup fem}. In particular, 
\begin{equation*}
 {\mu}^{-\frac{1}{2}} \leq {\mu_h}^{-\frac{1}{2}} \simeq \inf_{0 \neq u_h \in S^1_{h;0,*}} \sup_{0 \neq v_h \in S^1_{h;*,0}} \frac{a_h(u_h, v_h)}{|v_h|_{H^1(0,T)}},
\end{equation*}
for all $u_h\in S^1_{h;0,*}$, $v_h \in S^1_{h;*,0}$. Note that these arguments are not intended to replace the analysis of (\cite{Steinbach2019, Zank2020}), which is more complete. They are meant to give an empirical motivation that explains why the stabilization proposed by O. Steinbach and M. Zank actually works.

\subsection{Stabilization for quadratic isogeometric discretization with maximal regularity}
In this Section we consider three different stabilization techniques. The first two give poor results, whereas the performances of the last one are satisfactory.

Inspired by \eqref{ah fem}, we firstly try to perturb the quadratic isogeometric discretization with maximal regularity in the same way, i.e., defining
\begin{equation}\label{tent 1 iga}
a_h(u_h,v_h):=-\langle \partial_t u_h, \partial_t v_h \rangle _{L^2(0,T)} + \mu \langle u_h, Q_h^0 v_h \rangle_{L^2(0,T)},
\end{equation}
for all $u_h \in V^h_{0,*}$ of \eqref{iga space} and $v_h \in V^h_{*,0}$ of \eqref{iga test space}, and considering the modified discrete problem
\begin{equation}\label{perturbed iga}
\begin{cases}
\text{Find} \ u_h \in V^h_{0,*} \quad \text{such that}\\
a_h(u_h,v_h)=\langle f,v_h \rangle_{(0,T)} \quad  \forall v_h \in  V^h_{*,0}.
\end{cases}
\end{equation}
\begin{oss}
	Another possible perturbation consists of defining
	\begin{equation*}
	a_h(u_h,v_h):=-\langle \partial_t u_h, \partial_t v_h \rangle _{L^2(0,T)} + \mu \langle u_h, Q_h^1 v_h \rangle_{L^2(0,T)},
	\end{equation*}
	for all $u_h \in V^h_{0,*}$ and $v_h \in V^h_{*,0}$, where $Q^1_h: L^2(0,T) \rightarrow S^1_h(0,T)$ denotes the $L^2(0,T)$ orthogonal projection on the piecewise continuous linear finite element space $S^1_h(0,T)$.
\end{oss}

As in previous Sections, we test perturbation \eqref{tent 1 iga} by studying the error committed in the approximation of an exact solution of \eqref{var ode}. In particular, as before, as a numerical example we consider a uniform discretization of the time interval $(0,T)$ with $T = 10$ and a mesh-size $h = T/N$. For $\mu = 1000$ we choose the strong solution $u(t) = \sin^2\Big(\frac{5}{4}\pi t\Big)$
and we compute the integrals appearing at the right-hand side using high-order integration rules. The minimum number of elements chosen is $N = 4$, the maximum number is $N=4096$, as in Section \ref{sec errs iga}.

The results that we obtain are in Figure \ref{fig: err iga pert Qh0}, which are far from promising, since the errors are even larger than those ones made by the conditionally stable method \eqref{iga ode}. A possible explanation is that we are exceeding in the under-integration if compared to the regularity of the test and trial functions we consider.

\begin{figure}[h!]	
	\centering
	\includegraphics[scale=0.5]{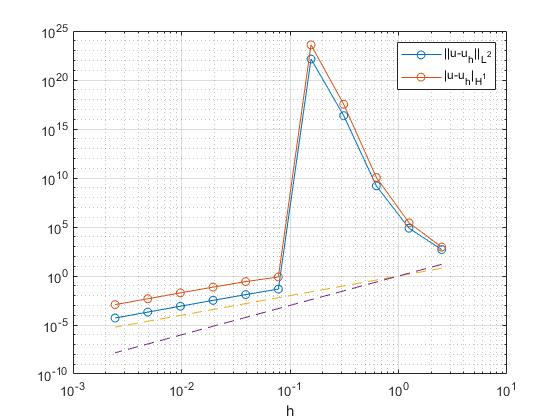}
	\caption{A $\log$-$\log$ plot of errors committed by perturbed quadratic IGA with maximal regularity \eqref{tent 1 iga} in $|\cdot|_{H^1(0,T)}$ seminorm and in $\|\cdot\|_{L^2(0,T)}$ norm, with respect to a uniform mesh-size $h$. The square of the wave number is $\mu=1000$.}
	\label{fig: err iga pert Qh0}
\end{figure}

We then decide to change perspective by considering the perturbation \eqref{rappre ah mesh unif}, inspired by the observations at the end of Section \ref{stab fem}. Since we numerically obtain the stability constraint \eqref{emp constraint}, we define the following perturbation
\begin{equation}\label{tent 2 iga}
a_h(u_h,v_h)=-\Bigg(1+\frac{\mu h^2}{9}\Bigg)\langle \partial_t u_h, \partial_t v_h \rangle _{L^2(0,T)} + \mu \langle u_h, v_h \rangle_{L^2(0,T)},
\end{equation}
for all $u_h \in V^h_{0,*}$ and $v_h \in V^h_{*,0}$, if the mesh-size is uniform. We then test this perturbation by studying the error committed by the discrete solution of \eqref{perturbed iga} considering the perturbed bilinear form \eqref{tent 2 iga}. We choose the same exact solution and assumptions of Figure \ref{fig: err iga pert Qh0}. The results that we achieve are in Figure \ref{fig: err iga pert tent 2} and are quite promising, as we could expect from the empirical analysis at the end of Section \ref{stab fem}. Indeed, bounded errors are a consequence of stability. However, we do not obtain the orders of convergence that we expect for quadratic IGA of maximal regularity, since we only get quadratic convergence in $\|\cdot\|_{L^2(0,T)}$ norm.

\begin{figure}[h!]	
	\centering
	\includegraphics[scale=0.5]{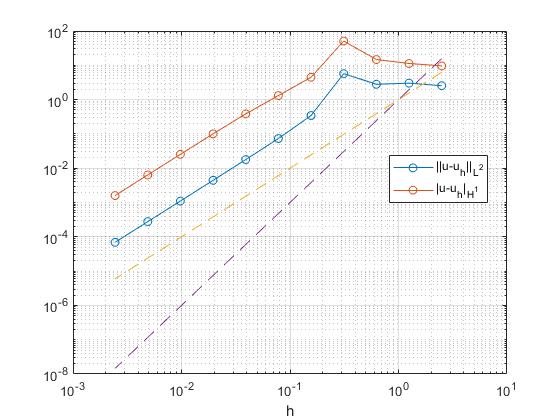}
	\caption{A $\log$-$\log$ plot of errors committed by perturbed quadratic IGA with maximal regularity \eqref{tent 2 iga} in $|\cdot|_{H^1(0,T)}$ seminorm and in $\|\cdot\|_{L^2(0,T)}$ norm, with respect to a uniform mesh-size $h$. The square of the wave number is $\mu=1000$.}
	\label{fig: err iga pert tent 2}
\end{figure}

\bigskip
In order to gain a method that is unconditionally stable whose orders of convergence in $|\cdot|_{H^1(0,T)}$ and $\|\cdot\|_{L^2(0,T)}$ are what we expect for quadratic splines with $C^1(0,T)$ global regularity, we decide to define the following perturbed bilinear form in the case of a uniform mesh-size 
\begin{equation}\label{tent 3 iga}
\begin{split}
a_h(u_h,v_h)=-\langle \partial_t u_h, \partial_t v_h \rangle _{L^2(0,T)} + \mu \langle u_h, &v_h \rangle_{L^2(0,T)} \\
&- \delta \mu h^4\langle \partial_t^2 u_h, \partial_t^2 v_h \rangle _{L^2(0,T)},
\end{split}
\end{equation}
for all $u_h \in V^h_{0,*}$ and $v_h \in V^h_{*,0}$, where $\delta >0$ is a fixed real value. 

Let us note that both the bilinear forms \eqref{tent 2 iga}, \eqref{tent 3 iga} are defined by a non consistent perturbation, since, in general, if $\tilde{u}_h \in V^h_{0,*}$ is a solution of their related perturbed problem \eqref{perturbed iga}, $\tilde{u}_h $ is not a solution of the non perturbed problem \eqref{iga ode}.

The behaviour of \eqref{tent 3 iga} depends on the choice of the coefficient $\delta$. Therefore, we test this perturbation by studying the error committed by the discrete solution of \eqref{perturbed iga} considering the perturbed bilinear form \eqref{tent 3 iga} with different choice of $\delta$. In particular, we consider $\delta \in \{\frac{1}{100},\frac{1}{10},1,10,100\}$ in order to see how the errors behave for different orders of magnitude of the perturbation coefficient. We choose the same exact solution and assumptions of Figures \ref{fig: err iga pert Qh0}, \ref{fig: err iga pert tent 2}. 

Figure \ref{fig: err and rel err iga delta} represents $\log$-$\log$ plots of relative errors committed by perturbed quadratic IGA with maximal regularity \eqref{tent 3 iga} in $|\cdot|_{H^1(0,T)}$ seminorm and in $\|\cdot\|_{L^2(0,T)}$ norm, with respect to a uniform mesh-size $h$. The smallest global errors are obtained for $\delta = \frac{1}{100}$. In particular, they are satisfactory since the errors in $|\cdot|_{H^1(0,T)}$ and $\|\cdot\|_{L^2(0,T)}$ have a maximum that is smaller than $10^2$, but also because we achieve quadratic convergence in $|\cdot|_{H^1(0,T)}$ and cubic convergence in $\|\cdot\|_{L^2(0,T)}$, as we expect from quadratic IGA discretization.

\begin{figure}[h!]	
	\centering
	\includegraphics[scale=0.4]{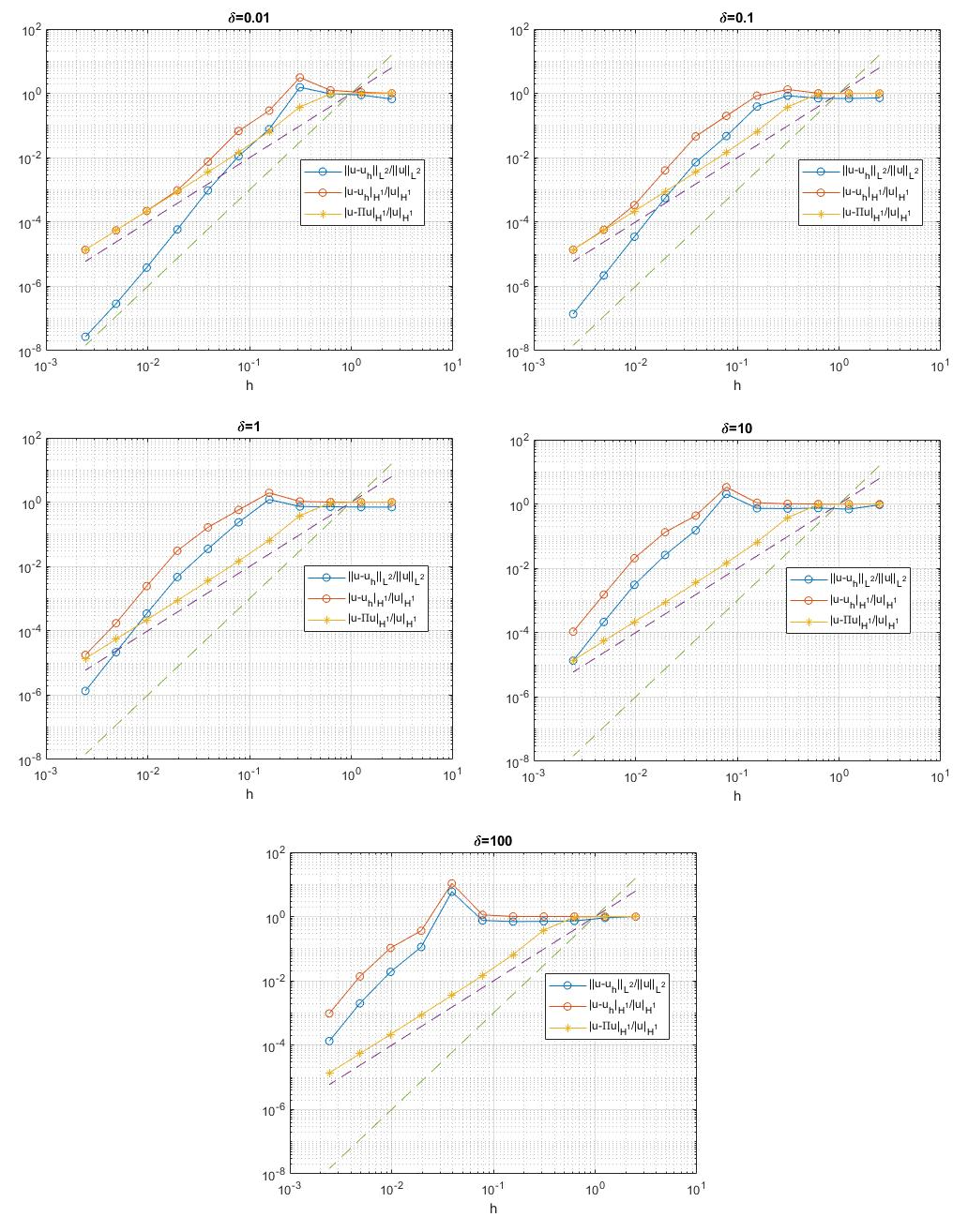}
	\caption{$\log$-$\log$ plots of relative errors committed by perturbed quadratic IGA with maximal regularity \eqref{tent 3 iga} in $|\cdot|_{H^1(0,T)}$ seminorm and in $\|\cdot\|_{L^2(0,T)}$ norm, with respect to a uniform mesh-size $h$. Also, the best approximation error in $|\cdot|_{H^1(0,T)}$ seminorm is represented. The square of the wave number is $\mu=1000$.}
	\label{fig: err and rel err iga delta}
\end{figure}

\begin{figure}
	\hspace{-1.5cm}
	\begin{minipage}[h!]{9cm}
		\centering
		\includegraphics[width=6.5cm]{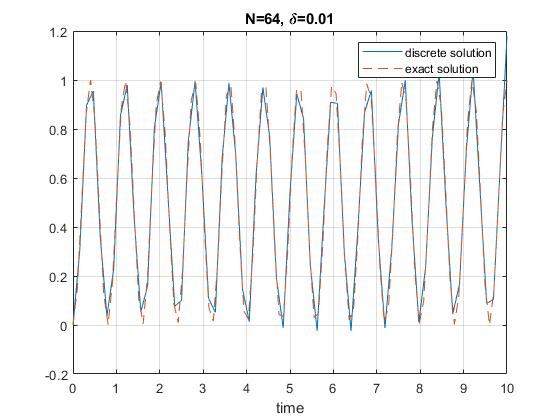}
		\caption{Exact and discrete solutions of the perturbed problem \eqref{perturbed iga} with \eqref{tent 3 iga} and $\delta=\frac{1}{100}$, for $\mu = 1000$ and $N = 64$ elements (i.e., h = 0.1563).}
	\end{minipage} 
	\hspace{-1.5cm}
	\begin{minipage}[h!]{8.5cm}
		\centering
		\includegraphics[width=6.5cm]{Figures/uex_udis_ode_iga_N64}
		\caption{Exact and discrete solutions by unperturbed quadratic IGA for $\mu = 1000$ and $N = 64$ elements (i.e., h = 0.1563).}
	\end{minipage}
\end{figure}

Since the new bilinear form \eqref{tent 3 iga} with $\delta = \frac{1}{100}$ returns significantly small errors that converge with ``the right orders'' of convergence, we are now interested in numerically studying how its inf-sup value behaves with respect to $(\mu,h)$. As before, we fix the final time T=10 and a uniform mesh and we numerically estimate the discrete inf-sup of the bilinear form \eqref{tent 3 iga} by means of Proposition \ref{estimate infsup}.

\begin{figure}[h!]
	\centering
	\includegraphics[scale=0.45]{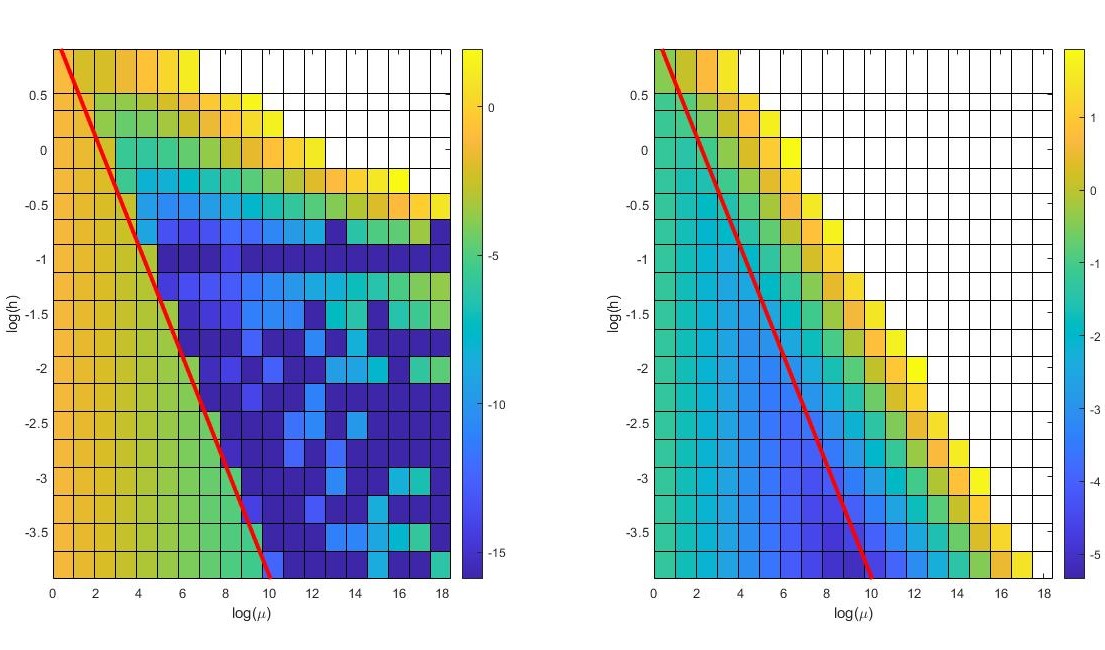}
	\caption{On the left a p-color plot of $\log(\beta)$ with respect to $(\log(\mu),log(h))$, where $\beta$ is the discrete inf-sup value of the bilinear form \eqref{bil ode}. On the right a p-color plot of $\log(\beta)$ with respect to $(\log(\mu),log(h))$, where $\beta$ is the inf-sup value of the bilinear form \eqref{tent 3 iga} with $\delta=\frac{1}{100}$. In both Figures the red line is the natural logarithm of the upper bound \eqref{emp constraint} and the square of the wave number is $\mu=1000$ .}
	\label{fig: beta stab iga}
\end{figure}

\begin{figure}[h!]
	\centering
	\includegraphics[scale=0.45]{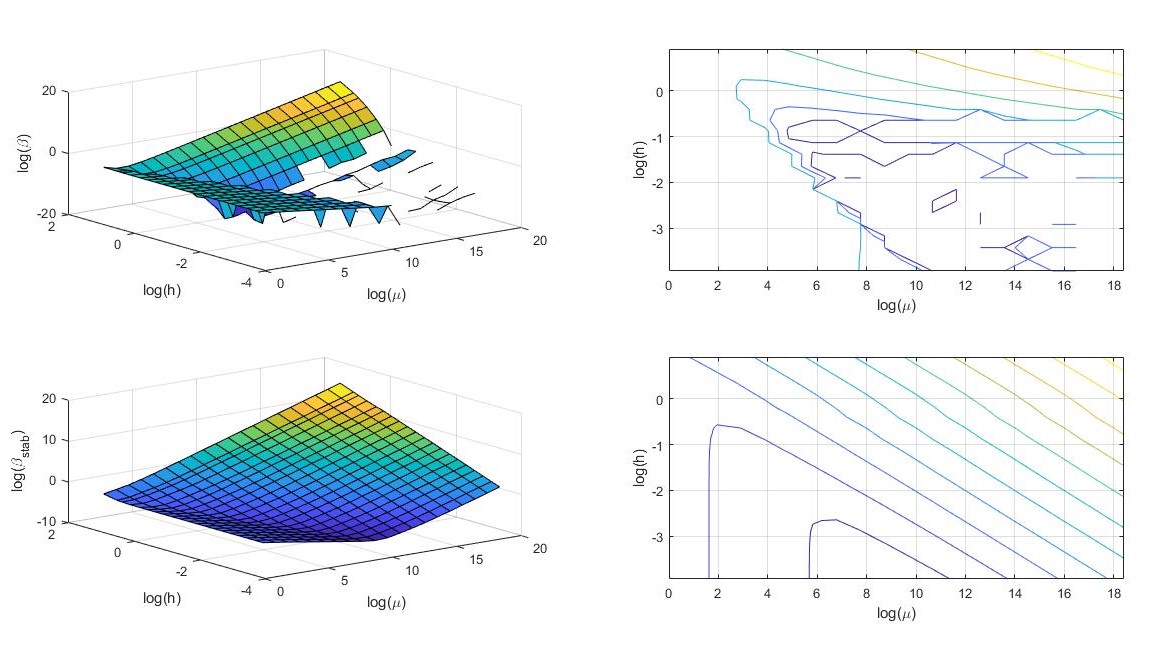}
	\caption{On the left, two three-dimensional plots of $\log(\beta)$ with respect to $(\log(\mu),\log(h))$. On the right, the contour lines of $\log(\beta)$ with respect to $(\log(\mu),\log(h))$. The top two figures correspond to the bilinear form defined in \eqref{bil ode} restricted to the discrete IGA spaces \eqref{iga space}, \eqref{iga test space}, and the bottom two to the stabilised bilinear form \eqref{tent 3 iga} with $\delta=\frac{1}{100}$. The square of the wave number is $\mu=1000$.}
	\label{fig: 3D plot beta stab iga}
\end{figure}

As one can see from Figures \ref{fig: beta stab iga}, \ref{fig: 3D plot beta stab iga}, the behaviour of the discrete inf-sup of \eqref{tent 3 iga} is uniformly (w.r.t. $h$) bounded. In particular, the red line of Figure \ref{fig: beta stab iga} corresponds to the upper-bound \eqref{emp constraint} and now, in the stabilized case, it no longer separates two different discrete inf-sup regimes. 

\begin{figure}[h!]	
	\centering
	\includegraphics[scale=0.4]{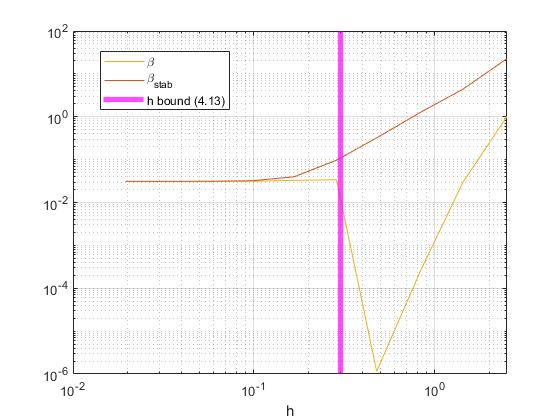}
	\caption{A $\log$-$\log$ plot of the isogeometric inf-sup values of \eqref{bil ode}, \eqref{tent 3 iga} with respect to a uniform mesh-size $h$, for $\delta = \frac{1}{100}$ and a fixed $\mu = 100$.}
	\label{fig: sez infsup iga depend on h}
\end{figure}

Figure \ref{fig: sez infsup iga depend on h} clearly shows the uniformly bounded behaviour (w.r.t. $h$) of the numerically estimated inf-sup of the perturbed bilinear form \eqref{tent 3 iga}.

In Figure \ref{fig: beta stab iga} one can note that the optimal discrete inf-sup constant shows a dependency on $\mu$ of order $\mu^{-\frac{1}{2}}$, as the discrete inf-sup constant of \eqref{bil ode} in the stability region given by \eqref{emp constraint}, i.e., defined by
\begin{equation*}
h < \sqrt{\frac{9}{\mu}}.
\end{equation*}

As before, in Figure \ref{fig: beta stab iga} we do not consider the values of discrete inf-sup that are greater than $10^2$, since they are not a physical phenomenon, but simply the result of an unsuitable discretization of the problem, as observed in Section \ref{sec infsup fem}. 

It might be interesting to further improve the stabilization by choosing the optimal $\delta$ in relation to the empirical stability constraint \eqref{emp constraint}, but the perturbation of order four that we are operating does not allow us to apply the reasoning explained at the end of Section \ref{stab fem}. However, an improvement of $\delta$ could be obtained by writing it as an appropriate function of $\mu$.

%% file: Chapters/Chapter5.tex

\chapter{Conclusions} 

\label{Chapter5} 




The goal of this thesis is to investigate the first steps towards an unconditionally stable space-time isogeometric method with maximal regularity, using a tensor-product approach, for the wave problem \eqref{eqonde}. 

Inspired by the work (\cite{Steinbach2019}), our starting point is studying the conditioned stability of the conforming quadratic IGA discretization with global $C^1(0,T)$ regularity of the related ordinary differential problem \eqref{eq ode}. In Chapter \ref{Chapter3} we hence obtain two explicit upper bounds on the mesh-size, which, if respected, guarantee stability. The first one \eqref{h bound}, i.e., 
\begin{equation*}
h \leq \frac{\pi^2}{\sqrt{2}(2+\sqrt{\mu}T)\mu T},
\end{equation*}
is an extension to quadratic IGA with maximal regularity of Theorem 4.7 of (\cite{Coercive}), which is a result for the piecewise continuous linear FEM discretization. In particular, our upper bound is about twice as large as the FEM one, and the discrete stability constant of \eqref{infsup zank}, i.e.,
\begin{equation*}
\beta_1(\mu,T):=\frac{2 \pi^2}{(2+\sqrt{\mu}T)^2(\pi^2+4\mu T^2)},
\end{equation*}
depends on the coefficient $\mu >0$ and on the final time $T>0$ of the ODE \eqref{eq ode} with the same order as the FEM one. The second upper bound \eqref{h bound gard}, i.e.,
\begin{equation*}
h \leq \frac{\pi^5}{(\pi^2+ 4\mu T^2)[\pi^2+2\mu T^2 (2+\sqrt{\mu} T)]} \sqrt{\frac{2b - \mu T^2}{2b(2+b)\mu}}
\end{equation*}
where $b > \frac{\mu T^2}{2}$ is an arbitrary fixed real value, is obtained by the theory of \textit{Galerkin method applied to G\aa rding-type problems}. 

The asymptotic case, i.e., $\mu$ that is significantly large, seems to us the most interesting situation for the problem of instability (Remark \ref{pollut effect}) and for the wave equation. Thus, in Remark \ref{oss asympt} we compare the obtained bounds and the corresponding stability constants for $\mu \rightarrow \infty$. It follows that, for ``a very large'' $\mu$, the first bound \eqref{h bound} is a weaker constraint than the second bound \eqref{h bound gard}. Moreover, some numerical results of Chapter \ref{Chapter4} show that the latter, with the optimal choice for $b$ \eqref{b}, can be a stronger constraint than the former even if $\mu$ is not extremely large.

Let us define 
\begin{equation*}
C_1(\mu,T):=\frac{1}{\beta_1(\mu,T)},
\end{equation*}
where $\beta_1(\mu,T)$ is the discrete inf-sup constant of \eqref{infsup zank} that we recall above. In Remark \ref{oss asympt} we note that the stability constant of \eqref{stab iga gard}, i.e.,
\begin{equation*}
C_2(\mu,T):=\Bigg[ \frac{3b + \mu T^2\big(\frac{8b}{\pi^2}-\frac{1}{2} \big)}{2b-\mu T^2} \Bigg] \quad \Bigg( \text{for a fixed} \ b > \frac{\mu T^2}{2} \Bigg),
\end{equation*}
that arises from the second upper bound \eqref{h bound gard} has a slower growth than $C_1(\mu,T)$ for $\mu \rightarrow \infty$. Thus, the theory of \textit{Galerkin method applied to G\aa rding-type problems} is useful for our problem leading to a lower bound of the discrete inf-sup that, asymptotically, is sharper than that obtained by extending the analysis of O. Steinbach and M. Zank (\cite{Coercive}). In Chapter \ref{Chapter4} there are also some numerical results showing that the inf-sup $\beta_2(\mu,T):=\frac{1}{C_2(\mu,T)}$ can be sharper than $\beta_1(\mu,T)$ even if $\mu$ is not significantly large.

In Chapter \ref{Chapter4} we observe that the two upper bounds \eqref{h bound}, \eqref{h bound gard} are not optimal. However, if the mesh-size is uniform, we manage to numerically find a stability constraint \eqref{emp constraint}, i.e., 
\begin{equation*}
h <  \sqrt{\frac{9}{\mu}},
\end{equation*}
which, from the numerical results that we obtain, seems to be sharp. The upper bound \eqref{emp constraint} is of the same order (w.r.t $\mu$) of the sharp upper bound \eqref{stab Linfty} for the stability of FEM discretization, i.e., 
\begin{equation*}
h < \sqrt{\frac{12}{\mu}}.
\end{equation*}
 
Quadratic IGA discretization of maximal regularity appears to be advantageous over piecewise continuous linear FEM. Indeed, the stability upper bounds on the mesh size of the former are very similar to the FEM ones, and the orders of convergence of the IGA method are of one order more than the FEM ones. Moreover, from the numerical tests we see that the error committed by the IGA is, for each mesh-size $h$, strictly smaller than the FEM one.\\ \vspace{0.8cm}

As observed in Remark \ref{oss grado più alto}, if we raise the degree of the splines to $p>2$ while keeping maximal regularity, the constraint on the mesh-size \eqref{h bound gard} and the resulting stability constant do not change. On the other hand, how the upper bound \eqref{h bound} behaves in $h$ is not immediately clear from the proof of Theorem \ref{teo stab IGA zank}. However, we are sure that the orders of convergence of the discrete solution to the exact solution will be $p$ in $|\cdot|_{H^1(0,T)}$ and $p+1$ in $\|\cdot\|_{L^2(0,T)}$, if the exact solution is sufficiently regular. Also, we could expect that the maximal error, and more generally for all values of $h$, is strictly smaller than the error committed by continuous, linear FEM and quadratic $C^1(0,T)$ IGA discretizations. This behaviour of the error would be a consequence of high degree of approximation of B-spline technology (\cite{HUGHES20084104, n-width}) and of the ``good behaviour'' of high-order methods with respect to wave propagation problems (\cite{Babuska2000}). These error considerations are indeed confirmed by our numerical tests in Figures \ref{fig: err grad 3}, \ref{fig: rel err grad 3}, \ref{fig: err grad 4}, \ref{fig: rel err grad 4}.

As in Chapter \ref{Chapter4}, as a numerical example for the Galerkin-Petrov finite element methods  we consider a uniform discretization of the time interval $(0,T)$ with $T = 10$ and a mesh-size $h = T/N$. For $\mu = 1000$ we consider the strong solution $u(t) = \sin^2\Big(\frac{5}{4}\pi t\Big)$
and we compute the integrals appearing at the right-hand side using high-order integration rules.

\begin{figure}
	\hspace{-1.3cm}
	\begin{minipage}[h!]{9cm}
		\centering
		\includegraphics[width=6.5cm]{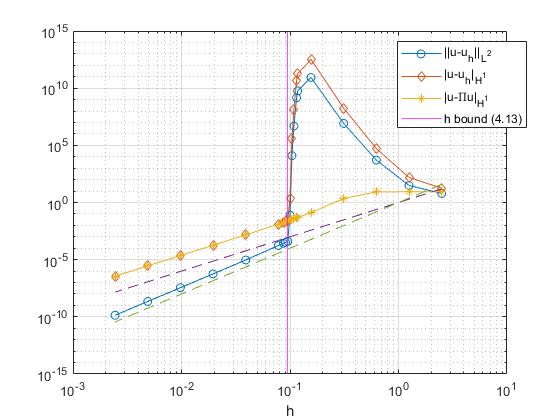}
		\caption{A $\log$-$\log$ plot of errors committed by cubic IGA, with maximal regularity, in \\$|\cdot|_{H^1(0,T)}$ seminorm and in \\$\|\cdot\|_{L^2(0,T)}$ norm, with respect to a uniform mesh-size $h$. Also, the best approximation error in\\ $|\cdot|_{H^1(0,T)}$ seminorm and the bound \eqref{emp constraint} are represented. The square of the wave number is $\mu=1000$}
		\label{fig: err grad 3}
	\end{minipage} 
	\hspace{-1.3cm}
	\begin{minipage}[h!]{9cm}
		\centering
		\includegraphics[width=6.5cm]{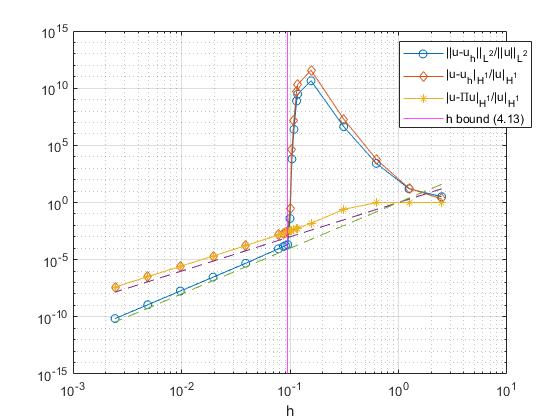}
		\caption{A $\log$-$\log$ plot of relative errors committed by cubic IGA, with maximal regularity, in $|\cdot|_{H^1(0,T)}$ seminorm and in \\$\|\cdot\|_{L^2(0,T)}$ norm, with respect to a uniform mesh-size $h$. Also, the best approximation error in\\ $|\cdot|_{H^1(0,T)}$ seminorm and the bound \eqref{emp constraint} are represented. The square of the wave number is $\mu=1000$}
		\label{fig: rel err grad 3}
	\end{minipage}
\end{figure}

\begin{figure}
	\hspace{-1.3cm}
	\begin{minipage}[h!]{9cm}
		\centering
		\includegraphics[width=6.5cm]{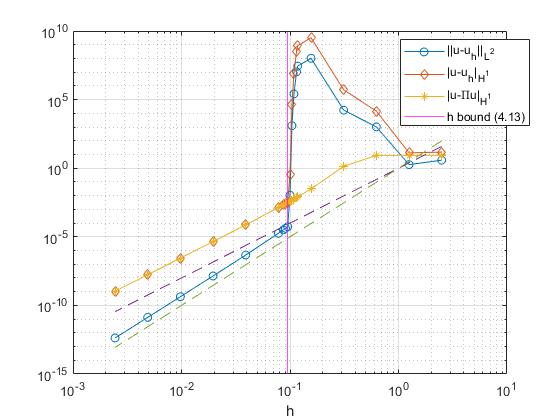}
		\caption{A $\log$-$\log$ plot of errors committed by IGA of fourth degree, with maximal regularity, in $|\cdot|_{H^1(0,T)}$ seminorm and in $\|\cdot\|_{L^2(0,T)}$ norm, with respect to a uniform mesh-size $h$. Also, the best approximation error in $|\cdot|_{H^1(0,T)}$ seminorm and the bound \eqref{emp constraint} are represented. The square of the wave number is $\mu=1000$}
		\label{fig: err grad 4}
	\end{minipage} 
	\hspace{-1.3cm}
	\begin{minipage}[h!]{9cm}
		\centering
		\includegraphics[width=6.5cm]{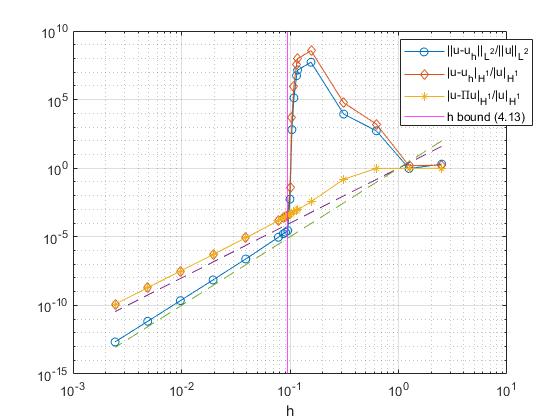}
		\caption{A $\log$-$\log$ plot of relative errors committed by IGA of fourth degree, with maximal regularity, in $|\cdot|_{H^1(0,T)}$ seminorm and in $\|\cdot\|_{L^2(0,T)}$ norm, with respect to a uniform mesh-size $h$. Also, the best approximation error in $|\cdot|_{H^1(0,T)}$ seminorm and the bound \eqref{emp constraint} are represented. The square of the wave number is $\mu=1000$}
		\label{fig: rel err grad 4}
	\end{minipage}
\end{figure}

As one can note in Figures \ref{fig: err grad 3}, \ref{fig: rel err grad 3}, \ref{fig: err grad 4}, \ref{fig: rel err grad 4} the constraint \eqref{emp constraint} seems to remain optimal with respect to the error by raising the degree and regularity of the splines.

Although the errors diminish by raising the degree and regularity of the splines, it is of significant importance to find a method that is stable for every degree and regularity, so that we are not forced to work with dense matrices, which cause a high computational cost. 

Finally, let us note that we expect the IGA discretization of degree $p \geq 2$ and regularity $C^{p-1}(0,T)$ to perform better than the piecewise continuous FEM of degree $p$. Indeed, although the IGA matrices have more non-zero entries, the FEM discretization uses more basis functions. Furthermore, we expect that, while approximating solutions of wave propagation problems, the error committed by the IGA method is smaller than the FEM one (\cite{Babuska2000}).\\ \bigskip

Our proposals of stabilizations are all based on non-consistent perturbations.

In order to stabilize the quadratic IGA with maximal regularity, if the mesh-size is uniform, we propose to consider the perturbed bilinear form \eqref{tent 3 iga}, i.e.,
\begin{equation*}
\begin{split}
a_h(u_h,v_h)=-\langle \partial_t u_h, \partial_t v_h \rangle _{L^2(0,T)} + \mu \langle u_h, &v_h \rangle_{L^2(0,T)} \\
&- \delta \mu h^4\langle \partial_t^2 u_h, \partial_t^2 v_h \rangle _{L^2(0,T)},
\end{split}
\end{equation*}
for all $u_h \in V^h_{0,*}$ and $v_h \in V^h_{*,0}$, where $\delta >0$ is a fixed real value. Our numerical results are very promising in the case of $\delta=\frac{1}{100}$. It would therefore be interesting to study the theory that could explain why this stabilization works and then propose an optimal $\delta$. 

For IGA method with generic polynomial degree $p$ and maximal regularity, we propose to consider the following perturbed bilinear form
\begin{equation}\label{ah grado magg iga}
\begin{split}
a_h(u_h,v_h)=-\langle \partial_t u_h, \partial_t v_h \rangle _{L^2(0,T)} + \mu \langle u_h, &v_h \rangle_{L^2(0,T)} \\
&- \delta \mu h^{2p}\langle \partial_t^p u_h, \partial_t^p v_h \rangle _{L^2(0,T)},
\end{split}
\end{equation}
for all $u_h$ and $v_h$, respectively, in the discrete trial and test spaces, where $\delta >0$ is a fixed real value.

If the mesh-size is non-uniform, we suggest considering
\begin{equation}\label{ah iga mesh generica}
\begin{split}
a_h(u_h,v_h)=-\langle \partial_t u_h, \partial_t v_h \rangle _{L^2(0,T)} + \mu \langle u_h, &v_h \rangle_{L^2(0,T)} \\
&- \delta\mu\sum_{l=1}^{N} h_l^{2p} \langle \partial_t^p u_h, \partial_t^p v_h\rangle_{L^2(\tau_l)},
\end{split}
\end{equation}
for all $u_h$ and $v_h$, respectively, in the discrete trial and test spaces, where $\tau_l$, for $l=1, \ldots,N$, are the subintervals of $(0,T)$ given by the Galerkin discretization and $\delta >0$ is a fixed real value. Actually, \eqref{ah grado magg iga} is a subcase of \eqref{ah iga mesh generica}. 

Finally, let us briefly consider the homogeneous Dirichlet problem for the second-order wave equation \eqref{eqonde}, i.e., 
\begin{equation}
\begin{cases}
\partial_{tt}u(x,t)-\Delta_xu(x,t)=g(x,t) \quad (x,t) \in Q:=\Omega \times (0,T)\\
u(x,t)=0 \quad (x,t) \in \partial{\Omega} \times [0,T]\\
u(x,0)=\partial_tu(x,t)_{|t=0}=0 \quad x \in \Omega,
\end{cases}
\end{equation}
where $\Omega \subset \mathbb{R}^d$, with $d=1,2,3$, is an open bounded Lipschitz domain and, for a real value $T>0$, $(0,T)$ is a time interval. In (\cite{Coercive}) the authors introduce a space-time variational formulation of \eqref{eqonde}, where integration by parts is also applied with respect to the time variable, and the classic anisotropic Sobolev spaces with homogeneous initial and boundary conditions are employed. Inspired by (\cite{Steinbach2019, Zank2020}), a possible unconditionally stable space-time IGA method with maximal regularity based on a tensor-product approach could be obtained by considering the perturbed bilinear form
\begin{equation}\label{onde}
\begin{split}
a_h(u_h,v_h)=-\langle \partial_t u_h, \partial_t v_h \rangle _{L^2(Q)} + &\langle \nabla_x u_h, \nabla_x v_h \rangle_{L^2(Q)} \\
&- \delta\sum_{m=1}^{d} \sum_{l=1}^{N_t} h_l^{2p} \langle \partial_t^p \partial_{x_m} u_h, \partial_t^p \partial_{x_m} v_h\rangle_{L^2(\Omega \times \tau_l)},
\end{split}
\end{equation}
for all $u_h$ and $v_h$, respectively, in the discrete trial and test spaces, where $\tau_l$, for $l=1, \ldots,Nt$, are the subintervals of $(0,T)$ given by the Galerkin discretization and $\delta >0$ is a fixed real value. We expect this stabilization to perform well, given the appreciable numerical results of the perturbation \eqref{tent 3 iga}. However, we have not tested \eqref{onde} yet, since we prefer to give priority to a full analysis of the IGA discretization of our model problem \eqref{var ode}, which, given its link to the wave equation, we expect to be significantly useful.